\documentclass[11pt,twoside,a4paper]{article}
\usepackage{makeidx}
\usepackage{graphicx}
\usepackage{amscd}
\usepackage{amsmath}
\usepackage{amssymb}
\usepackage{latexsym}
\usepackage[latin1]{inputenc}
\usepackage[T1]{fontenc}
\usepackage[francais]{babel}
\usepackage[T1]{fontenc}
\newcommand{\tun}{\begin{picture}(5,0)(-2,-1)
\put(0,0){\circle*{2}}
\end{picture}}

\newcommand{\tdeux}{\begin{picture}(7,7)(0,-1)
\put(3,0){\circle*{2}}
\put(3,0){\line(0,1){5}}
\put(3,5){\circle*{2}}
\end{picture}}

\newcommand{\ttroisun}{\begin{picture}(15,8)(-5,-1)
\put(3,0){\circle*{2}}
\put(-0.65,0){$\vee$}
\put(6,7){\circle*{2}}
\put(0,7){\circle*{2}}
\end{picture}}
\newcommand{\ttroisdeux}{\begin{picture}(5,12)(-2,-1)
\put(0,0){\circle*{2}}
\put(0,0){\line(0,1){5}}
\put(0,5){\circle*{2}}
\put(0,5){\line(0,1){5}}
\put(0,10){\circle*{2}}
\end{picture}}

\newcommand{\tquatreun}{\begin{picture}(15,12)(-5,-1)
\put(3,0){\circle*{2}}
\put(-0.65,0){$\vee$}
\put(6,7){\circle*{2}}
\put(0,7){\circle*{2}}
\put(3,7){\circle*{2}}
\put(3,0){\line(0,1){7}}
\end{picture}}
\newcommand{\tquatredeux}{\begin{picture}(15,18)(-5,-1)
\put(3,0){\circle*{2}}
\put(-0.65,0){$\vee$}
\put(6,7){\circle*{2}}
\put(0,7){\circle*{2}}
\put(0,14){\circle*{2}}
\put(0,7){\line(0,1){7}}
\end{picture}}
\newcommand{\tquatretrois}{\begin{picture}(15,18)(-5,-1)
\put(3,0){\circle*{2}}
\put(-0.65,0){$\vee$}
\put(6,7){\circle*{2}}
\put(0,7){\circle*{2}}
\put(6,14){\circle*{2}}
\put(6,7){\line(0,1){7}}
\end{picture}}
\newcommand{\tquatrequatre}{\begin{picture}(15,18)(-5,-1)
\put(3,5){\circle*{2}}
\put(-0.65,5){$\vee$}
\put(6,12){\circle*{2}}
\put(0,12){\circle*{2}}
\put(3,0){\circle*{2}}
\put(3,0){\line(0,1){5}}
\end{picture}}
\newcommand{\tquatrecinq}{\begin{picture}(9,19)(-2,-1)
\put(0,0){\circle*{2}}
\put(0,0){\line(0,1){5}}
\put(0,5){\circle*{2}}
\put(0,5){\line(0,1){5}}
\put(0,10){\circle*{2}}
\put(0,10){\line(0,1){5}}
\put(0,15){\circle*{2}}
\end{picture}}

\newcommand{\tdun}[1]{\begin{picture}(10,5)(-2,-1)
\put(0,0){\circle*{2}}
\put(3,-2){\tiny #1}
\end{picture}}

\newcommand{\tddeux}[2]{\begin{picture}(12,5)(0,-1)
\put(3,0){\circle*{2}}
\put(3,0){\line(0,1){5}}
\put(3,5){\circle*{2}}
\put(6,-2){\tiny #1}
\put(6,3){\tiny #2}
\end{picture}}

\newcommand{\tdtroisun}[3]{\begin{picture}(20,12)(-5,-1)
\put(3,0){\circle*{2}}
\put(-0.65,0){$\vee$}
\put(6,7){\circle*{2}}
\put(0,7){\circle*{2}}
\put(5,-2){\tiny #1}
\put(9,5){\tiny #2}
\put(-5,5){\tiny #3}
\end{picture}}
\newcommand{\tdtroisdeux}[3]{\begin{picture}(12,14)(-2,-1)
\put(0,0){\circle*{2}}
\put(0,0){\line(0,1){5}}
\put(0,5){\circle*{2}}
\put(0,5){\line(0,1){5}}
\put(0,10){\circle*{2}}
\put(3,-2){\tiny #1}
\put(3,3){\tiny #2}
\put(3,9){\tiny #3}
\end{picture}}

\newcommand{\tdquatreun}[4]{\begin{picture}(20,12)(-5,-1)
\put(3,0){\circle*{2}}
\put(-0.6,0){$\vee$}
\put(6,7){\circle*{2}}
\put(0,7){\circle*{2}}
\put(3,7){\circle*{2}}
\put(3,0){\line(0,1){7}}
\put(5,-2){\tiny #1}
\put(8.5,5){\tiny #2}
\put(1,10){\tiny #3}
\put(-5,5){\tiny #4}
\end{picture}}
\newcommand{\tdquatredeux}[4]{\begin{picture}(20,20)(-5,-1)
\put(3,0){\circle*{2}}
\put(-.65,0){$\vee$}
\put(6,7){\circle*{2}}
\put(0,7){\circle*{2}}
\put(0,14){\circle*{2}}
\put(0,7){\line(0,1){7}}
\put(5,-2){\tiny #1}
\put(9,5){\tiny #2}
\put(-5,5){\tiny #3}
\put(-5,12){\tiny #4}
\end{picture}}
\newcommand{\tdquatretrois}[4]{\begin{picture}(20,20)(-5,-1)
\put(3,0){\circle*{2}}
\put(-.65,0){$\vee$}
\put(6,7){\circle*{2}}
\put(0,7){\circle*{2}}
\put(6,14){\circle*{2}}
\put(6,7){\line(0,1){7}}
\put(5,-2){\tiny #1}
\put(9,5){\tiny #2}
\put(-5,5){\tiny #4}
\put(9,12){\tiny #3}
\end{picture}}
\newcommand{\tdquatrequatre}[4]{\begin{picture}(20,14)(-5,-1)
\put(3,5){\circle*{2}}
\put(-.65,5){$\vee$}
\put(6,12){\circle*{2}}
\put(0,12){\circle*{2}}
\put(3,0){\circle*{2}}
\put(3,0){\line(0,1){5}}
\put(6,-3){\tiny #1}
\put(6,4){\tiny #2}
\put(9,12){\tiny #3}
\put(-5,12){\tiny #4}
\end{picture}}
\newcommand{\tdquatrecinq}[4]{\begin{picture}(12,19)(-2,-1)
\put(0,0){\circle*{2}}
\put(0,0){\line(0,1){5}}
\put(0,5){\circle*{2}}
\put(0,5){\line(0,1){5}}
\put(0,10){\circle*{2}}
\put(0,10){\line(0,1){5}}
\put(0,15){\circle*{2}}
\put(3,-2){\tiny #1}
\put(3,3){\tiny #2}
\put(3,9){\tiny #3}
\put(3,14){\tiny #4}
\end{picture}}

\newcommand{\tdcinqquatre}[5]{\begin{picture}(15,14)(-5,-1)
\put(3,0){\circle*{2}}
\put(-0.65,0){$\vee$}
\put(6,7){\circle*{2}}
\put(0,7){\circle*{2}}
\put(3,7){\circle*{2}}
\put(3,0){\line(0,1){7}}
\put(6,7){\line(0,1){7}}
\put(6,14){\circle*{2}}
\put(5,-2){\tiny #1}
\put(9,5){\tiny #2}
\put(1,10){\tiny #4}
\put(-5,5){\tiny #5}
\put(9,12){\tiny #3}
\end{picture}}
\newcommand{\tdcinqsept}[5]{\begin{picture}(15,8)(-5,-1)
\put(3,0){\circle*{2}}
\put(-0.65,0){$\vee$}
\put(6,7){\circle*{2}}
\put(0,7){\circle*{2}}
\put(2.35,7){$\vee$}
\put(3,14){\circle*{2}}
\put(9,14){\circle*{2}}
\put(5,-2){\tiny #1}
\put(9,5){\tiny #2}
\put(-1,12){\tiny #4}
\put(-5,5){\tiny #5}
\put(11,12){\tiny #3}
\end{picture}}
\newcommand{\tdcinqneuf}[5]{\begin{picture}(15,26)(-5,-1)
\put(3,0){\circle*{2}}
\put(-0.65,0){$\vee$}
\put(6,7){\circle*{2}}
\put(0,7){\circle*{2}}
\put(6,14){\circle*{2}}
\put(6,7){\line(0,1){7}}
\put(6,21){\circle*{2}}
\put(6,14){\line(0,1){7}}
\put(5,-2){\tiny #1}
\put(9,5){\tiny #2}
\put(-5,5){\tiny #5}
\put(9,12){\tiny #3}
\put(9,19){\tiny #4}
\end{picture}}

\newcommand{\tdelta}{\tilde{\Delta}}
\newcommand{\mmodels}{\mid \hspace{-.2mm} \models}
\newcommand{\D}{\mathcal{D}}

\title{Algèbres de greffes}
\date{}

\author{Anthony Mansuy \\ \\
{\small{\it Laboratoire de Mathématiques, Université de Reims}}\\
\small{{\it Moulin de la Housse - BP 1039 - 51687 REIMS Cedex 2, France}}\\
\small{e-mail : anthony.mansuy@univ-reims.fr}}

\newtheorem{defi}{\indent Définition}
\newtheorem{lemma}[defi]{\indent Lemme}
\newtheorem{cor}[defi]{\indent Corollaire}
\newtheorem{theo}[defi]{\indent Théorème}
\newtheorem{prop}[defi]{\indent Proposition}

\newenvironment{proof}{{\bf Preuve.}}{\hfill $\Box$}
\usepackage{fancyhdr}
\begin{document}
\maketitle

\textbf{Résumé.} Pour étudier certains ensembles de probabilités, appelés moyennes induites par J. Ecalle, F. Menous introduit deux opérateurs de greffes $ B^{+} $ et $ B^{-} $. \`A partir de ces deux opérateurs, nous construisons des algèbres de Hopf d'arbres enracinés et ordonnés $ \mathcal{B}^{i} $, $ i \in \mathbb{N}^{\ast} $, $ \mathcal{B}^{\infty} $ et $ \mathcal{B} $ vérifiant les relations d'inclusions $ \mathcal{B}^{1} \subseteq \hdots \mathcal{B}^{i} \subseteq \mathcal{B}^{i+1} \subseteq \hdots \subseteq \mathcal{B}^{\infty} \subseteq \mathcal{B} $. On munit $ \mathcal{B} $ d'une structure de bialgèbre dupliciale dendriforme et on en déduit que $ \mathcal{B} $ est colibre et auto-duale. Nous introduisons enfin la notion d'algèbre bigreffe et démontrons que $ \mathcal{B} $ est engendrée comme algèbre bigreffe par l'élément $ \tdun{1} $.\\

\textbf{Abstract.} In order to study some sets of probabilities, called induced averages by J. Ecalle, F. Menous introduces two grafting operators $ B^{+} $ and $ B^{-} $. With these two operators, we construct Hopf algebras of rooted and ordered trees $ \mathcal{B}^{i} $, $ i \in \mathbb{N}^{\ast} $, $ \mathcal{B}^{\infty} $ and $ \mathcal{B} $ satisfying the inclusion relations $ \mathcal{B}^{1} \subseteq \hdots \mathcal{B}^{i} \subseteq \mathcal{B}^{i+1} \subseteq \hdots \subseteq \mathcal{B}^{\infty} \subseteq \mathcal{B} $. We endow $ \mathcal{B} $ with a structure of duplicial dendriform bialgebra and we deduce that $ \mathcal{B} $ is cofree and self-dual. Finally, we introduce the notion of bigraft algebra and we prove that $ \mathcal{B} $ is generated as bigraft algebra by the element $ \tdun{1} $.\\

\textbf{Keywords.} Planar rooted trees, Hopf algebra of ordered forests, Duplicial dendriform bialgebra.\\

\textbf{AMS Classification.} 05C05, 16W30.

\tableofcontents

\section*{Introduction}

L'algèbre de Hopf des arbres enracinés de A. Connes et D. Kreimer est décrite dans \cite{Connes} dans le contexte des théories de champs quantiques: elle est utilisée lors de la procédure de Renormalisation. Cette algèbre de Hopf est engendrée par l'ensemble des arbres enracinés, et le coproduit est donné par les coupes admissibles. D'autres algèbres de Hopf d'arbres sont obtenues à partir de celle-ci en ajoutant une structure supplémentaire aux arbres enracinés. Par exemple, avec des données planaires en plus, on obtient l'algèbre de Hopf des arbres plans $ \mathcal{H}_{PR} $ ainsi que sa version décorée $ \mathcal{H}_{PR}^{\mathcal{D}} $ \cite{Foissy1,Foissy2,Holtkamp}; en munissant les sommets d'un ordre total, on obtient l'algèbre de Hopf des arbres ordonnés $ \mathcal{H}_{o} $. Dans cet article, nous nous intéressons à différentes sous-algèbres de Hopf de $ \mathcal{H}_{o} $, dont la construction a été inspirée par les travaux de F. Menous.

Dans \cite{Menous}, F. Menous étudie certains ensembles de probabilités, appelés moyennes induites par J. Ecalle, associés à une variable aléatoire sur $ \mathbb{R} $. Une moyenne est un ensemble de "poids" indexés par des mots sur l'alphabet à deux éléments $ \{+,-\} $. Pour un mot $ (\varepsilon_{1}, \hdots , \varepsilon_{n}) $ donné ($ \varepsilon_{i} = \pm $), le poids est simplement la probabilité d'appartenir à $ \mathbb{R}^{\varepsilon_{1}} $ au temps 1, à $ \mathbb{R}^{\varepsilon_{2}} $ au temps 2, ..., à $ \mathbb{R}^{\varepsilon_{n}} $ au temps $ n $. F. Menous prouve dans \cite{Menous} qu'un tel coefficient peut être décomposé en une somme de coefficients élémentaires qui sont indexés par des arbres et des forêts ordonnés. Pour cela, il construit par récurence, avec un formalisme proche de celui du calcul moulien, un ensemble de forêts ordonnées. C'est cet ensemble, noté $ \mathcal{G} $, qu'on se propose d'étudier ici.

Pour cela, nous définissons deux opérateurs de greffes $ B^{+} $ et $ B^{-} $, à partir desquels nous proposons une construction de l'ensemble $ \mathcal{G} $, ce qui permet de définir l'algèbre de Hopf $ \mathcal{B}^{\infty} = \mathbb{K} [\mathcal{G}] $. Une étude combinatoire de $ \mathcal{G} $ permet de construire un "dévissage" de l'algèbre de Hopf $ \mathcal{B}^{\infty} $ de sorte qu'on a les inclusions $ \mathcal{B}^{1} \subseteq \hdots \subseteq \mathcal{B}^{i} \subseteq \mathcal{B}^{i+1} \subseteq \hdots \subseteq \mathcal{B}^{\infty} $, où $ \mathcal{B}^{i} $ est une algèbre de Hopf pour tout $ i \in \mathbb{N}^{\ast} $.

En généralisant la construction de F. Menous, nous définissons une algèbre de Hopf $ \mathcal{B} = \mathbb{K} [\mathcal{T}] $ à partir d'un ensemble d'arbres ordonnés $ \mathcal{T} $ construit avec les opérateurs $ B^{+} $ et $ B^{-} $. Par des raisonnements proches de ceux utilisés dans \cite{Foissy4}, on démontre la coliberté et l'auto-dualité de $ \mathcal{B} $. On s'attache ensuite à munir $ \mathcal{B} $ d'une structure d'algèbre de greffes à gauche, d'algèbre de greffes à droite qui sont des quotients d'algèbres diptères (voir \cite{Loday}). Cela nous amène naturellement à la notion d'algèbre bigreffe, et on démontre que $ \mathcal{B} $ est engendrée comme algèbre bigreffe par l'unique arbre de degré 1.\\

Le texte est organisé comme suit: dans la première partie, nous construisons, avec des opérateurs de greffes $ B^{+} $ et $ B^{-} $, un ensemble $ \mathcal{G} $ d'arbres ordonnés et nous exhibons différentes propriétés combinatoires de $ \mathcal{G} $. \`A partir de cet ensemble, nous définissons l'algèbre $ \mathcal{B}^{\infty} $, nous démontrons que c'est une algèbre de Hopf et nous calculons sa série formelle. On construit ensuite des sous-algèbres $ \mathcal{B}^{i} $ de $ \mathcal{B}^{\infty} $, pour $ i \in \mathbb{N}^{\ast} $, et on démontre que ce sont des algèbres de Hopf. Dans la seconde partie, nous définissons l'algèbre $ \mathcal{B} $ à partir d'un ensemble $ \mathcal{T} $ lui aussi construit avec les opérateurs de greffes $ B^{+} $ et $ B^{-} $. On munit $ \mathcal{B} $ d'une structure de bialgèbre dupliciale dendriforme (voir \cite{Foissy3,Foissy4,Loday2}) et cela nous permet de montrer la coliberté et l'auto-dualité de $ \mathcal{B} $. Nous définissons enfin les algèbres de greffes à gauche, les algèbres de greffes à droite et les algèbres bigreffes et on démontre que $ \mathcal{B} $ est engendrée comme algèbre bigreffe par l'élément $ \tdun{1} $.\\

{\bf Remerciements.} {J'aimerais remercier mon directeur de thèse, Loïc Foissy, qui m'a proposé d'étudier ces opérateurs de greffes et soutenu dans mes recherches.}

\begin{center}
\textsc{Notations et rappels}
\end{center}

Tout les espaces vectoriels et les algèbres sont définis sur un corps $ \mathbb{K} $. \'Etant donné un ensemble $ X $, nous noterons par $ \mathbb{K} [X] $ l'espace vectoriel engendré par $ X $. Pour tout espaces vectoriels $ V $ et $ W $, $ V \otimes W $ désignera le produit tensoriel de $ V $ par $ W $ sur $ \mathbb{K} $.\\

Nous rappelons brièvement la construction de l'algèbre de Hopf des arbres enracinés plans \cite{Foissy1,Holtkamp}, qui généralise celle de Connes-Kreimer des arbres enracinés \cite{Connes}. Un \textit{arbre enraciné} est un graphe fini sans boucle avec un sommet particulier appelé la \textit{racine} \cite{Stanley}. Une \textit{forêt enracinée} est un graphe fini $ F $ dont toutes les composantes connexes sont des arbres enracinés. L'ensemble des sommets d'une forêt enracinée $ F $ est noté $ V(F) $. Le \textit{degré} d'une forêt $ F $ est le nombre, noté $ \left| F \right| $, de sommets de $ F $. L'unique forêt de degré $ 0 $ est l'arbre vide noté $ 1 $. La \textit{longueur} d'une forêt $ F $ est le nombre de composantes connexes de $ F $. Voici par exemple les forêts enracinées de degré $ \leq 4 $:
$$ 1, \tun, \tun\tun, \tdeux, \tun\tun\tun, \tdeux\tun, \ttroisun, \ttroisdeux, \tun\tun\tun\tun, \tdeux\tun\tun, \tdeux\tdeux, \ttroisun\tun, \ttroisdeux\tun, \tquatreun, \tquatredeux, \tquatrequatre, \tquatrecinq .$$

Soit $ F $ une forêt enracinée. Les arêtes de $ F $ sont orientées vers le bas (des feuilles vers les racines). Si $ v,w \in V(F) $, on notera $ v \rightarrow w $ s'il y a une arête de $ F $ de $ v $ vers $ w $ et $ v \twoheadrightarrow w $ si il y a un chemin orienté de $ v $ vers $ w $ dans $ F $. Par convention, $ v \twoheadrightarrow v $ pour tout $ v \in V(F) $. Si $ v , w \in V(F) $, on dira que $ v $ est un \textit{descendant} de $ w $ si $ v \twoheadrightarrow w $ et que $ v $ est un \textit{ancêtre} de $ w $ si $ w \twoheadrightarrow v $. Si $ v \in V(F) $, nous noterons $ h(v) $ la \textit{hauteur} de $ v $, c'est-à-dire le nombre d'arêtes sur le chemin orienté entre $ v $ et la racine de l'arbre qui a $ v $ pour sommet, et $ f_{v} $ la \textit{fertilité} de $ v $ égale au cardinal de $ \{w \in V(F)\:\mid \:w \rightarrow v\} $. On dira qu'un arbre est une \textit{échelle} si tous ses sommets qui ne sont pas des feuilles sont de fertilité égale à 1.\\

Une \textit{forêt plane} est une forêt enracinée $ F $ tel que l'ensemble des racines de $ F $ est totalement ordonné et, pour tout sommet $ v \in V(F) $, l'ensemble $\{w \in V(F)\:\mid \:w \rightarrow v\}$ est totalement ordonné. Par exemple, voici les forêts planes de degré $ \leq 4 $:
$$ 1,\tun, \tun\tun,\tdeux,\tun\tun\tun,\tdeux\tun,\tun\tdeux,\ttroisun,\ttroisdeux,\tun\tun\tun\tun,\tdeux\tun\tun,\tun\tdeux\tun,\tun\tun\tdeux,\tdeux\tdeux,
\ttroisun\tun,\ttroisdeux\tun,\tun\ttroisun,\tun\ttroisdeux,\tquatreun,\tquatredeux,\tquatretrois,\tquatrequatre,\tquatrecinq .$$
On peut définir un ordre total sur les sommets de $ F $ de la manière suivante: si $ v,w \in V(F) $, $ v \neq w $, alors $ v < w $ si l'une des conditions suivantes est satisfaite:
\begin{enumerate}
\item $ v $ et $ w $ ne sont pas les sommets d'un même arbre de la forêt $ F $, et l'arbre ayant pour sommet $ v $ est plus à gauche que l'arbre ayant pour sommet $ w $
\item $ v $ et $ w $ sont les sommets d'un même arbre de la forêt $ F $, et il existe un chemin orienté de $ v $ vers $ w $, c'est-à-dire $ v \twoheadrightarrow w $,
\item $ v $ et $ w $ sont les sommets d'un même arbre de la forêt $ F $, il n'existe pas de chemin orienté de $ v $ vers $ w $ ni de $ w $ vers $ v $ et, en notant $ x $ l'ancêtre commun de $ v $ et $ w $ de hauteur maximale, l'arête terminant le chemin de $ v $ vers $ x $ est plus à gauche que l'arête terminant le chemin de $ w $ vers $ x $.
\end{enumerate}
De plus, $ v \leq w $ si $ v = w $ ou $ v < w $.\\

Soit $ \boldsymbol{v} $ un sous-ensemble de $ V(F) $. On dira que $ \boldsymbol{v} $ est une \textit{coupe admissible} de $ F $, et on écrira $ \boldsymbol{v} \models V(F) $, si $ \boldsymbol{v} $ est totalement déconnecté, c'est-à-dire que $ v \twoheadrightarrow w \hspace{-.7cm} / \hspace{.7cm} $ pour tout couple $ (v,w) $ d'éléments distincts de $ \boldsymbol{v} $. Si $ \boldsymbol{v} \models V(F) $, $ Lea_{\boldsymbol{v}} (F) $ est la sous-forêt enracinée plane de $ F $ obtenu en gardant seulement les sommets "au-dessus" de $ \boldsymbol{v} $, c'est-à-dire $\{ w \in V(F) \:\mid \: \exists v \in \boldsymbol{v}, \:w \twoheadrightarrow v \}$. Remarquons que $ \boldsymbol{v} \subseteq Lea_{\boldsymbol{v}}(F) $. $ Roo_{\boldsymbol{v}}(F) $ est la sous-forêt enracinée plane obtenue en gardant les autres sommets.

En particulier, si $ \boldsymbol{v} = \emptyset $, alors $ Lea_{\boldsymbol{v}}(F) = 1 $ et $ Roo_{\boldsymbol{v}}(F) = F $: c'est la \textit{coupe vide} de $ F $. Si $ \boldsymbol{v} $ contient les racines de $ F $, alors il contient uniquement les racines de $ F $, et $ Lea_{\boldsymbol{v}}(F) = F $, $ Roo_{\boldsymbol{v}}(F) = 1 $: c'est la \textit{coupe totale} de $ F $. On écrira $ \boldsymbol{v} \mmodels V(F) $ si $ \boldsymbol{v} $ est une coupe admissible non vide, non totale de $ F $. Une coupe admissible $ \boldsymbol{v} $ est une \textit{coupe simple} si $ {\rm card} (\boldsymbol{v}) = 1 $.

Il est prouvé dans \cite{Foissy1} que l'espace $ \mathcal{H}_{PR} $ généré par les forêts planes est une bigèbre. Le produit est donné par la concaténation des forêts planes et le coproduit est définit pour toute forêt enracinée plane $ F $ par:
\begin{eqnarray*}
\Delta(F) &= & \sum_{\boldsymbol{v} \models V(F)} Lea_{\boldsymbol{v}}(F) \otimes Roo_{\boldsymbol{v}}(F),\\
& = & F \otimes 1+1\otimes F+\sum_{\boldsymbol{v} \mmodels V(F)} Lea_{\boldsymbol{v}}(F) \otimes Roo_{\boldsymbol{v}}(F), {\rm ~ si ~} F \neq 1.
\end{eqnarray*}
Si $ F $ est non vide, on pose $ \tdelta(F) = \Delta(F) - (F \otimes 1+1\otimes F) $.\\

Nous aurons besoin d'une version décorée de cette algèbre de Hopf. Si $ \mathcal{D} $ est un ensemble non vide, une \textit{forêt plane décorée} est un couple $ (F,d) $, où $ F $ est une forêt enracinée plane et $ d : V(F) \rightarrow \mathcal{D} $ une application. L'algèbre des forêts planes décorées $ \mathcal{H}_{PR}^{\mathcal{D}} $ est encore une algèbre de Hopf. Nous donnons ci-dessous les arbres plans décorés de degré $ \leq 4 $:
$$ \tdun{a}, ~ a \in \mathcal{D}, \hspace{0.5cm} \tddeux{a}{b}, ~ (a,b) \in \mathcal{D}^{2}, \hspace{0.5cm} \tdtroisun{a}{b}{c}, \tdtroisdeux{a}{b}{c}, ~ (a,b,c) \in \mathcal{D}^{3}, $$
$$ \tdquatreun{a}{b}{c}{d}, \tdquatredeux{a}{b}{c}{d}, \tdquatretrois{a}{b}{c}{d}, \tdquatrequatre{a}{b}{c}{d}, \tdquatrecinq{a}{b}{c}{d}, ~ (a,b,c,d) \in \mathcal{D}^{4} .$$

Rappelons aussi la notion de forêt ordonnée. Une \textit{forêt ordonnée} est une forêt enracinée avec un ordre total sur l'ensemble de ces sommets. Notons $ \mathcal{H}_{o} $ le $ \mathbb{K} $-espace vectoriel engendré par les forêts ordonnées. Si $ F $ et $ G $ sont deux forêts ordonnées, alors la forêt ordonnée $ F G $ est aussi une forêt ordonnée avec, pour tout $ v \in F $, $ w \in G $, $ v < w $. Cela définit un produit non commutatif sur l'ensemble des forêts ordonnées, qui s'étend linéairement à $ \mathcal{H}_{o} $. Par exemple, $ (\tddeux{1}{2}) \times (\tdun{1} \tdtroisun{4}{3}{2}) = \tddeux{1}{2} \tdun{3} \tdtroisun{6}{5}{4} $. Voici les forêts ordonnées de degré $ \leq 3 $:
$$ 1,\tdun{1},\tdun{1}\tdun{2},\tddeux{1}{2},\tddeux{2}{1},\tdun{1}\tdun{2}\tdun{3},\tdun{1}\tddeux{2}{3},\tdun{1}\tddeux{3}{2},\tddeux{1}{3}\tdun{2},\tdun{2}\tddeux{3}{1},\tddeux{1}{2}\tdun{3},\tddeux{2}{1}\tdun{3},
\tdtroisun{1}{3}{2},\tdtroisun{2}{3}{1},\tdtroisun{3}{2}{1},\tdtroisdeux{1}{2}{3},\tdtroisdeux{1}{3}{2},\tdtroisdeux{2}{1}{3},\tdtroisdeux{2}{3}{1},\tdtroisdeux{3}{1}{2},\tdtroisdeux{3}{2}{1} .$$

Si $ F $ est une forêt ordonnée, alors toute sous-forêt de $ F $ est aussi ordonnée. On peut donc définir un coproduit sur $ \mathcal{H}_{o} $ comme suit: pour toute forêt ordonnée $ F $,
\begin{eqnarray} \label{coproduitordonné}
\Delta(F)=\sum_{\boldsymbol{v} \models V(F)} Lea_{\boldsymbol{v}}(F) \otimes Roo_{\boldsymbol{v}}(F).
\end{eqnarray}
Par exemple,
$$\Delta\left(\tdquatredeux{2}{3}{4}{1}\right)=\tdquatredeux{2}{3}{4}{1} \otimes 1+1\otimes \tdquatredeux{2}{3}{4}{1}
+\tdun{1} \otimes \tdtroisun{1}{2}{3}+\tddeux{2}{1} \otimes \tddeux{1}{2}+\tdun{1} \otimes \tdtroisdeux{2}{3}{1}
+\tdun{1}\tdun{2} \otimes \tddeux{1}{2}+\tddeux{3}{1}\tdun{2} \otimes \tdun{1}.$$
Les forêts planes sont ordonnées, en numérotant les sommets suivant l'ordre défini plus haut. Réciproquement, les forêts ordonnées sont planes. Il suffit de supprimer l'indexation des sommets.

\section{Les algèbres de Hopf $ \mathcal{B}^{i} $}

\subsection{Construction et étude de $ \mathcal{G} $} \label{partie1}

Commençons par introduire deux opérateurs de greffes $ B^{+} $ et $ B^{-} $ qui sont utilisés dans tout ce qui suit. Pour cela, considèrons une suite de $ m $ arbres ordonnés non vides $ T_{1}, \hdots ,T_{m} $ dont la somme des degrés est notée $ n $. On pose:
\begin{enumerate}
\item $ B^{-}(T_{1}, \hdots ,T_{m}) $ l'arbre ordonné de degré $ n+1 $ obtenu comme suit: on considère $ T_{1}, \hdots ,T_{m} $ comme la suite des sous-arbres d'un arbre enraciné ayant pour racine le sommet indexé par $ n+1 $. De plus, on conviendra que $ B^{-}(1) $ est égale à l'arbre $ \tdun{1} $, où $ 1 $ désigne l'arbre vide.
\item $ B^{+}(T_{1}, \hdots ,T_{m}) $ l'arbre ordonné de degré $ n+1 $ construit en greffant le sommet indexé par $ n+1 $ comme le fils le plus à droite de la racine de $ T_{1} $ et en considérant alors $ T_{2},\hdots,T_{m} $ comme la suite des sous-arbres issus du sommet indexé par $ n+1 $. En particulier, on notera $ B^{+}(T_{1}) = B^{+}(T_{1},1) $ l'arbre obtenu en greffant le sommet indexé par $ \left| T_{1} \right| + 1 $ comme le fils le plus à droite de la racine de $ T_{1} $. De plus, on conviendra que $ B^{+}(1) $ est égale à l'arbre $ \tdun{1} $.
\end{enumerate}

\vspace{0.3cm}

{\bf Note.} {Les opérateurs $ B^{+} $ et $ B^{-} $ sont différents de ceux introduits dans \cite{Connes} par A. Connes et D. Kreimer}

\vspace{0.3cm}

{\bf Exemples.} {$ B^{-}(\tddeux{1}{2},\tdun{1}) = \tdquatredeux{4}{3}{1}{2} $, $ B^{+}(\tdtroisun{3}{2}{1}) = B^{+}(\tdtroisun{3}{2}{1},1) = \tdquatreun{3}{4}{2}{1} $, $ B^{+}(\tddeux{1}{2},\tdun{1}) = \tdquatretrois{1}{4}{3}{2} $.}

\vspace{0.5cm}

Pour toute suite $ \underline{\varepsilon} = \varepsilon_{1}, \ldots ,\varepsilon_{n} \in \lbrace +,- \rbrace ^{n} $, avec $ n \geq 1 $, on définit par récurrence un ensemble $ \mathcal{G}^{(\underline{\varepsilon})} $ qui correspond à un ensemble de forêts ordonnées de degré $ n $. \\

Si $ \underline{n=1} $, $ \mathcal{G}^{(\varepsilon_{1})} $, pour $ \varepsilon_{1} $ quelconque, est l'ensemble réduit à un seul élément, la forêt de degré 1, l'unique sommet étant évidemment indexé par 1.\\

Si $ \underline{n \geq 2} $, considérons l'ensemble $ \mathcal{G}^{(\varepsilon_{1}, \hdots,\varepsilon_{n-1})} $ déjà construit.

\begin{enumerate}
\item Si $ \varepsilon_{n}=- $, les éléments $ F $ de $ \mathcal{G}^{(\varepsilon_{1}, \hdots,\varepsilon_{n})} $ sont obtenus par la transformation suivante. On prend un élément $ F' = T_{1} \hdots T_{m} $ de $ \mathcal{G}^{(\varepsilon_{1}, \hdots,\varepsilon_{n-1})} $, avec $ m \geq 1 $, et on considère $ T_{1}, \hdots ,T_{m} $ comme la suite des sous-arbres d'un arbre enraciné ayant pour racine le sommet indexé par $ n $. Cela donne ainsi naissance à une nouvelle forêt ordonnée de degré $ n $, avec un seul arbre, égale à $ B^{-}(T_{1}, \hdots,T_{m}) $.

\item Si $ \varepsilon_{n}=+ $, alors comme précédement on considère un élément $ F' $ de $ \mathcal{G}^{(\varepsilon_{1}, \hdots,\varepsilon_{n-1})} $. Nous avons alors plusieurs possibilités pour ajouter un nouveau sommet indexé par $ n $. Notons encore $ F'=T_{1} \hdots T_{m} $, où $ T_{1},\hdots,T_{m} $ est la suite des arbres qui composent la forêt $ F' $ et $ m \geq 1 $. On peut alors faire l'une des transformations suivantes pour obtenir un élément de $ \mathcal{G}^{(\varepsilon_{1}, \hdots,\varepsilon_{n})} $.
\begin{enumerate}
\item Concaténer l'unique arbre de degré 1 indexé par $ n $ à la forêt $ T_{1} \hdots T_{m} $ sur la droite. Cela donne la forêt ordonnée de degré $ n $ égale à $ T_{1} \hdots T_{m} \tdun{$n$} $.
\item Pour $ 1 \leq i \leq m $, greffer le sommet indexé par $ n $ comme le fils le plus à droite de la racine de $ T_{i} $ et considérer alors $ T_{i+1},\hdots,T_{m} $ comme la suite des sous-arbres issus du sommet indexé par $ n $. On obtient alors la forêt ordonnée de degré $ n $ égale à $ T_{1} \hdots T_{i-1} B^{+}(T_{i}, \hdots,T_{m}) $.
\end{enumerate}
\end{enumerate}

Voici une illustration de cette construction pour $ n = 1,2,3,4 $:
\vspace{-0.3cm}

\begin{eqnarray*}
\mathcal{G}^{(+)}=\mathcal{G}^{(-)}&=&\{\tdun{1}\}\\
\mathcal{G}^{(+,+)}=\mathcal{G}^{(-,+)}&=&\{\tdun{1}\tdun{2},\tddeux{1}{2}\}\\
\mathcal{G}^{(+,-)}=\mathcal{G}^{(-,-)}&=&\{\tddeux{2}{1}\}\\
\mathcal{G}^{(+,+,+)}=\mathcal{G}^{(-,+,+)}&=&\{\tdun{1}\tdun{2}\tdun{3},\tdun{1}\tddeux{2}{3},\tdtroisdeux{1}{3}{2},
\tddeux{1}{2}\tdun{3},\tdtroisun{1}{3}{2}\}\\
\mathcal{G}^{(+,-,+)}=\mathcal{G}^{(-,-,+)}&=&\{\tddeux{2}{1}\tdun{3},\tdtroisun{2}{3}{1}\}\\
\mathcal{G}^{(+,+,-)}=\mathcal{G}^{(-,+,-)}&=&\{\tdtroisun{3}{2}{1},\tdtroisdeux{3}{1}{2}\}\\
\mathcal{G}^{(+,-,-)}=\mathcal{G}^{(-,-,-)}&=&\{\tdtroisdeux{3}{2}{1}\}\\
\mathcal{G}^{(+,+,+,+)}=\mathcal{G}^{(-,+,+,+)}&=&\{\tdun{1}\tdun{2}\tdun{3}\tdun{4},\tdun{1}\tdun{2}\tddeux{3}{4},
\tdun{1}\tdtroisdeux{2}{4}{3},\tdquatrequatre{1}{4}{3}{2},\tdun{1}\tddeux{2}{3}\tdun{4},\tdun{1}\tdtroisun{2}{4}{3},
\tdquatrecinq{1}{4}{2}{3},\\
&&\tdtroisdeux{1}{3}{2}\tdun{4},\tdquatredeux{1}{4}{3}{2},\tddeux{1}{2}\tdun{3}\tdun{4},\tddeux{1}{2}\tddeux{3}{4},
\tdquatretrois{1}{4}{3}{2},\tdtroisun{1}{3}{2}\tdun{4},\tdquatreun{1}{4}{3}{2}\}\\
\mathcal{G}^{(+,-,+,+)}=\mathcal{G}^{(-,-,+,+)}&=&\{\tddeux{2}{1}\tdun{3}\tdun{4}, \tddeux{2}{1}\tddeux{3}{4},
\tdquatretrois{2}{4}{3}{1},\tdtroisun{2}{3}{1}\tdun{4},\tdquatreun{2}{4}{3}{1}\}\\
\mathcal{G}^{(+,+,-,+)}=\mathcal{G}^{(-,+,-,+)}&=&\{\tdtroisun{3}{2}{1}\tdun{4},\tdquatreun{3}{4}{2}{1},
\tdtroisdeux{3}{1}{2}\tdun{4},\tdquatredeux{3}{4}{1}{2}\}\\
\mathcal{G}^{(+,-,-,+)}=\mathcal{G}^{(-,-,-,+)}&=&\{\tdtroisdeux{3}{2}{1}\tdun{4},\tdquatredeux{3}{4}{2}{1}\}\\
\mathcal{G}^{(+,+,+,-)}=\mathcal{G}^{(-,+,+,-)}&=&\{\tdquatreun{4}{3}{2}{1},\tdquatretrois{4}{2}{3}{1},
\tdquatrecinq{4}{1}{3}{2},\tdquatredeux{4}{3}{1}{2},\tdquatrequatre{4}{1}{3}{2}\}
\end{eqnarray*}

\begin{eqnarray*}
\mathcal{G}^{(+,-,+,-)}=\mathcal{G}^{(-,-,+,-)}&=&\{\tdquatredeux{4}{3}{2}{1},\tdquatrequatre{4}{2}{3}{1}\} \hspace{5cm} \\
\mathcal{G}^{(+,+,-,-)}=\mathcal{G}^{(-,+,-,-)}&=&\{\tdquatrequatre{4}{3}{2}{1},\tdquatrecinq{4}{3}{1}{2}\}\\
\mathcal{G}^{(+,-,-,-)}=\mathcal{G}^{(-,-,-,-)}&=&\{\tdquatrecinq{4}{3}{2}{1}\}
\end{eqnarray*}

Dans la suite, étant donné $ \underline{\varepsilon} \in \lbrace +,- \rbrace ^{n} $ avec $ n \geq 1 $, on identifiera toujours les deux ensembles $ \mathcal{G}^{(+,\underline{\varepsilon})} $ et $ \mathcal{G}^{(-,\underline{\varepsilon})} $ (et les deux ensembles $ \mathcal{G}^{(+)} $ et $ \mathcal{G}^{(-)} $). On préferera, suivant les cas, dire qu'une forêt appartient à $ \mathcal{G}^{(+,\underline{\varepsilon})} $ ou à $ \mathcal{G}^{(-,\underline{\varepsilon})} $.\\

Pour $ n \geq 2 $, considérons un élément $ F $ de $ \mathcal{G}^{(\varepsilon_{1}, \hdots,\varepsilon_{n-1})} $. Notons $ S_{\varepsilon_{n}}(F) $ l'ensemble des forêts construites à partir de $ F $ par les méthodes de construction ci-dessus, suivant la valeur de $ \varepsilon_{n} $. Alors,
\begin{eqnarray} \label{union}
\mathcal{G}^{(\varepsilon_{1}, \hdots,\varepsilon_{n})} = \bigcup_{F \in \mathcal{G}^{(\varepsilon_{1}, \hdots,\varepsilon_{n-1})}} S_{\varepsilon_{n}}(F) .
\end{eqnarray}

Le lemme suivant sera utile dans la suite:

\begin{lemma} \label{prelim} Soient $ \underline{\varepsilon}, \underline{\varepsilon}' \in \displaystyle\bigcup_{n \geq 1} \left\lbrace +,- \right\rbrace ^{n} $, avec $ \underline{\varepsilon} \neq \underline{\varepsilon}' $. Alors $ \mathcal{G}^{(+,\underline{\varepsilon})} \cap \mathcal{G}^{(+,\underline{\varepsilon}')} = \emptyset $.
\end{lemma}

\begin{proof}
Soient $ \underline{\varepsilon}, \underline{\varepsilon}' \in \cup_{n \geq 1} \left\lbrace +,- \right\rbrace ^{n} $, avec $ \underline{\varepsilon} \neq \underline{\varepsilon}' $. Tout d'abord, supposons que la suite $ \underline{\varepsilon} $ est de longueur $ n $ et que la suite $ \underline{\varepsilon}' $ est de longueur $ n' $, avec $ n \neq n' $. Comme les éléments de $ \mathcal{G}^{(+,\underline{\varepsilon})} $ sont de degré $ n+1 $ et ceux de $ \mathcal{G}^{(+,\underline{\varepsilon}')} $ sont de degré $ n'+1 $, $ \mathcal{G}^{(+,\underline{\varepsilon})} \cap \mathcal{G}^{(+,\underline{\varepsilon}')} = \emptyset $.\\

Supposons maintenant que les suites $ \underline{\varepsilon} $ et $ \underline{\varepsilon}' $ sont de même longueur $ n \geq 1 $ et raisonnons par récurrence sur $ n $. Le résultat est trivial pour $ n = 1 $. Supposons $ n \geq 2 $. On distingue alors deux cas:
\begin{enumerate}
\item Si $ \varepsilon_{n} \neq \varepsilon_{n}' $, par exemple $ \varepsilon_{n} = - $ et $ \varepsilon_{n}' = + $. Les éléments de $ \mathcal{G}^{(+,\underline{\varepsilon})} $ sont tous des arbres dont la racine est indexée par $ n+1 $. L'ensemble $ \mathcal{G}^{(+,\underline{\varepsilon}')} $ est constitué d'arbres et de forêts (de longueur $ \geq 2 $). Par construction, les arbres de $ \mathcal{G}^{(+,\underline{\varepsilon}')} $ sont de la forme $ B^{+}(T_{1}, \hdots , T_{m}) $, avec $ m \geq 1 $ et $ T_{1} \hdots T_{m} \in \mathcal{G}^{(+,\varepsilon_{1}', \ldots ,\varepsilon_{n-1}')} $. En particulier, la racine de ces arbres est indexée par un entier $ < n+1 $. Donc $ \mathcal{G}^{(+,\underline{\varepsilon})} \cap \mathcal{G}^{(+,\underline{\varepsilon}')} = \emptyset $.
\item Si $ \varepsilon_{n} = \varepsilon_{n}' $, comme $ \underline{\varepsilon} \neq \underline{\varepsilon}' $, $ \varepsilon_{1}, \ldots ,\varepsilon_{n-1} \neq \varepsilon_{1}', \ldots ,\varepsilon_{n-1}' $. Par l'absurde, supposons que $ \mathcal{G}^{(+,\underline{\varepsilon})} \cap \mathcal{G}^{(+,\underline{\varepsilon}')} \neq \emptyset $. Alors il existe une forêt $ T_{1} \hdots T_{m} \in \mathcal{G}^{(+,\varepsilon_{1}, \ldots ,\varepsilon_{n-1})} $ et une forêt $ T_{1}' \hdots T_{l}' \in \mathcal{G}^{(+,\varepsilon_{1}', \ldots ,\varepsilon_{n-1}')} $ telles que $ S_{\varepsilon_{n}}(T_{1} \hdots T_{m}) \cap S_{\varepsilon_{n}'}(T_{1}' \hdots T_{l}') \neq \emptyset $. Par hypothèse de récurrence, $ \mathcal{G}^{(+,\varepsilon_{1}, \ldots ,\varepsilon_{n-1})} \cap \mathcal{G}^{(+,\varepsilon_{1}', \ldots ,\varepsilon_{n-1}')} = \emptyset $, donc $ T_{1} \hdots T_{m} \neq T_{1}' \hdots T_{l}' $. Or
\begin{enumerate}
\item si $ \varepsilon_{n} = \varepsilon_{n}' = - $, $ S_{-}(T_{1} \hdots T_{m}) = \left\lbrace B^{-}(T_{1}, \hdots ,T_{m}) \right\rbrace $, $ S_{-}(T_{1}' \hdots T_{l}') = \left\lbrace B^{-}(T_{1}', \hdots ,T_{l}') \right\rbrace $. Donc $ B^{-}(T_{1}, \hdots ,T_{m}) = B^{-}(T_{1}', \hdots ,T_{l}') $ et, nécessairement, $ m = l $ et $ T_{1} \hdots T_{m} = T_{1}' \hdots T_{l}' $. On aboutit donc à une contradiction.
\item si $ \varepsilon_{n} = \varepsilon_{n}' = + $, alors
\begin{eqnarray*}
S_{+}(T_{1} \hdots T_{m}) & = & \{ B^{+}(T_{1}, \hdots ,T_{m}) , T_{1} B^{+}(T_{2}, \hdots ,T_{m}), \hdots , T_{1} \hdots T_{m-1} B^{+}(T_{m}),\\
&& T_{1} \hdots T_{m} \tun_{n+1} \} \\
S_{+}(T_{1}' \hdots T_{l}') & = & \{ B^{+}(T_{1}', \hdots ,T_{l}') , T_{1}' B^{+}(T_{2}', \hdots ,T_{l}'), \hdots , T_{1}' \hdots T_{l-1}' B^{+}(T_{l}'),\\
&& T_{1}' \hdots T_{l}' \tun_{n+1} \} .
\end{eqnarray*}
Alors il existe $ i \in \left\lbrace 1, \hdots m \right\rbrace $ et $ j \in \left\lbrace 1, \hdots l \right\rbrace $ tels que $ T_{1} \hdots T_{i-1} B^{+}(T_{i}, \hdots ,T_{m}) = T_{1}' \hdots T_{j-1}' B^{+}(T_{j}', \hdots, T_{l}') $. Nécessairement, on doit avoir $ i=j $, $ T_{1} = T_{1}', \hdots, T_{i-1}=T_{i-1}' $ et $ B^{+}(T_{i}, \hdots ,T_{m}) = B^{+}(T_{i}', \hdots, T_{l}') $. Or cette dernière égalité implique que $ m = l $ et $ T_{i} = T_{i}' , \hdots , T_{m} = T_{m}' $. Ici encore, cela contredit $ T_{1} \hdots T_{m} \neq T_{1}' \hdots T_{l}' $.
\end{enumerate}
\end{enumerate}
Par récurrence, le résultat est ainsi démontré.
\end{proof}

\vspace{0.5cm}

En reprenant la preuve précédente, remarquons que dans l'égalité (\ref{union}) l'union est disjointe. De plus, dans chaque ensemble $ S_{\varepsilon_{n}}(F) $, toutes les forêts sont distinctes car elles ne sont pas de même longueur. Ainsi, pour tout $ \underline{\varepsilon} \in \lbrace +,- \rbrace ^{n} $, avec $ n \geq 1 $, les éléments de $ \mathcal{G}^{(\underline{\varepsilon})} $ sont tous distincts.\\

Considérons les ensembles suivants :
$$ \mathcal{G}=\bigcup_{\underline{\varepsilon} \in \lbrace +,- \rbrace ^{n} , n \geq 1 } \mathcal{G}^{(\underline{\varepsilon})} {\rm ~ et ~} \mathcal{G}^{0}=\bigcup_{n \geq 1} \mathcal{G}^{(\overbrace{+, \hdots ,+}^{n {\rm ~ fois}})} .$$

D'après le lemme \ref{prelim}, les deux unions précédentes sont disjointes (à l'identification près des ensembles $ \mathcal{G}^{(+,\underline{\varepsilon})} $ et $ \mathcal{G}^{(-,\underline{\varepsilon})} $) et il n'y a donc pas de redondances dans la construction des forêts appartenant à $ \mathcal{G} $.\\

Remarquons que $ \mathcal{G} $ n'est pas stable pour l'opération de concaténation. Par exemple, les arbres $ \tddeux{1}{2} $ et $ \tddeux{2}{1} $ appartiennent à $ \mathcal{G} $ mais la forêt $ \tddeux{1}{2} \tddeux{2}{1} \notin \mathcal{G} $. Cela est dû au fait que l'arbre de droite qui compose la forêt $ \tddeux{1}{2} \tddeux{2}{1} $ a été construit avec un $ B^{-} $ ( $ \tddeux{2}{1} \in \mathcal{G}^{(+,-)} $). Plus précisément, on a le résultat suivant:

\begin{lemma} \label{1}
Les assertions suivantes sont équivalentes:
\begin{enumerate}
\item La forêt $ T_{1} \hdots T_{m} $ appartient à $ \mathcal{G} $.
\item $ T_{1} \in \mathcal{G} $ et $ T_{2}, \hdots , T_{m} \in \mathcal{G}^{0} $.
\end{enumerate}
En particulier, pour toute forêt $ T_{1} \hdots T_{m} \in \mathcal{G} $, $ T_{1}, \hdots , T_{m} $ appartiennent à $ \mathcal{G} $.
\end{lemma}

\begin{proof}
Démontrons tout d'abord le sens direct. Soit une forêt $ T_{1} \hdots T_{m} $ appartenant à $ \mathcal{G} $ et notons $ n = \left| T_{1} \right| + \hdots + \left| T_{m} \right| $ son degré. D'après le lemme \ref{prelim}, il existe un unique $ \underline{\varepsilon} \in \lbrace +,- \rbrace ^{n} $ tel que $ T_{1} \hdots T_{m} \in \mathcal{G}^{(\underline{\varepsilon})} $. Soit $ i $ le plus grand indice tel que $ \varepsilon_{i} = - $. Par construction, les éléments appartenant à $ \mathcal{G}^{(\varepsilon_{1}, \hdots , \varepsilon_{i})} $ sont des arbres de la forme $ B^{-} $ de l'arbre vide si $ i = 1 $ ou $ B^{-} $ d'une suite d'arbres $ G_{1}, \hdots ,G_{k} $ telle que $ G_{1} \hdots G_{k} \in \mathcal{G}^{(\varepsilon_{1}, \hdots , \varepsilon_{i-1})} $ si $ i \geq 2 $. $ T_{1} $ contient donc un sous-arbre $ B^{-}(G_{1}, \hdots ,G_{k}) $ ($ B^{-}(1) $ si $ i =1 $) appartenant à $ \mathcal{G}^{(\varepsilon_{1}, \hdots , \varepsilon_{i})} $. Autrement dit, $ \exists ~ l \geq 0 $ tel que
$$ T_{1} = \overbrace{B^{+}(\hdots B^{+}}^{l ~ {\rm fois}}(B^{-}(G_{1}, \hdots ,G_{k}), \hdots), \hdots ) .$$
Ainsi, $ \left| T_{1} \right| \geq i $. Par ailleurs, remarquons que si $ F_{1} F_{2} \in \mathcal{G}^{(\underline{\varepsilon})} $, alors $ F_{1} \in \mathcal{G}^{(\varepsilon_{1}, \hdots, \varepsilon_{\left|F_{1}\right|})} $ et $ F_{2} \in \mathcal{G}^{(\varepsilon_{\left|F_{1}\right|+1}, \hdots, \varepsilon_{\left|F_{1}\right| + \left|F_{2}\right|})} $. Ainsi, $ T_{1} \in \mathcal{G} $ et $ T_{2}, \hdots , T_{m} \in \mathcal{G}^{0} $.\\

Réciproquement, montrons que si $ T_{1} \in \mathcal{G} $ et $ T_{2}, \hdots , T_{m} \in \mathcal{G}^{0} $, alors $ T_{1} T_{2} \hdots T_{m} \in \mathcal{G} $ par récurrence sur $ \left| F \right| $, où $ F = T_{2} \hdots T_{m} $. Si $ \left| F \right| = 1 $, c'est-à-dire $ F = \tdun{1} $, alors $ T_{1} F = T_{1} {\begin{picture}(5,0)(-2,-1) \put(0,0){\circle*{2}} \end{picture}}_{\left|T_{1}\right|+1} $, et par construction de $ \mathcal{G} $, $ T_{1} F \in \mathcal{G} $. Soit $ n \geq 1 $ et supposons le résultat vérifié pour tout $ F $ tel que $ \left| F \right| \leq n $. Considérons $ F = T_{2} \hdots T_{m} $ de degré $ n+1 $, avec $ T_{2}, \hdots , T_{m} \in \mathcal{G}^{0} $. Si $ \left| T_{m} \right| = 1 $, par hypothèse de récurence $ T_{1} T_{2} \hdots T_{m-1} \in \mathcal{G} $ et comme pour l'initialisation $ T_{1} T_{2} \hdots T_{m-1} {\begin{picture}(5,0)(-2,-1) \put(0,0){\circle*{2}} \end{picture}}_{\left|T_{1}\right|+n+1} \in \mathcal{G} $. Sinon, comme $ T_{m} \in \mathcal{G}^{0} $, il existe $ G_{1}, \hdots , G_{k} \in \mathcal{G}^{0} $ tel que $ T_{m} = B^{+}(G_{1}, \hdots , G_{k}) $, avec $ G_{1} \hdots G_{k} $ une foret de degré $ \left| T_{m} \right| - 1 $. Par hypothèse de récurrence, $ T_{1} \hdots T_{m-1} G_{1} \hdots G_{k} \in \mathcal{G} $, donc $ T_{1} F = T_{1} T_{2} \hdots T_{m-1} B^{+}(G_{1}, \hdots , G_{k}) $ appartient bien à $ \mathcal{G} $.
\end{proof}
\\

D'après le lemme précédent, $ \mathcal{G} $ est stable par concaténation à gauche par des éléments de $ \mathcal{G}^{0} $. Pour $ \mathcal{G}^{0} $, on a le

\begin{lemma} \label{magma}
$ \mathcal{G}^{0} \cup \left\lbrace 1 \right\rbrace $ est un monoïde libre pour l'opération de concaténation.
\end{lemma}

\begin{proof} En effet, si $ T_{1} T_{2} \in \mathcal{G}^{0} $, $ T_{1} T_{2} \in \mathcal{G}^{\overbrace{(+,\hdots,+)}^{\left|T_{1}\right| + \left|T_{2}\right| ~ {\rm fois}}} $, donc $ T_{1} \in \mathcal{G}^{\overbrace{(+,\hdots,+)}^{\left|T_{1}\right| ~ {\rm fois}}} \subseteq \mathcal{G}^{0} $ et $ T_{2} \in \mathcal{G}^{\overbrace{(+,\hdots,+)}^{\left|T_{2}\right| ~ {\rm fois}}} \subseteq \mathcal{G}^{0} $.\\

Réciproquement, supposons que $ T_{1}, T_{2} \in \mathcal{G}^{0} $ et raisonnons par récurrence sur le degré de $ T_{2} $. Si $ \left|T_{2}\right| = 1 $, $ T_{2} = \tdun{1} $ et alors $ T_{1} T_{2} \in \mathcal{G}^{0} $ par construction de $ \mathcal{G}^{0} $. Supposons $ \left|T_{2}\right| \geq 2 $. Comme $ T_{2} \in \mathcal{G}^{0} $, il existe $ G_{1}, \hdots , G_{k} \in \mathcal{G}^{0} $ tels que $ T_{2} = B^{+}(G_{1}, \hdots ,G_{k}) $. Alors la forêt $ T_{1} G_{1} \hdots G_{k} \in \mathcal{G}^{0} $ par hypothèse de récurrence, et donc $ T_{1} T_{2} = T_{1} B^{+}(G_{1}, \hdots ,G_{k}) \in \mathcal{G}^{0} $, par construction de $ \mathcal{G}^{0} $.
\end{proof}
\\

{\bf Remarque.} {Posons $ \Delta_{\boldsymbol{l}} (F) = \displaystyle\sum_{\boldsymbol{v} \models V(F) {\rm ~ et ~} roo_{\boldsymbol{l}}(F) \in Roo_{\boldsymbol{v}} (F)} Lea_{\boldsymbol{v}} (F) \otimes Roo_{\boldsymbol{v}} (F) $ pour toute forêt non vide $ F $, où $ roo_{\boldsymbol{l}}(F) $ est la racine de l'arbre le plus à gauche de la forêt $ F $. Soient $ T_{1}, \hdots, T_{m} $ $ m $ arbres non vides, $ m \geq 1 $. Alors, en étudiant les coupes admissibles:
\begin{eqnarray*}
\Delta (B^{-}(T_{1}, \hdots ,T_{m})) & = & (Id \otimes B^{-}) \circ \Delta (T_{1} \hdots T_{m}) + B^{-}(T_{1} \hdots T_{m}) \otimes 1 ,\\
\Delta (B^{+}(T_{1}, \hdots ,T_{m})) & = & (Id \otimes B^{+}) \circ \Delta_{\boldsymbol{l}} (T_{1} \hdots T_{m}) + B^{+} (T_{1}, \hdots ,T_{m}) \otimes 1\\
& \hspace{0.5 cm} & + \Delta_{\boldsymbol{l}} (T_{1}) \cdot (B^{-}(T_{2}, \hdots, T_{m}) \otimes 1).
\end{eqnarray*}
}

\vspace{-0.3cm}

\begin{lemma} \label{comod}
Soit $ T $ un arbre appartenant à $ \mathcal{G}^{0} \cup \left\lbrace 1 \right\rbrace $. Alors, pour toute coupe admissible $ \boldsymbol{v} \models V(T) $, $ Roo_{\boldsymbol{v}} (T) \in \mathcal{G}^{0} \cup \left\lbrace 1 \right\rbrace $.
\end{lemma}

\begin{proof}
Il suffit de montrer que, pour toute coupe simple $ \boldsymbol{v} \models V(T) $ et pour tout arbre $ T \in \mathcal{G}^{0} \cup \left\lbrace 1 \right\rbrace $, $ Roo_{\boldsymbol{v}} (T) \in \mathcal{G}^{0} \cup \left\lbrace 1 \right\rbrace $. On raisonne par récurrence sur le degré $ n $ de $ T \in \mathcal{G}^{0} \cup \left\lbrace 1 \right\rbrace $, le résultat étant trivial si $ n=0,1,2 $.\\

Supposons $ n \geq 3 $. Comme $ T \in \mathcal{G}^{0} $, il existe $ m \geq 1 $, $ T_{1}, \hdots , T_{m} \in \mathcal{G}^{0} $, tels que $ T = B^{+}(T_{1}, \hdots , T_{m}) $. Soit $ \boldsymbol{v} \models V(T) $ une coupe simple de $ T $. Il y a trois cas possibles:
\begin{enumerate}
\item Si $ \boldsymbol{v} \models V(T_{1}) $, par hypothèse de récurrence, $ Roo_{\boldsymbol{v}} (T_{1}) \in \mathcal{G}^{0} \cup \left\lbrace 1 \right\rbrace $. Si $ Roo_{\boldsymbol{v}} (T_{1}) = 1 $, $ Roo_{\boldsymbol{v}} (T) = 1 $. Si $ Roo_{\boldsymbol{v}} (T_{1}) \neq 1 $, alors, avec le lemme \ref{magma}, $ Roo_{\boldsymbol{v}} (T_{1}) T_{2} \hdots T_{m} \in \mathcal{G}^{0} $, donc $ Roo_{\boldsymbol{v}} (T) = B^{+}(Roo_{\boldsymbol{v}} (T_{1}) ,T_{2}, \hdots ,T_{m}) \in \mathcal{G}^{0} $.
\item Si $ \boldsymbol{v} \models V(T_{i}) $, avec $ i \geq 2 $ (on inclut ici le cas de la coupe totale qui correspond à couper l'arête entre le sommet indexé par $ n $ et la racine de $ T_{i} $). Par hypothèse de récurrence, $ Roo_{\boldsymbol{v}} (T_{i}) \in \mathcal{G}^{0} \cup \left\lbrace 1 \right\rbrace $. Avec le lemme \ref{magma}, $ T_{1} \hdots Roo_{\boldsymbol{v}} (T_{i}) \hdots T_{m} \in \mathcal{G}^{0} $, et ainsi $ Roo_{\boldsymbol{v}} (T) = B^{+}(T_{1}, \hdots ,Roo_{\boldsymbol{v}} (T_{i}), \hdots, T_{m}) \in \mathcal{G}^{0} $.
\item Enfin, si on coupe l'arête joignant la racine de $ T $ (qui est aussi la racine de $ T_{1} $) et le sommet indexé par $ n $, alors $ Roo_{\boldsymbol{v}} (T) = T_{1} \in \mathcal{G}^{0} $.
\end{enumerate}
Ainsi, dans tous les cas, $ Roo_{\boldsymbol{v}} (T) \in \mathcal{G}^{0} $, et on peut conclure par le principe de récurrence.
\end{proof}
\\

{\bf Remarque.} {Par contre, étant donné un arbre $ T $ appartenant à $ \mathcal{G}^{0} $, il existe certaines coupes $ \boldsymbol{v} \models V(T) $ telles que $ Lea_{\boldsymbol{v}} (T) \not\in \mathcal{G}^{0} \cup \left\lbrace 1 \right\rbrace $. Par exemple, considérons $ T_{1}, \hdots , T_{m} \in \mathcal{G}^{0} $, avec $ m \geq 2 $, et $ T = B^{+}(T_{1}, \hdots, T_{m}) \in \mathcal{G}^{0} $. Alors, si on réalise la coupe simple $ \boldsymbol{v} \models V(T) $ consistant à couper l'arête joignant la racine de $ T $ (qui est aussi la racine de $ T_{1} $) et le sommet indexé par $ \left| T \right| $, $ Lea_{\boldsymbol{v}}(T) = B^{-}(T_{2}, \hdots , T_{m}) \in \mathcal{G}^{(\hdots , -)} $ et donc $ \not\in \mathcal{G}^{0} $ en utilisant le lemme \ref{prelim}.
}
\\

Rappelons (voir \cite{Connes}) que l'on peut définir sur $ \mathcal{H}_{PR} $ un opérateur de greffe $ B : \mathcal{H}_{PR} \rightarrow \mathcal{H}_{PR} $ qui, étant donnée une suite de $ m $ arbres plans non vides $ T_{1}, \hdots ,T_{m} $, associe l'arbre $ B(T_{1}, \hdots ,T_{m}) $ obtenu en considérant $ T_{1}, \hdots ,T_{m} $ comme la suite des sous-arbres ayant une racine commune (avec la convention $ B(1) = \tun $). Par exemple, $ B(\tun, \tdeux) = \tquatretrois $.\\

Soit $ T $ un arbre enraciné plan non vide $ \in \mathcal{H}_{PR} $. Une question naturelle est de savoir de combien de façons on peut indexer les sommets de $ T $ pour obtenir un élément de $ \mathcal{G}^{0} $ ou de $ \mathcal{G} $.\\

Dans le cas de $ \mathcal{G}^{0} $, il n'y a qu'une seule et unique indexation possible. En effet, considérons un arbre plan $ T $. Si $ T $ est de degré 1, le résultat est trivial. Sinon, $ T = B(T_{1}, \hdots, T_{m}) $ avec $ T_{1}, \hdots , T_{m} $ $ m $ arbres non vides $ \in \mathcal{H}_{PR} $ et $ m \geq 1 $. Par hypothèse de récurrence, il y a une unique indexation de $ B(T_{1}, \hdots, T_{m-1}) $ en un élément noté $ G $ de $ \mathcal{G}^{0} $. Si $ T_{m} = \tun $, comme le descendant le plus à droite de la racine de $ T $ doit nécessairement être indexé par $ \left| T \right| $, $ B^{+}(G) $ est l'unique arbre de $ \mathcal{G}^{0} $ tel que l'arbre plan associé (en supprimant l'indexation) soit $ T $. Sinon, $ T_{m} = B(T_{m,1},\hdots,T_{m,k}) $. Avec l'hypothèse de récurrence, $ \forall ~ 1 \leq i \leq k $, l'arbre plan $ T_{m,i} $ a une unique indexation en un élément noté $ G_{i} $ de $ \mathcal{G}^{0} $. Alors, $ B^{+}(G,G_{1}, \hdots,G_{k}) \in \mathcal{G}^{0} $, et c'est par construction l'unique arbre de $ \mathcal{G}^{0} $ tel que l'arbre plan associé soit $ T $. On définit ainsi une bijection entre les arbres de $ \mathcal{H}_{PR} $ et les arbres de $ \mathcal{G}^{0} $.\\

Pour le cas plus complexe de $ \mathcal{G} $, on a la

\begin{prop} \label{2}
Soit un arbre plan $ T \in \mathcal{H}_{PR} $. Rappelons que, si $ v $ est un sommet de $ T $, $ f_{v} $ désigne la fertilité de $ v $ et $ h(v) $ la hauteur de $ v $. Notons $ l_{T} $ la feuille la plus à gauche de $ T $. Alors, il y a
$$ 1+\displaystyle\sum^{h(l_{T})-1}_{i=0} \prod_{l_{T} \twoheadrightarrow v ~ {\rm et} ~ h(v) \leq i} f_{v} $$
façons d'indexer $ T $ pour obtenir un élément de $ \mathcal{G} $.
\end{prop}

\begin{proof}
Rappelons que si $ T $ est un arbre non vide de $ \mathcal{G} $, $ B^{+}(T,1) = B^{+}(T) $ est l'arbre construit en greffant le sommet indexé par $ \left| T \right| +1 $ comme le fils le plus à droite de la racine de $ T $. En particulier, on ne cherchera pas dans cette preuve à simplifier un produit d'arbres en supprimant les éventuels arbres vides.\\

Raisonnons par récurrence sur $ k = h(l_{T}) $ la hauteur de $ l_{T} $. Pour éviter d'alourdir les notations, on notera indifférement un arbre appartenant à $ \mathcal{H}_{PR} $ et l'arbre (ou les arbres) obtenu après indexation des sommets appartenant à $ \mathcal{G} $.\\

Si $ k=0 $, $ T $ est l'arbre plan réduit à sa racine, il y a donc une unique façon de le numéroter, en $ \tdun{1} $. Ceci démontre la formule au rang $ 0 $.\\

Si $ k=1 $, $ T = B(\tun, T_{1}, \hdots , T_{n}) $, avec $ T_{1}=B(T_{1,1}, \hdots T_{1,m_{1}}), \hdots , T_{n}=B(T_{n,1}, \hdots , T_{n,m_{n}}) $ (si $ T_{i} = \tun $, $ T_{i} = B(T_{i,1}) $ avec $ T_{i,1}=1 $). Il y a alors deux cas possibles pour l'indexation de $ T $:
\begin{enumerate}
\item Si l'indice de la racine de $ T $ est plus petit que l'indice de son descendant direct le plus à gauche. Alors,
$$ T=B^{+}(\hdots B^{+}(B^{+}(\tdun{1}, 1), T_{1,1} , \hdots , T_{1,m_{1}}), \hdots ), T_{n,1} , \hdots, T_{n,m_{n}}) ,$$
avec, d'après le lemme \ref{1}, $ T_{1,1}, \hdots , T_{n,m_{n}} \in \mathcal{G}^{0} $, et donc une unique façon d'indexer les sommets de $ T_{1,1}, \hdots , T_{n,m_{n}} $. Ainsi $ T $ appartient à $ \mathcal{G}^{0} \subseteq \mathcal{G} $ et il y a, dans ce cas, une seule façon de numéroter les sommets de $ T $.
\item Si l'indice de la racine de $ T $ est plus grand que l'indice de son descendant direct le plus à gauche. Alors, pour tout $ 1 \leq i \leq n $,
$$ \overbrace{B^{+}(\hdots B^{+}}^{n-i ~ {\rm fois}}(B^{-}(\tdun{1}, T_{1} , \hdots , T_{i}), T_{i+1,1} , \hdots , T_{i+1,m_{i+1}}), \hdots ), T_{n,1} , \hdots, T_{n,m_{n}}) \in \mathcal{G} .$$
où, toujours avec le lemme \ref{1}, on doit avoir $ T_{1}, \hdots , T_{i}, T_{i+1,1}, \hdots , T_{n,m_{n}} \in \mathcal{G}^{0} $ et il y a donc une unique façon d'indexer leurs sommets. Ainsi, cela fait $ n $ façons de numéroter les sommets de $ T $ dans ce cas.
\end{enumerate}
Au final, il y a $ 1+n $ façons d'indexer $ T $ pour obtenir un élément de $ \mathcal{G} $. Le cas $ k=1 $ est démontré car $ n $ est égal à la fertilité de la racine de $ T $.\\

Si $ k \geq 2 $, $ T = B(T_{1}, \hdots , T_{n}) $, avec $ T_{1}=B(T_{1,1}, \hdots T_{1,m_{1}}), \hdots , T_{n}=B(T_{n,1}, \hdots , T_{n,m_{n}}) $ (si $ T_{i} = \tun $, $ T_{i} = B(T_{i,1}) $ avec $ T_{i,1}=1 $). Il y a deux cas possibles pour indexer les sommets de $ T $:
\begin{enumerate}
\item Si l'indice de la racine de $ T $ est plus petit que l'indice de son descendant direct le plus à gauche. Alors, pour que $ T \in \mathcal{G} $,
$$ T=B^{+}(\hdots B^{+}(B^{+}(\tdun{1}, T_{1,1} , \hdots , T_{1,m_{1}}), \hdots ), T_{n,1} , \hdots, T_{n,m_{n}}) .$$
Avec le lemme \ref{1}, on doit avoir $ T_{1,1}, \hdots , T_{n,m_{n}} \in \mathcal{G}^{0} $, et il y a une unique façon d'indexer leurs sommets. On a ainsi une unique façon de numéroter $ T $ dans ce cas.
\item Si l'indice de la racine de $ T $ est plus grand que l'indice de son descendant direct le plus à gauche. Deux sous-cas sont possibles:
\begin{enumerate}
\item Si l'indice de la racine de $ T_{1} $ est plus petit que l'indice de son descendant direct le plus à gauche, c'est-à-dire si $ T_{1} = B^{+}(\hdots B^{+}(\tdun{1}, \hdots), \hdots) $, alors $ T_{1} \in \mathcal{G}^{0} $ (avec le lemme \ref{1}), et il y a une seule possibilité d'indexer $ T_{1} $. Pour tout $ 1 \leq i \leq n $,
$$ \overbrace{B^{+}(\hdots B^{+}}^{n-i ~ {\rm fois}}(B^{-}(T_{1} , \hdots , T_{i}), T_{i+1,1} , \hdots , T_{i+1,m_{i+1}}), \hdots ), T_{n,1} , \hdots, T_{n,m_{n}}) \in \mathcal{G} .$$
D'après le lemme \ref{1}, $ T_{2}, \hdots , T_{i}, T_{i+1,1}, \hdots , T_{n,m_{n}} \in \mathcal{G}^{0} $, donc il y a une unique façon de les indexer. On a ainsi $ n $ possibilités pour indexer $ T $ dans ce cas.
\item Si l'indice de la racine de $ T_{1} $ est plus grande que l'indice de son descendant direct le plus à gauche, c'est-à-dire si
$$ T_{1} = \overbrace{B^{+}(\hdots B^{+}}^{m_{1}-i ~ {\rm fois}}(B^{-}(T_{1,1} , \hdots , T_{1,i}), \hdots ), \hdots ) ,$$
pour un $ 1 \leq i \leq m_{1} $. Par hypothèse de récurrence, il y a $ 1+\displaystyle\sum^{k-1}_{i=2} \prod_{l_{T} \twoheadrightarrow v ~ {\rm et} ~ 2 \leq h(v) \leq i} f_{v} $ façons de numéroter $ T_{1,1} $ pour obtenir un élément de $ \mathcal{G} $. Toujours avec le lemme  \ref{1}, il y a une unique façon de numéroter $ T_{1,2}, \hdots T_{1,m_{1}} $ car on doit nécessairement avoir après indexation $ T_{1,2}, \hdots T_{1,m_{1}} \in \mathcal{G}^{0} $. Cela donne $ m_{1}  \left( 1+\displaystyle\sum^{k-1}_{i=2} \prod_{l_{T} \twoheadrightarrow v ~ {\rm et} ~ 2 \leq h(v) \leq i} f_{v} \right) $ possibilités pour indexer $ T_{1} $. Comme l'indice de la racine de $ T $ doit être plus grande que l'indice de son descendant directe le plus à gauche, pour que $ T \in \mathcal{G} $, on doit avoir
$$ T = B^{+}( \hdots B^{+}(B^{-}(T_{1}, \hdots, T_{i}),T_{i+1,1}, \hdots, T_{i+1,m_{i+1}}), \hdots ), T_{n,1} , \hdots, T_{n,m_{n}}) ,$$
où $ 1 \leq i \leq n $. D'après le lemme \ref{1}, $ T_{2}, \hdots , T_{i}, T_{i+1,1}, \hdots , T_{n,m_{n}} \in \mathcal{G}^{0} $ ont une unique indexation possible. Ainsi, dans ce cas, il y a $ n m_{1} \left( 1+\displaystyle\sum^{k-1}_{i=2} \prod_{l_{T} \twoheadrightarrow v ~ {\rm et} ~ 2 \leq h(v) \leq i} f_{v} \right) $ possibilités d'indexer $ T $.
\end{enumerate}
\end{enumerate}
En tout, cela fait bien
$$ 1+\displaystyle\sum^{k-1}_{i=0} \prod_{l_{T} \twoheadrightarrow v ~ {\rm et} ~ h(v) \leq i} f_{v} $$
possibilités pour indexer les sommets de $ T $. Le principe de récurrence permet de conclure.
\end{proof}
\\

Dans le cas des échelles, il existe un résultat plus précis :

\begin{cor}
Pour tout $ n \geq 1 $ et pour tout $ i \in \left\lbrace 1, \hdots , n \right\rbrace  $, il existe une unique échelle (de degré $ n $) dans chaque sous-ensemble $ \mathcal{G}^{(\overbrace{+, \hdots ,+}^{i {\rm ~ fois}} ,\overbrace{-, \hdots ,-}^{n-i {\rm ~ fois}})} $.
\end{cor}

\begin{proof}
Raisonnons par récurrence sur $ n $. Le résultat étant trivial pour $ n=1,2 $, supposons $ n \geq 3 $. D'après l'hypothèse de récurrence, pour tout $ i \in \left\lbrace 1, \hdots , n-1 \right\rbrace  $, il existe une unique échelle $ T_{i} $ de degré $ n-1 $ dans chaque sous-ensemble $ \mathcal{G}^{(\overbrace{+, \hdots ,+}^{i {\rm ~ fois}} ,\overbrace{-, \hdots ,-}^{n-1-i {\rm ~ fois}})} $. En réalisant un $ B^{-} $, on construit $ n-1 $ échelles $ E_{i} = B^{-}(T_{i}) $ de degré $ n $, et $ \forall ~ i \in \left\lbrace 1, \hdots , n-1 \right\rbrace $, $ E_{i} \in \mathcal{G}^{(\overbrace{+, \hdots ,+}^{i {\rm ~ fois}} ,\overbrace{-, \hdots ,-}^{n-i {\rm ~ fois}})} $. De plus, toujours par hypothèse de récurrence, il existe une unique échelle $ T $ de degré $ n-2 $ appartenant à $ \mathcal{G}^{0} $. Alors, $ \tdun{1} T \in \mathcal{G}^{0} $, en utilisant le lemme \ref{magma}, et donc $ E_{n} = B^{+}(\tdun{1},T) \in \mathcal{G}^{0} $ et c'est par construction une échelle.\\

Donc $ \forall ~ i \in \left\lbrace 1, \hdots , n \right\rbrace $, $ E_{i} \in \mathcal{G}^{(\overbrace{+, \hdots ,+}^{i {\rm ~ fois}} ,\overbrace{-, \hdots ,-}^{n-i {\rm ~ fois}})} $. D'après le lemme \ref{prelim}, les échelles $ E_{i} $ sont toutes distinctes. Comme il existe exactement $ n $ échelles de degré $ n $ dans $ \mathcal{G} $ (d'après la proposition \ref{2}), il y a donc une unique échelle dans chaque sous-ensemble $ \mathcal{G}^{(\overbrace{+, \hdots ,+}^{i {\rm ~ fois}} ,\overbrace{-, \hdots ,-}^{n-i {\rm ~ fois}})} $, pour $ i \in \left\lbrace 1, \hdots , n \right\rbrace  $. Le résultat est ainsi démontré au rang $ n $ et on conclut par le principe de récurrence.
\end{proof}
\\

\subsection{L'algèbre de Hopf $ \mathcal{B}^{\infty} $}

Notons $ \mathcal{B}^{\infty} = \mathbb{K} [\mathcal{G} \cup \left\lbrace 1 \right\rbrace] $ l'algèbre engendrée par $ \mathcal{G} \cup \left\lbrace 1 \right\rbrace $. D'après le lemme \ref{1}, $ \mathcal{B}^{\infty} $ est engendrée librement par les arbres appartenant à $ \mathcal{G} $.\\

Nous avons le résultat remarquable suivant:

\begin{prop}
L'algèbre $ \mathcal{B}^{\infty} $ est une algèbre de Hopf.
\end{prop}

\begin{proof}
Il suffit de montrer que, en réalisant une coupe simple d'un arbre appartenant à $ \mathcal{G} $, la branche et le tronc appartiennent respectivement à $ \mathcal{G} $ et $ \mathcal{G} \cup \left\lbrace 1 \right\rbrace $. En effet, si ce résultat est démontré, on aura alors le résultat pour une coupe admissible quelconque puisque $ \mathcal{B}^{\infty} $ est engendrée par $ \mathcal{G} \cup \left\lbrace 1 \right\rbrace $ comme algèbre. Travaillons par récurrence sur le degré des arbres. Le résultat est trivial pour $ n=2,3 $. Au rang $ n \geq 4 $, considérons un arbre $ T \in \mathcal{G} $ de degré $ n $, et $ \boldsymbol{v} \models V(T) $ une coupe simple. Il y a deux cas possibles:
\begin{enumerate}
\item Si l'arbre est de la forme $ T = B^{-}(T_{1},\hdots,T_{m}) \in \mathcal{G}^{(\hdots,-)} $. Par construction, la forêt $ T_{1} \hdots T_{m} \in \mathcal{G} $ et, avec le lemme \ref{1}, $ T_{1} \in \mathcal{G} $ et $ T_{2}, \hdots ,T_{m} \in \mathcal{G}^{0} $. Si $ \boldsymbol{v} $ est la coupe totale, le résultat est trivial. Sinon, comme $ \boldsymbol{v} $ est une coupe simple de $ T $, il existe un unique $ i \in \left\lbrace 1,\hdots,m \right\rbrace $ tel que $ \boldsymbol{v} \models V(T_{i}) $ (on inclut ici le cas de la coupe totale qui correspond à couper l'arête entre la racine de $ T $ et celle de $ T_{i} $). Par récurrence, $ Lea_{\boldsymbol{v}}(T_{i}) $ appartient à $ \mathcal{G} $, car $ T_{i} \in \mathcal{G} $. De même, par récurrence, $ Roo_{\boldsymbol{v}}(T_{i}) $ appartient à $ \mathcal{G} \cup \left\lbrace 1 \right\rbrace $. Alors la forêt $ T_{1} \hdots Roo_{\boldsymbol{v}}(T_{i}) \hdots T_{m} \in \mathcal{G} \cup \left\lbrace 1 \right\rbrace $ car :
\begin{enumerate}
\item si $ i=1 $, $ Roo_{\boldsymbol{v}}(T_{1}) \in \mathcal{G} \cup \left\lbrace 1 \right\rbrace $ et comme $ T_{2}, \hdots ,T_{m} \in \mathcal{G}^{0} $, $ Roo_{\boldsymbol{v}}(T_{1}) T_{2} \hdots T_{m} \in \mathcal{G} \cup \{ 1 \} $ en utilisant le lemme \ref{1}.
\item si $ i \geq 2 $, $ T_{i} \in \mathcal{G}^{0} $ donc, d'après le lemme \ref{comod}, $ Roo_{\boldsymbol{v}}(T_{i}) $ appartient à $ \mathcal{G}^{0} \cup \left\lbrace 1 \right\rbrace $. Ainsi, toujours avec le lemme \ref{1}, la forêt $ T_{1} \hdots Roo_{\boldsymbol{v}}(T_{i}) \hdots T_{m} $ appartient à $ \mathcal{G} $.
\end{enumerate}
Donc $ Lea_{\boldsymbol{v}}(T) = Lea_{\boldsymbol{v}}(T_{i}) \in \mathcal{G} $ et $ Roo_{\boldsymbol{v}}(T) = B^{-}(T_{1}, \hdots ,Roo_{\boldsymbol{v}}(T_{i}), \hdots ,T_{m}) \in \mathcal{G} \cup \left\lbrace 1 \right\rbrace $.
\item Si l'arbre est de la forme $ T = B^{+}(T_{1},\hdots,T_{m}) \in \mathcal{G}^{(\hdots,+)} $. Par construction, la forêt $ T_{1} \hdots T_{m} \in \mathcal{G} $, donc $ T_{1} \in \mathcal{G} $ et $ T_{2}, \hdots ,T_{m} \in \mathcal{G}^{0} $. Si $ \boldsymbol{v} $ est la coupe simple correspondant à couper l'arête joignant la racine de $ T $ (qui est la racine de $ T_{1} $) et le sommet indexé par $ n $ joignant les racines communes de $ T_{2}, \hdots , T_{m} $, alors $ Roo_{\boldsymbol{v}}(T) = T_{1} \in \mathcal{G} $ et $ Lea_{\boldsymbol{v}}(T) = B^{-}(T_{2}, \hdots ,T_{m}) \in \mathcal{G} $. Le résultat est donc vérifié dans ce cas. Sinon, comme $ \boldsymbol{v} $ est une coupe simple de $ T $, il existe un unique $ i \in \left\lbrace 1,\hdots,m \right\rbrace $ tel que $ \boldsymbol{v} \models V(T_{i}) $ (si $ i \geq 2 $, le cas de la coupe totale correspond à couper l'arête entre le sommet de $ T $ indexé par $ n $ et la racine de $ T_{i} $). Il y a alors deux cas à distinguer:
\begin{enumerate}
\item Si $ i=1 $, c'est-à-dire si $ \boldsymbol{v} \models V(T_{1}) $. Si $ \boldsymbol{v} $ est totale, alors $ Lea_{\boldsymbol{v}}(T) = T $ et $ Roo_{\boldsymbol{v}}(T) = 1 $ et le résultat est trivial. Sinon, par récurrence, $ Lea_{\boldsymbol{v}}(T_{1}) \in \mathcal{G} $, et $ Roo_{\boldsymbol{v}}(T_{1}) $ étant un arbre non vide $ Roo_{\boldsymbol{v}}(T_{1}) \in \mathcal{G} $. Ainsi, avec le lemme \ref{1}, $ Roo_{\boldsymbol{v}}(T_{1}) T_{2} \hdots T_{m} \in \mathcal{G} $. D'où $ Roo_{\boldsymbol{v}}(T) = B^{+}(Roo_{\boldsymbol{v}}(T_{1}), T_{2}, \hdots ,T_{m}) $ et $ Lea_{\boldsymbol{v}}(T) = Lea_{\boldsymbol{v}}(T_{1}) $ appartiennent à $ \mathcal{G} $.
\item Si $ i \geq 2 $, c'est-à-dire si $ \boldsymbol{v} \models V(T_{i}) $, toujours par récurrence, $ Lea_{\boldsymbol{v}}(T_{i}) \in \mathcal{G} $ et avec le lemme \ref{comod}, $ Roo_{\boldsymbol{v}}(T_{i}) \in \mathcal{G}^{0} \cup \left\lbrace 1 \right\rbrace $. Donc la forêt $ T_{1} \hdots Roo_{\boldsymbol{v}}(T_{i}) \hdots T_{m} \in \mathcal{G} $. Ainsi $ Roo_{\boldsymbol{v}}(T) = B^{+}(T_{1}, \hdots ,Roo_{\boldsymbol{v}}(T_{i}), \hdots ,T_{m}) $ et $ Lea_{\boldsymbol{v}}(T) = Lea_{\boldsymbol{v}}(T_{i}) $ sont des éléments de $ \mathcal{G} $.
\end{enumerate}
Dans tous les cas, $ Roo_{\boldsymbol{v}}(T) \in \mathcal{G} \cup \left\lbrace 1 \right\rbrace $, $ Lea_{\boldsymbol{v}}(T) \in \mathcal{G} $.
\end{enumerate}
Par récurrence, le résultat est démontré.
\end{proof}

\begin{prop}
La série formelle de l'algèbre de Hopf $ \mathcal{B}^{\infty} $ est donnée par la formule:
$$ F_{\mathcal{B}^{\infty}}(x) = \dfrac{1}{2 \sqrt{1-4x}} + \dfrac{1}{2} .$$
\end{prop}

\begin{proof}
Pour calculer la série formelle de l'algèbre $ \mathcal{B}^{\infty} $, nous introduisons quelques notations. Posons $ f_{i,j}^{\mathcal{B}^{\infty}} $ le nombre de forêts de longueur $ i $ et de degré $ j $, et $ f_{j}^{\mathcal{B}^{\infty}} $ le nombre de forêts de degré $ j $. En particulier, $ f_{j}^{\mathcal{B}^{\infty}}=\displaystyle\sum_{1 \leq i \leq j} f_{i,j}^{\mathcal{B}^{\infty}} $. Par construction de $ \mathcal{B}^{\infty} $, on a les relations suivantes:

\begin{eqnarray*}
f_{1,1}^{\mathcal{B}^{\infty}} & = & 1 \\
f_{k,1}^{\mathcal{B}^{\infty}} & = & 0 ~ {\rm si} ~ k \geq 2 \\
f_{1,n}^{\mathcal{B}^{\infty}} & = & 2 f_{n-1}^{\mathcal{B}^{\infty}} \\
f_{k,n}^{\mathcal{B}^{\infty}} & = & f_{k-1,n-1}^{\mathcal{B}^{\infty}} + \hdots + f_{n-1,n-1}^{\mathcal{B}^{\infty}}  ~ {\rm si} ~ k \geq 2 , n \geq 2. 
\end{eqnarray*}

Posons $ F(x)=\displaystyle\sum_{i \geq 1} f_{i}^{\mathcal{B}^{\infty}} x^{i} $ et, pour $ k \geq 1 $, $ F_{k}(x) = \displaystyle\sum_{i \geq 1} f_{k,i}^{\mathcal{B}^{\infty}} x^{i} $. Comme $ f_{k,i}^{\mathcal{B}^{\infty}} = 0 $ si $ i < k $, $ F_{k}(x) = \displaystyle\sum_{i \geq k} f_{k,i}^{\mathcal{B}^{\infty}} x^{i} $. Alors,

\begin{eqnarray*}
F(x) & = & \sum_{k \geq 1} F_{k}(x) \\
F_{1}(x) & = & 2xF(x)+x \\
F_{k}(x) & = & x \left( \sum_{i \geq k-1} F_{i}(x) \right) = x \bigg( F(x) -F_{1}(x)- \hdots - F_{k-2}(x) \bigg)   ~ {\rm si} ~  k \geq 2.
\end{eqnarray*}

Par différence, on obtient pour tout $ l \geq 1 $:
$$ F_{l+2}-F_{l+1}+xF_{l}=0 .$$
Il existe donc $ A(x),B(x) \in \mathbb{C} [[x,x^{-1}]] $ tels que $ \forall l \geq 1 $,
$$ F_{l}(x)=A(x) \left( \dfrac{1-\sqrt{1-4x}}{2} \right) ^{l} + B(x)  \left( \dfrac{1+\sqrt{1-4x}}{2} \right) ^{l} ,$$
d'où
\begin{equation}\label{S1}\begin{array}{rcl}
F(x) & =& A(x) \dfrac{1-2x-\sqrt{1-4x}}{2} + B(x) \dfrac{1-2x+\sqrt{1-4x}}{2}.
\end{array} \end{equation}

En faisant $ l=1,2 $, nous avons les deux relations suivantes:
\begin{equation}\label{S2}\left\{\begin{array}{rcl}
A(x) & = & \dfrac{1+\sqrt{1-4x}}{2} F(x) + \dfrac{1}{2}+\dfrac{1-2x}{2 \sqrt{1-4x}} \\
B(x) & = & \dfrac{1-\sqrt{1-4x}}{2} F(x) + \dfrac{1}{2}+\dfrac{1-2x}{2 \sqrt{1-4x}}
\end{array}\right. \end{equation}

Montrons que $ B(x) = 0 $. Si $ B(x) \neq 0 $, $ B(x) = a_{k} x^{k}+ \hdots $, $ a_{k} \neq 0 $. De plus,
$$ F_{l}(x) = A(x) \left( \underbrace{\dfrac{1-\sqrt{1-4x}}{2}}_{=x+ \hdots} \right) ^{l} + B(x)  \left( \underbrace{\dfrac{1+\sqrt{1-4x}}{2}}_{=1+ \hdots} \right) ^{l} ,$$
donc, si $ l>k $,
\begin{eqnarray*}
A(x) \left( \dfrac{1-\sqrt{1-4x}}{2} \right) ^{l} & = & \mathcal{O}(x^{l}) \\
B(x)  \left( \dfrac{1+\sqrt{1-4x}}{2} \right) ^{l} & = & a_{k} x^{k} + \hdots
\end{eqnarray*}
D'où $ F_{l}(x) = a_{k} x^{k} + \hdots $. Or $ F_{l}(x) = \displaystyle\sum_{i \geq l} f_{l,i}^{\mathcal{B}^{\infty}} x^{i} $, donc $ F_{l}(x)=\mathcal{O}(x^{l}) $, et $ a_{k}=0 $.\\

Ainsi $ B(x)=0 $, et avec (\ref{S1}) et (\ref{S2}),
\begin{eqnarray*}
F(x) & = & \dfrac{1-4x-\sqrt{1-4x}}{2(4x-1)} = \dfrac{1}{2 \sqrt{1-4x}} - \dfrac{1}{2} \\
A(x) & = & \dfrac{1}{2}+\dfrac{1}{2\sqrt{1-4x}}
\end{eqnarray*}
Au passage, on obtient la formule pour les $ F_{k} $:
$$ F_{k}(x) = \dfrac{x}{\sqrt{1-4x}} \left( \dfrac{1-\sqrt{1-4x}}{2} \right)^{k-1} .$$
Finalement la série formelle de $ \mathcal{B}^{\infty} $ est donnée par:
$$ F_{\mathcal{B}^{\infty}}(x)=\dfrac{1}{2 \sqrt{1-4x}} + \dfrac{1}{2} .$$
\end{proof}
\\

Ainsi, pour tout $ k \geq 1 $, $ f_{1,k}^{\mathcal{B}^{\infty}} = \dfrac{(2k-2)!}{(k-1)! (k-1)!} $ et $ f_{k}^{\mathcal{B}^{\infty}} = \dfrac{(2k)!}{2(k!)^{2}} $. Voici quelques valeurs numériques:

$$\begin{array}{c|c|c|c|c|c|c|c|c}
k&1&2&3&4&5&6&7&8\\
\hline f_{1,k}^{\mathcal{B}^{\infty}}&1&2&6&20&70&252&924&3432\\
\hline f_{k}^{\mathcal{B}^{\infty}}&1&3&10&35&126&462&1716&6435
\end{array}$$

Ce sont les séquences A000984 et A001700 de \cite{Sloane}.

\subsection{La sous-algèbre de Hopf $ \mathcal{B}^{1} $}

Muni de la concaténation, l'ensemble $ \mathcal{G}^{0} \cup \{ 1 \} $, constitué de l'arbre vide et de tous les arbres construits uniquement avec des $ B^{+} $, est un monoïde. Notons $ \mathcal{B}^{0} $ l'algèbre unitaire engendrée par ce monoïde. Si $ T $ est un arbre appartenant à $ \mathcal{G}^{0} $, il existe certaines coupes $ \boldsymbol{v} \models V(T) $ telles que $ Lea_{\boldsymbol{v}} (T) \not\in \mathcal{G}^{0} \cup \left\lbrace 1 \right\rbrace $ (voir la remarque qui suit le lemme \ref{comod}). $ \mathcal{B}^{0} $ n'est donc pas une cogèbre. Par contre, d'après le lemme \ref{comod}, $ \mathcal{B}^{0} $ est un comodule à droite de l'algèbre de Hopf $ \mathcal{B}^{\infty} $.\\

Remarquons que $ \mathcal{B}^{0} $ est isomorphe en tant qu'algèbre à l'algèbre (de Hopf) des arbres enracinés plans $ \mathcal{H}_{PR} $. En effet, on a vu au paragraphe qui précède la proposition \ref{2} que les arbres de $ \mathcal{G}^{0} $ sont en bijection avec ceux de $ \mathcal{H}_{PR} $. Cette bijection s'étend en un isomorphisme d'algèbres graduées de $ \mathcal{H}_{PR} $ dans $ \mathcal{B}^{0} $. En particulier, ces deux algèbres ont la même série formelle :
$$ F_{\mathcal{B}^{0}}(x)=\dfrac{1-\sqrt{1-4x}}{2x}=F_{\mathcal{H}_{PR}}(x) .$$

Considérons maintenant l'ensemble
$$ \mathcal{G}^{1} = \bigcup_{n \geq 1,\varepsilon_{n} \in \left\lbrace +,- \right\rbrace } \mathcal{G}^{(\overbrace{+, \hdots ,+}^{n-1 {\rm ~ fois}} ,\varepsilon_{n})} .$$
et notons $ \mathcal{B}^{1} $ l'algèbre $ \mathbb{K} [\mathcal{G}^{1} \cup \left\lbrace 1 \right\rbrace] $. Alors :

\begin{prop} \label{B1Hopf}
L'algèbre $ \mathcal{B}^{1} $ est une algèbre de Hopf.
\end{prop}

\begin{proof}
Il suffit de montrer que, étant donné un arbre appartenant à $ \mathcal{G}^{1} $, la branche et le tronc de cet arbre, après avoir réalisé une coupe simple, appartiennent respectivement à $ \mathcal{G}^{1} $ et $ \mathcal{G}^{1} \cup \left\lbrace 1 \right\rbrace $. On travaille par récurrence sur le degré $ n $ des arbres de $ \mathcal{B}^{1} $. Ceci est évident pour $ n \leq 3 $, en le vérifiant rapidement à la main.\\

Au rang $ n \geq 4 $. Considérons un arbre $ T \in \mathcal{G}^{1} $ de degré $ n $ et $ \boldsymbol{v} \models V(T) $ une coupe simple. Il y a deux cas:
\begin{enumerate}
\item Si $ T $ est de la forme $ B^{-}(T_{1},\hdots,T_{m}) \in \mathcal{G}^{(+, \hdots ,+,-)} $, avec $ T_{1},\hdots,T_{m} \in \mathcal{G}^{0} $. Si $ \boldsymbol{v} $ est la coupe totale, le résultat est trivial. Sinon, comme $ \boldsymbol{v} $ est une coupe simple de $ T $, il existe un unique $ i \in \left\lbrace 1,\hdots,m \right\rbrace $ tel que $ \boldsymbol{v} \models V(T_{i}) $ (on inclut le cas de la coupe totale qui correspond à couper l'arête entre la racine de $ T $ et celle de $ T_{i} $). Comme $ T_{i} \in \mathcal{G}^{0} $, par récurrence, $ Lea_{\boldsymbol{v}}(T) = Lea_{\boldsymbol{v}}(T_{i}) \in \mathcal{G}^{1} $. D'autre part, avec le lemme \ref{comod}, $ T_{i} $ appartenant à $ \mathcal{G}^{0} $, $ Roo_{\boldsymbol{v}}(T_{i}) \in \mathcal{G}^{0} \cup \left\lbrace 1 \right\rbrace $. Donc la forêt $ T_{1} \hdots Roo_{\boldsymbol{v}}(T_{i}) \hdots T_{m} $ appartient à $ \mathcal{G}^{0} \cup \left\lbrace 1 \right\rbrace $, et $ Roo_{\boldsymbol{v}}(T) = B^{-}(T_{1}, \hdots, Roo_{\boldsymbol{v}}(T_{i}) ,\hdots ,T_{m}) \in \mathcal{G}^{1} $. Le résultat est ainsi démontré si $ T $ est de la forme $ B^{-}(T_{1},\hdots,T_{m}) \in \mathcal{G}^{(+, \hdots ,+,-)} $.
\item Si $ T $ est de la forme $ B^{+}(T_{1},\hdots,T_{m}) \in \mathcal{G}^{(+, \hdots ,+)} $, avec $ T_{1},\hdots,T_{m} \in \mathcal{G}^{0} $. Si $ \boldsymbol{v} $ est la coupe simple correspondant à couper l'arête joignant la racine de $ T $ (qui est la racine de $ T_{1} $) et le sommet indexé par $ n $ joignant les racines communes de $ T_{2}, \hdots , T_{m} $, alors $ Roo_{\boldsymbol{v}}(T) = T_{1} \in \mathcal{G}^{0} $ et $ Lea_{\boldsymbol{v}}(T) = B^{-}(T_{2}, \hdots ,T_{m}) \in \mathcal{G}^{1} $. Le résultat est donc démontré dans ce cas. Sinon, comme $ \boldsymbol{v} $ est une coupe simple de $ T $, il existe un unique $ i \in \left\lbrace 1,\hdots,m \right\rbrace $ tel que $ \boldsymbol{v} \models V(T_{i}) $ (on inclut le cas de la coupe totale qui, si $ i \geq 2 $, correspond à couper l'arête entre le sommet de $ T $ indexé par $ n $ et la racine de $ T_{i} $). Il y a alors deux cas:
\begin{enumerate}
\item Si $ i = 1 $, c'est-à-dire si $ \boldsymbol{v} \models V(T_{1}) $. Si $ \boldsymbol{v} $ est la coupe totale, alors $ Lea_{\boldsymbol{v}}(T) = T $ et $ Roo_{\boldsymbol{v}}(T) = 1 $ et le résultat est établi. Sinon, par récurrence, $ Lea_{\boldsymbol{v}}(T) = Lea_{\boldsymbol{v}}(T_{1}) \in \mathcal{G}^{1} $. Avec le lemme \ref{comod}, comme $ Roo_{\boldsymbol{v}}(T_{1}) $ est différent de l'arbre vide, $ Roo_{\boldsymbol{v}}(T_{1}) \in \mathcal{G}^{0} $, donc $ Roo_{\boldsymbol{v}}(T_{1}) T_{2} \hdots T_{m} \in \mathcal{G}^{0} $. Et ainsi $ Roo_{\boldsymbol{v}}(T) = B^{+}(Roo_{\boldsymbol{v}}(T_{1}), T_{2}, \hdots ,T_{m}) \in \mathcal{G}^{0} \subseteq \mathcal{G}^{1} \cup \left\lbrace 1 \right\rbrace $.
\item Si $ i \geq 2 $. Par récurrence, comme $ T_{i} \in \mathcal{G}^{0} $, $ Lea_{\boldsymbol{v}}(T) = Lea_{\boldsymbol{v}}(T_{i}) \in \mathcal{G}^{1} $. D'autre part, avec le lemme \ref{comod}, $ Roo_{\boldsymbol{v}}(T_{i}) \in \mathcal{G}^{0} \cup \left\lbrace 1 \right\rbrace $. Donc la forêt $ T_{1} \hdots Roo_{\boldsymbol{v}}(T_{i}) \hdots T_{m} $ appartient à $ \mathcal{G}^{0} $, et $ Roo_{\boldsymbol{v}}(T) = B^{+}(T_{1}, \hdots, Roo_{\boldsymbol{v}}(T_{i}), \hdots ,T_{m}) \in \mathcal{G}^{0} \subseteq \mathcal{G}^{1} \cup \left\lbrace 1 \right\rbrace $.
\end{enumerate}
\end{enumerate}
Ceci démontre la stabilité de $ \mathcal{B}^{1} $ par coupe simple, et donc par coupe admissible, car $ \mathcal{B}^{1} $ est stable pour le produit. Ainsi $ \mathcal{B}^{1} $ est bien une algèbre de Hopf.
\end{proof}
\\

{\bf Remarque.} {On peut voir en reprenant cette démonstration que l'algèbre $ \mathcal{B}^{0} $ est même un comodule à droite de l'algèbre de Hopf $ \mathcal{B}^{1} $.}

\subsection{Généralisation de l'algèbre de Hopf $ \mathcal{B}^{1} $}

En utilisant le même modèle que pour la construction de l'algèbre de Hopf $ \mathcal{B}^{1} $, nous allons à présent construire une infinité d'algèbres de Hopf $ \mathcal{B}^{i} $, pour $ i \geq 2 $, telles qu'on ait les relations d'inclusions suivantes:
$$ \mathcal{B}^{1} \subseteq \mathcal{B}^{2} \subseteq \hdots \subseteq \mathcal{B}^{i} \subseteq \mathcal{B}^{i+1} \subseteq \hdots \subseteq \mathcal{B}^{\infty} .$$

Pour cela, posons, quelque soit $ i \geq 2 $,
$$ \mathcal{G}^{i} = \bigcup_{n \geq 1,(\varepsilon_{n-i+1}, \hdots , \varepsilon_{n}) \in \left\lbrace +,- \right\rbrace ^{i} } \mathcal{G}^{(\overbrace{+, \hdots ,+}^{n-i {\rm ~ fois}} ,\varepsilon_{n-i+1}, \hdots ,\varepsilon_{n})} .$$

et $ \mathcal{B}^{i} = \mathbb{K}[\mathcal{G}^{i} \cup \{ 1 \}] $. Remarquons que $ \mathcal{G}^{i} \subseteq \mathcal{G}^{i+1} $, $ \forall i \geq 1 $. Comme annoncé, on a alors le résultat suivant:

\begin{prop} 
Pour tout $ i \geq 1 $, l'algèbre $ \mathcal{B}^{i} $ est une algèbre de Hopf.
\end{prop}

\begin{proof}
Travaillons par récurrence sur $ i \geq 1 $. Le cas $ i=1 $ ayant déjà été démontré, supposons $ i \geq 2 $.\\

Il faut montrer que $ \mathcal{B}^{i} $ est stable par coupe simple. Pour cela, faisons une récurrence sur le degré $ n $ des arbres de $ \mathcal{G}^{i} $. Ceci est clair pour $ n \leq 3 $ et découle directement du fait que $ \mathcal{B}^{\infty} $ est une algèbre de Hopf.\\

Au rang $ n \geq 4 $. Considérons un arbre $ T \in \mathcal{G}^{i} $ de degré $ n $ et $ \boldsymbol{v} \models V(T) $ une coupe simple. Ici encore, il faut distinguer deux cas:
\begin{enumerate}
\item Si $ T $ est de la forme $ B^{-}(T_{1},\hdots,T_{m}) \in \mathcal{G}^{( \hdots ,-)} $. D'après le lemme \ref{1}, $ T_{1} \in \mathcal{G} $ et $ T_{2},\hdots,T_{m} \in \mathcal{G}^{0} $. Comme $ T \in \mathcal{G}^{i} $, la forêt $ T_{1} \hdots T_{m} $ appartient à $ \mathcal{G}^{i-1} $. En posant $ k = \left| T_{2}\right| + \hdots + \left| T_{m} \right| $, $ T_{1} \in \mathcal{G}^{i-1-k} $. Supposons $ \boldsymbol{v} $ différent de la coupe totale, le résultat étant trivial dans ce cas. Comme $ \boldsymbol{v} $ est une coupe simple de $ T $, il existe un unique $ j \in \left\lbrace 1,\hdots,m \right\rbrace $ tel que $ \boldsymbol{v} \models V(T_{j}) $ (on inclut le cas de la coupe totale qui correspond à couper l'arête entre la racine de $ T $ et celle de $ T_{j} $). Comme $ T_{j} \in \mathcal{G}^{i} $, par récurrence, $ Lea_{\boldsymbol{v}}(T) = Lea_{\boldsymbol{v}}(T_{j}) \in \mathcal{G}^{i} $. Pour $ Roo_{\boldsymbol{v}}(T) $, on doit alors différencier deux sous-cas:
\begin{enumerate}
\item Si $ j=1 $. $ T_{1} $ appartient à $ \mathcal{G}^{i-1-k} $. Par récurrence, $ Roo_{\boldsymbol{v}}(T_{1}) $ est un arbre appartenant à $ \mathcal{G}^{i-1-k} \cup \{ 1 \} $. D'après le lemme \ref{1}, la forêt $ Roo_{\boldsymbol{v}}(T_{1}) T_{2} \hdots T_{m} $ appartient à $ \mathcal{G}^{i-1} \cup \{ 1 \} $. Donc $ Roo_{\boldsymbol{v}}(T) = B^{-}(Roo_{\boldsymbol{v}}(T_{1}), T_{2}, \hdots ,T_{m}) \in \mathcal{G}^{i} $.
\item Si $ j \geq 2 $. Par le lemme \ref{comod}, $ Roo_{\boldsymbol{v}}(T_{j}) \in \mathcal{G}^{0} \cup \left\lbrace 1 \right\rbrace $. Notons $ l $ la somme $ \left| T_{2}\right| + \hdots + \left| Roo_{\boldsymbol{v}}(T_{j}) \right| + \hdots + \left| T_{m} \right| $. Alors $ l \leq k $ et la forêt $ T_{1} \hdots Roo_{\boldsymbol{v}}(T_{j}) \hdots T_{m} $ appartient à $ \mathcal{G}^{i-1+l-k} \subseteq \mathcal{G}^{i-1} $. Et ainsi $ Roo_{\boldsymbol{v}}(T) = B^{-}(T_{1}, \hdots ,Roo_{\boldsymbol{v}}(T_{j}), \hdots ,T_{m}) \in \mathcal{G}^{i} $.
\end{enumerate}
Ceci termine la démonstration dans le premier cas.
\item Si $ T $ est de la forme $ B^{+}(T_{1},\hdots,T_{m}) \in \mathcal{G}^{( \hdots ,+)} $. Avec le même raisonnement que dans le premier cas, en notant $ k = \left| T_{2}\right| + \hdots + \left| T_{m} \right| $, $ T_{1} \in \mathcal{G}^{i-1-k} $ et $ T_{2},\hdots,T_{m} \in \mathcal{G}^{0} $. Si $ \boldsymbol{v} $ est la coupe simple correspondant à couper l'arête joignant la racine de $ T $ (qui est la racine de $ T_{1} $) et le sommet indexé par $ n $ joignant les racines communes de $ T_{2}, \hdots , T_{m} $, alors $ Roo_{\boldsymbol{v}}(T) = T_{1} \in \mathcal{G}^{i-1-k} \subseteq \mathcal{G}^{i} $ et $ Lea_{\boldsymbol{v}}(T) = B^{-}(T_{2}, \hdots ,T_{m}) \in \mathcal{G}^{1} \subseteq \mathcal{G}^{i} $. Le résultat est donc démontré dans ce cas. Sinon, comme $ \boldsymbol{v} $ est une coupe simple de $ T $, il existe un unique $ j \in \left\lbrace 1,\hdots,m \right\rbrace $ tel que $ \boldsymbol{v} \models V(T_{j}) $ (on inclut le cas de la coupe totale qui, si $ j \geq 2 $, correspond à couper l'arête entre le sommet de $ T $ indexé par $ n $ et la racine de $ T_{j} $). Il y a alors deux sous-cas:
\begin{enumerate}
\item Si $ j=1 $. Dans le cas où $ \boldsymbol{v} $ est la coupe totale, $ Lea_{\boldsymbol{v}}(T) = T $ et $ Roo_{\boldsymbol{v}}(T) = 1 $ et le résultat est trivial. Sinon, par récurrence, comme $ T_{1} $ appartient à $ \mathcal{G}^{i-1-k} $, $ Lea_{\boldsymbol{v}}(T) = Lea_{\boldsymbol{v}}(T_{1}) \in \mathcal{G}^{i-1-k} \subseteq \mathcal{G}^{i} $ et $ Roo_{\boldsymbol{v}}(T_{1}) $ est un arbre appartenant à $ \mathcal{G}^{i-1-k} $. Avec le lemme \ref{1}, $ Roo_{\boldsymbol{v}}(T_{1}) T_{2} \hdots T_{m} \in \mathcal{G}^{i-1} $. Donc $ Roo_{\boldsymbol{v}}(T) = B^{+}(Roo_{\boldsymbol{v}}(T_{1}), T_{2}, \hdots ,T_{m}) \in \mathcal{G}^{i} $.
\item Si $ j \geq 2 $. Comme $ T_{j} \in \mathcal{G}^{0} $, par la proposition \ref{B1Hopf}, $ Lea_{\boldsymbol{v}}(T) = Lea_{\boldsymbol{v}}(T_{j}) \in \mathcal{G}^{1} \subseteq \mathcal{G}^{i} $. Par le lemme \ref{comod}, $ Roo_{\boldsymbol{v}}(T_{j}) \in \mathcal{G}^{0} \cup \left\lbrace 1 \right\rbrace $. Si on note $ l $ la somme $ \left| T_{2}\right| + \hdots + \left| Roo_{\boldsymbol{v}}(T_{j}) \right| + \hdots + \left| T_{m} \right| $, alors $ l \leq k $ et la forêt $ T_{1} \hdots Roo_{\boldsymbol{v}}(T_{j}) \hdots T_{m} $ appartient à $ \mathcal{G}^{i-1+l-k} \subseteq \mathcal{G}^{i-1} $. Ainsi $ Roo_{\boldsymbol{v}}(T) = B^{-}(T_{1}, \hdots ,Roo_{\boldsymbol{v}}(T_{j}), \hdots ,T_{m}) \in \mathcal{G}^{i} $.
\end{enumerate}
Le deuxième cas est donc démontré.
\end{enumerate}
Nous avons ainsi obtenu, par le principe de récurrence, le résultat annoncé.
\end{proof}
\\

{\bf Remarque.} {En reprenant la démonstration, on a, pour tout $ i \geq 2 $, $ \tdelta( \mathcal{B}^{i} ) \subseteq \mathcal{B}^{i} \otimes \mathcal{B}^{i-1} $.}
\\

Il est alors possible de calculer le nombre d'arbres de degré $ k $ de $ \mathcal{B}^{i} $:

\begin{prop}
Soit $ k \geq 1 $. Notons, pour tout $ i \geq 1 $, $ f_{1,k}^{\mathcal{B}^{i}} $ le nombre d'arbres de $ \mathcal{B}^{i} $ de degré $ k $. Rappelons que $ f_{1,k}^{\mathcal{B}^{\infty}} $ désigne le nombre d'arbres de $ \mathcal{B}^{\infty} $ de degré $ k $, et $ f_{k}^{\mathcal{B}^{0}} $ le nombre de forêts de $ \mathcal{B}^{0} $ de degré $ k $. Alors, pour tout $ k \geq 1 $, $ i \geq 1 $,
$$ f_{1,k}^{\mathcal{B}^{i}} = \left\{ \begin{array}{rcl}
& f_{1,k}^{\mathcal{B}^{\infty}} & {\rm ~ si ~ } k \leq i+1,\\
& f_{1,k}^{\mathcal{B}^{\infty}} - \displaystyle\sum_{1 \leq j \leq k-i-1} f_{j}^{\mathcal{B}^{0}} f_{1,k-j}^{\mathcal{B}^{\infty}} & {\rm ~ si ~ } k \geq i+2.
\end{array} \right. $$
\end{prop}

\begin{proof}
Fixons $ i \in \mathbb{N}^{\ast} $. La formule est évidente pour $ k \leq i+1 $ (en utilisant l'identification $ \mathcal{G}^{(+,\underline{\varepsilon})} = \mathcal{G}^{(-,\underline{\varepsilon})} $ lorsque $ k = i+1 $). Supposons $ k \geq i+1 $. Notons $ A^{k} $ l'ensemble
$$ \left\lbrace \underline{\varepsilon} \in \left\lbrace +,- \right\rbrace^{k} \mid \varepsilon_{1} = \hdots = \varepsilon_{k-i} = + \right\rbrace \subseteq \left\lbrace +,- \right\rbrace^{k} .$$
Le complémentaire dans $ \{ +,- \}^{k} $ de $ A^{k} $ est $ \displaystyle\bigcup_{0 \leq j \leq k-i-1} A^{k}_{j} $ (l'union étant disjointe) où
$$ A^{k}_{j} = \left\lbrace \underline{\varepsilon} \in \left\lbrace +,- \right\rbrace^{k} \mid \varepsilon_{1} = \hdots = \varepsilon_{j} = + {\rm ~ et ~} \varepsilon_{j+1} = - \right\rbrace \subseteq \left\lbrace +,- \right\rbrace^{k} .$$
D'après le lemme \ref{prelim}, pour tout $ 0 \leq j \leq k-i-1 $, les ensembles $ \left\lbrace  {\rm arbre ~ de ~ } \mathcal{G}^{(+,\underline{\varepsilon})} \mid \underline{\varepsilon} \in A^{k}_{j} \right\rbrace $ sont disjoints. De plus
\begin{eqnarray*}
{\rm card} \left( \left\lbrace  {\rm arbre ~ de ~ } \mathcal{G}^{(+,\underline{\varepsilon})} \mid \underline{\varepsilon} \in \{ +,- \}^{k} \right\rbrace \right) = f_{1,k+1}^{\mathcal{B}^{\infty}},\\
{\rm card} \left(  \left\lbrace  {\rm arbre ~ de ~ } \mathcal{G}^{(+,\underline{\varepsilon})} \mid \underline{\varepsilon} \in A^{k}_{j} \right\rbrace \right) = f_{j+1}^{\mathcal{B}^{0}} f_{1,k-j}^{\mathcal{B}^{\infty}}.
\end{eqnarray*}
Alors
\begin{eqnarray*} \begin{array}{rcl}
f_{1,k+1}^{\mathcal{B}^{i}} & = & {\rm card} \left( \left\lbrace  {\rm arbre ~ de ~ } \mathcal{G}^{(+,\underline{\varepsilon})} \mid \underline{\varepsilon} \in A^{k} \right\rbrace \right) \\
& = & {\rm card} \left( \left\lbrace  {\rm arbre ~ de ~ } \mathcal{G}^{(+,\underline{\varepsilon})} \mid \underline{\varepsilon} \in \{ +,- \}^{k} \right\rbrace \right) \\
& & \hspace{5cm} - {\rm card} \left(  \left\lbrace  {\rm arbre ~ de ~ } \mathcal{G}^{(+,\underline{\varepsilon})} \mid \underline{\varepsilon} \in \displaystyle\bigcup_{0 \leq j \leq k-i-1} A^{k}_{j} \right\rbrace \right) \\
& = & f_{1,k+1}^{\mathcal{B}^{\infty}} - {\rm card} \left( \displaystyle\bigcup_{0 \leq j \leq k-i-1} \left\lbrace  {\rm arbre ~ de ~ } \mathcal{G}^{(+,\underline{\varepsilon})} \mid \underline{\varepsilon} \in A^{k}_{j} \right\rbrace \right) \\
& = & f_{1,k+1}^{\mathcal{B}^{\infty}} - \displaystyle\sum_{0 \leq j \leq k-i-1} {\rm card} \left(  \left\lbrace  {\rm arbre ~ de ~ } \mathcal{G}^{(+,\underline{\varepsilon})} \mid \underline{\varepsilon} \in A^{k}_{j} \right\rbrace \right)\\
& = & f_{1,k+1}^{\mathcal{B}^{\infty}} - \displaystyle\sum_{0 \leq j \leq k-i-1} f_{j+1}^{\mathcal{B}^{0}} f_{1,k-j}^{\mathcal{B}^{\infty}} .
\end{array}
\end{eqnarray*}
\vspace{-0.3cm}
\end{proof}
\\

Voici quelques valeurs numériques:

\begin{enumerate}
\item Pour les arbres:
$$\begin{array}{c|c|c|c|c|c|c|c|c}
k&1&2&3&4&5&6&7&8\\
\hline f_{1,k}^{\mathcal{B}^{0}}&1&1&2&5&14&42&132&429\\
\hline f_{1,k}^{\mathcal{B}^{1}}&1&2&4&10&28&84&264&858\\
\hline f_{1,k}^{\mathcal{B}^{2}}&1&2&6&14&38&112&348&1122\\
\hline f_{1,k}^{\mathcal{B}^{3}}&1&2&6&20&50&142&432&1374\\
\hline f_{1,k}^{\mathcal{B}^{4}}&1&2&6&20&70&182&532&1654\\
\hline f_{1,k}^{\mathcal{B}^{5}}&1&2&6&20&70&252&672&2004\\
\hline f_{1,k}^{\mathcal{B}^{6}}&1&2&6&20&70&252&924&2508\\
\end{array}$$
\item Pour les forêts:
$$\begin{array}{c|c|c|c|c|c|c|c|c}
k&1&2&3&4&5&6&7&8\\
\hline f_{k}^{\mathcal{B}^{0}}&1&2&5&14&42&132&429&1430\\
\hline f_{k}^{\mathcal{B}^{1}}&1&3&9&29&97&333&1165&4135\\
\hline f_{k}^{\mathcal{B}^{2}}&1&3&11&37&129&461&1669&6107\\
\hline f_{k}^{\mathcal{B}^{3}}&1&3&11&43&153&557&2065&7739\\
\hline f_{k}^{\mathcal{B}^{4}}&1&3&11&43&173&637&2385&9059\\
\hline f_{k}^{\mathcal{B}^{5}}&1&3&11&43&173&707&2665&10179\\
\hline f_{k}^{\mathcal{B}^{6}}&1&3&11&43&173&707&2917&11187\\
\end{array}$$
\end{enumerate}

\section{L'algèbre de Hopf $ \mathcal{B} $}

\subsection{Construction de $ \mathcal{B} $ et premières propriétés}

A partir d'une construction similaire à celle décrite dans la partie \ref{partie1}, nous allons définir une nouvelle algèbre notée $ \mathcal{B} $ contenant l'algèbre de Hopf $ \mathcal{B}^{\infty} $. Pour cela, on construit inductivement un ensemble d'arbres $ \mathcal{T}^{(n,\varepsilon)} $ de degré $ n \geq 1 $, pour $ \varepsilon \in \lbrace +,- \rbrace $.\\

Si $ \underline{n=1} $, $ \mathcal{T}^{(1,\varepsilon)} $, pour $ \varepsilon $ quelconque, est l'ensemble réduit à un seul élément, l'unique arbre de degré 1.\\

Si $ \underline{n \geq 2} $, supposons les ensembles $ \mathcal{T}^{(k,\varepsilon)} $ construits, pour tout $ 1 \leq k \leq n-1 $.

\begin{enumerate}
\item Si $ \varepsilon=- $, les arbres $ T $ de $ \mathcal{T}^{(n,-)} $ sont construits comme suit. On prend une forêt $ F=T_{1} \hdots T_{m} $ de degré $ n-1 $ construite à partir de $ m \geq 1 $ arbres appartenant à l'ensemble $ \bigcup_{1 \leq k \leq n-1,\varepsilon \in \lbrace +,- \rbrace} \mathcal{T}^{(k,\varepsilon)} $ et on considère $ T_{1}, \hdots ,T_{m} $ comme la suite des sous-arbres d'un arbre enraciné ayant pour racine le sommet indexé par $ n $. Cela donne ainsi naissance à un nouvel arbre ordonné $ T=B^{-}(T_{1}, \hdots ,T_{m}) $ de degré $ n $.
\item Si $ \varepsilon=+ $, on construit alors les arbres $ T $ de $ \mathcal{T}^{(n,+)} $ par la transformation suivante. On prend une forêt $ F=T_{1} \hdots T_{m} $ de degré $ n-1 $ avec $ T_{1}, \hdots ,T_{m} \in \bigcup_{1 \leq k \leq n-1,\varepsilon \in \lbrace +,- \rbrace} \mathcal{T}^{(k,\varepsilon)} $ et $ m \geq 1 $. On construit alors $ T $ en greffant le sommet indexé par $ n $ comme le fils le plus à droite de la racine de $ T_{1} $ et en considerant $ T_{2}, \hdots ,T_{m} $ comme la suite des sous-arbres issus du sommet indexé par $ n $. Cela donne un nouvel arbre ordonné $ T=B^{+}(T_{1}, \hdots ,T_{m}) $ de degré $ n $.
\end{enumerate}

Considérons un arbre $ T $ construit à partir des instructions précédentes, de degré $ n $. Soit le sommet indexé par $ n $ est la racine et $ T \in \mathcal{T}^{(n,-)} $, soit le sommet indexé par $ n $ est le fils le plus à droite de la racine et $ T \in \mathcal{T}^{(n,+)} $. De plus, deux arbres de $ \mathcal{T}^{(n,-)} $ (resp. $ \mathcal{T}^{(n,+)} $) sont égaux si et seulement si les forêts à partir desquelles ils sont construits sont égales. Ainsi, les arbres construits avec les instructions précédentes sont tous distincts. En particulier, $ {\rm card} \left( \mathcal{T}^{(n,+)} \right) = {\rm card} \left( \mathcal{T}^{(n,-)} \right) $, $ \forall ~ n \geq 1 $.\\

On pose alors $ \mathcal{T} $ l'ensemble $ \bigcup_{n \geq 1,\varepsilon \in \lbrace +,- \rbrace } \mathcal{T}^{(n,\varepsilon)} $, l'union étant disjointe (on identifie $ \mathcal{T}^{(1,-)} $ et $ \mathcal{T}^{(1,+)} $). Voici une illustration pour $ n=1,2,3 $ et $ 4 $:

\begin{eqnarray*}
\mathcal{T}^{(1,-)} = \mathcal{T}^{(1,+)} & = & \{\tdun{1}\}\\
\mathcal{T}^{(2,+)} & = & \{\tddeux{1}{2}\}\\
\mathcal{T}^{(2,-)} & = & \{\tddeux{2}{1}\}\\
\mathcal{T}^{(3,+)} & = & \{\tdtroisun{1}{3}{2},\tdtroisun{2}{3}{1},\tdtroisdeux{1}{3}{2}\}\\
\mathcal{T}^{(3,-)} & = & \{\tdtroisdeux{3}{1}{2},\tdtroisdeux{3}{2}{1},\tdtroisun{3}{2}{1}\}\\
\mathcal{T}^{(4,+)} & = & \{\tdquatreun{1}{4}{3}{2},\tdquatredeux{3}{4}{1}{2} ,\tdquatreun{2}{4}{3}{1},\tdquatredeux{3}{4}{2}{1} ,\tdquatreun{3}{4}{2}{1},\tdquatredeux{1}{4}{3}{2} ,\tdquatretrois{1}{4}{3}{2} ,\tdquatrecinq{1}{4}{2}{3} ,\tdquatretrois{2}{4}{3}{1} ,\tdquatrecinq{1}{4}{3}{2} ,\tdquatrequatre{1}{4}{3}{2} \}\\
\mathcal{T}^{(4,-)} & = & \{ \tdquatrequatre{4}{1}{3}{2},\tdquatrecinq{4}{3}{1}{2} ,\tdquatrequatre{4}{2}{3}{1},\tdquatrecinq{4}{3}{2}{1} ,\tdquatrequatre{4}{3}{2}{1},\tdquatrecinq{4}{1}{3}{2} ,\tdquatredeux{4}{3}{1}{2} ,\tdquatretrois{4}{2}{3}{1} ,\tdquatredeux{4}{3}{2}{1} ,\tdquatretrois{4}{3}{2}{1} ,\tdquatreun{4}{3}{2}{1} \}\\
\end{eqnarray*}

Soit $ T $ un arbre enraciné plan non vide $ \in \mathcal{H}_{PR} $. Ici aussi il est possible de dénombrer le nombre de façons d'indexer les sommets de $ T $ pour obtenir un élément de $ \mathcal{T} $ :

\begin{prop} \label{ordre}
Soit $ T $ un arbre plan. On définit pour chaque sommet $ v $ de $ T $, un entier $ a_{v} $ par récurrence sur la hauteur:
\begin{enumerate}
\item Si $ v $ est une feuille, $ a_{v} = 1 $.
\item Sinon,
$$ a_{v} = \prod_{v' \rightarrow v} a_{v'}  + \sum_{v' \rightarrow v} \left( \prod_{v'' \rightarrow v {\rm ~ tq ~ }  v'' < v' } a_{v''}\right)  \left(  \prod_{v'' \rightarrow w \rightarrow v {\rm ~ tq ~ }  v' \leq w} a_{v''} \right) ,$$
l'ordre sur les sommets de $ T $ étant celui défini dans l'introduction.
\end{enumerate}
Alors, il y a $ a_{racine(T)} $ façons d'indexer les sommets de $ T $ pour obtenir un élément de $ \mathcal{T} $.
\end{prop}

\begin{proof}
Pour éviter d'alourdir les notations, on notera indifférement un arbre appartenant à $ \mathcal{H}_{PR} $ et l'arbre (ou les arbres) obtenu après indexation des sommets appartenant à $ \mathcal{T} $.\\

Par récurrence sur le degré $ k $ de $ T \in \mathcal{T} $. Si $ k = 1 $, c'est trivial. Supposons $ k \geq 2 $ et le résultat vérifié au rang inférieur. On utilise les mêmes notations que pour la démonstration de la proposition \ref{2}: $ T = B(T_{1}, \hdots , T_{n}) $, avec $ T_{1}=B(T_{1,1}, \hdots T_{1,m_{1}}), \hdots , T_{n}=B(T_{n,1}, \hdots , T_{n,m_{n}}) $ (si $ T_{i} = \tun $, $ T_{i} = B(T_{i,1}) $ avec $ T_{i,1}=1 $). Comme dans la preuve de la proposition \ref{2}, on ne cherchera pas ici à simplifier un produit d'arbres en supprimant les éventuels arbres vides. Il y a $ n+1 $ situations possibles au niveau de la racine de $ T $: pour $ 1 \leq i \leq n $,
$$ \overbrace{B^{+}(\hdots B^{+}}^{n-i ~ {\rm fois}}(B^{-}(T_{1} , \hdots , T_{i}), T_{i+1,1} , \hdots , T_{i+1,m_{i+1}}), \hdots ), T_{n,1} , \hdots, T_{n,m_{n}}) \in \mathcal{T} ,$$
et $ B^{+}( \hdots B^{+}(\tdun{1}, T_{1,1}, \hdots , T_{1,m_{1}}), \hdots ), T_{n,1} , \hdots, T_{n,m_{n}}) \in \mathcal{T} $. Pour deux valeurs distinctes de $ i $, l'indexation des sommets de $ T $ sera différente. Fixons un $ 0 \leq i \leq n $. Par hypothèse de récurrence, il y a $ a_{racine(T_{j})} $ façons d'indexer $ T_{j} $ en un élément de $ \mathcal{T} $, pour $ 1 \leq j \leq i $; et $ a_{racine(T_{j,l})} $ façons d'indexer $ T_{j,l} $ en un élément de $ \mathcal{T} $, pour $ i+1 \leq j \leq n $ et $ 1 \leq l \leq m_{j} $ (si $ T_{j,l} $ est différent de l'arbre vide). Au total, il y a bien
$$ a_{racine(T)} = \prod_{v' \rightarrow r} a_{v'}  + \sum_{v' \rightarrow r} \left( \prod_{v'' \rightarrow r {\rm ~ tq ~ }  v'' < v' } a_{v''}\right)  \left(  \prod_{v'' \rightarrow w \rightarrow r {\rm ~ tq ~ }  v' \leq w} a_{v''} \right) $$
façons d'indexer les sommets de $ T $ pour obtenir un élément de $ \mathcal{T} $.
\end{proof}
\\

Définissons l'algèbre $ \mathcal{B} $ comme l'algèbre librement engendrée par l'ensemble $ \mathcal{T} \cup \left\lbrace 1 \right\rbrace $. On a alors la propriété remarquable suivante :

\begin{prop} L'algèbre $ \mathcal{B} $ est une algèbre de Hopf.
\end{prop}

\begin{proof}
Montrons que, en réalisant une coupe simple d'un arbre appartenant à $ \mathcal{T} $, la branche et le tronc appartiennent respectivement à $ \mathcal{T} $ et $ \mathcal{T} \cup \left\lbrace 1 \right\rbrace $. On travaille par récurrence sur le degré des arbres. Le résultat est trivial pour $ n=1,2 $ et $ 3 $. Au rang $ n \geq 4 $. Considérons un arbre $ T \in \mathcal{T} $ de degré $ n $ et $ \boldsymbol{v} \models V(T) $ une coupe simple. Il y a deux cas:

\begin{enumerate}
\item Si l'arbre est de la forme $ T = B^{-}(T_{1},\hdots,T_{m}) \in \mathcal{T}^{(n,-)} $, avec $ m \geq 1 $. Par construction, $ T_{1}, \hdots ,T_{m} \in \mathcal{T} $. Si $ \boldsymbol{v} $ est la coupe totale, le résultat est évident. Sinon, comme $ \boldsymbol{v} $ est une coupe simple de $ T $, il existe un unique $ i \in \left\lbrace 1,\hdots,m \right\rbrace $ tel que $ \boldsymbol{v} \models V(T_{i}) $ (on inclut le cas de la coupe totale qui correspond à couper l'arête entre la racine de $ T $ et celle de $ T_{i} $). Par récurrence, comme $ T_{i} \in \mathcal{T} $, $ Lea_{\boldsymbol{v}}(T) = Lea_{\boldsymbol{v}}(T_{i}) $ appartient à $ \mathcal{T} $. De même, par récurrence $ Roo_{\boldsymbol{v}}(T_{i}) $ appartient à $ \mathcal{T} \cup \left\lbrace 1 \right\rbrace $ donc la forêt $ T_{1} \hdots Roo_{\boldsymbol{v}}(T_{i}) \hdots T_{m} $ est constituée d'arbres appartenant à $ \mathcal{T} \cup \left\lbrace 1 \right\rbrace $ et ainsi $ Roo_{\boldsymbol{v}}(T) = B^{-}(T_{1}, \hdots, Roo_{\boldsymbol{v}}(T_{i}), \hdots ,T_{m}) \in \mathcal{T} $.
\item Si l'arbre est de la forme $ T = B^{+}(T_{1},\hdots,T_{m}) \in \mathcal{T}^{(n,+)} $, avec $ m \geq 1 $. Comme précédemment, $ T_{1}, \hdots ,T_{m} \in \mathcal{T} $. Alors:
\begin{enumerate}
\item Si $ \boldsymbol{v} $ est la coupe simple correspondant à couper l'arête joignant la racine de $ T $ (qui est la racine de $ T_{1} $) et le sommet indexé par $ n $ joignant les racines communes de $ T_{2}, \hdots , T_{m} $, alors $ Roo_{\boldsymbol{v}}(T) = T_{1} \in \mathcal{T} $ et $ Lea_{\boldsymbol{v}}(T) = B^{-}(T_{2}, \hdots ,T_{m}) \in \mathcal{T} $. Le résultat est donc démontré dans ce cas.
\item Sinon, comme $ \boldsymbol{v} $ est une coupe simple de $ T $, il existe un unique $ i \in \left\lbrace 1,\hdots,m \right\rbrace $ tel que $ \boldsymbol{v} \models V(T_{i}) $ (on inclut le cas de la coupe totale qui, si $ i \geq 2 $, correspond à couper l'arête entre le sommet de $ T $ indexé par $ n $ et la racine de $ T_{j} $). Deux cas sont à distinguer:
\begin{enumerate}
\item Si $ i = 1 $. Si $ \boldsymbol{v} \models V(T_{1}) $ est la coupe totale, alors $ Roo_{\boldsymbol{v}}(T) = 1 $ et $ Lea_{\boldsymbol{v}}(T) = T $ et le résultat est trivial. Sinon, par récurrence, $ Lea_{\boldsymbol{v}}(T) = Lea_{\boldsymbol{v}}(T_{1}) \in \mathcal{T} $ et $ Roo_{\boldsymbol{v}}(T_{1}) \in \mathcal{T} $, donc $ Roo_{\boldsymbol{v}}(T) = B^{+}( Roo_{\boldsymbol{v}}(T_{1}), T_{2}, \hdots, T_{m}) \in \mathcal{T} $.
\item Si $ i \geq 2 $. Par récurrence, $ Lea_{\boldsymbol{v}}(T) = Lea_{\boldsymbol{v}}(T_{i}) $ appartient à $ \mathcal{T} $ et $ Roo_{\boldsymbol{v}}(T_{i}) $ appartient à $ \mathcal{T} \cup \left\lbrace 1 \right\rbrace $. Donc $ Roo_{\boldsymbol{v}}(T) = B^{+}(T_{1}, \hdots ,Roo_{\boldsymbol{v}}(T_{i}), \hdots, T_{m}) \in \mathcal{T} $.
\end{enumerate}
\end{enumerate}
\end{enumerate}
Ainsi, dans tous les cas, $ Roo_{\boldsymbol{v}}(T) \in \mathcal{T} \cup \left\lbrace 1 \right\rbrace $ et $ Lea_{\boldsymbol{v}}(T) \in \mathcal{T} $, et on peut conclure par le principe de récurrence.
\end{proof}
\\

{\bf Remarque.} {Les relations d'inclusion suivantes sont évidement vérifiées :
$$ \mathcal{B}^{0} \subseteq \hdots \subseteq \mathcal{B}^{i} \subseteq \hdots \subseteq \mathcal{B}^{\infty} \subseteq \mathcal{B} .$$
Rappelons que $ \mathcal{B}^{0} $ est la sous-algèbre de $ \mathcal{B} $ engendrée par les arbres construits uniquement avec des $ B^{+} $. De la même façon, on peut définir une sous-algèbre de $ \mathcal{B} $ en considérant la sous-algèbre engendrée par les arbres de $ \mathcal{B} $ qui sont construits uniquement avec des $ B^{-} $. Notons-la $ \mathcal{B}_{l} $ (cette terminologie est justifiée par le théorème \ref{libregauche}). Elle est clairement stable par coupe admissible, c'est donc une algèbre de Hopf. Il existe un isomorphisme d'algèbres de Hopf entre $ \mathcal{B}_{l} $ et $ \mathcal{H}_{PR} $: à chaque arbre de $ \mathcal{H}_{PR} $ il y a une seule et unique façon de numéroter les sommets (en numérotant les sommets dans l'ordre croissant pour l'ordre défini dans l'introduction); cela définit une bijection entre les arbres de $ \mathcal{B}_{l} $ et les arbres de $ \mathcal{H}_{PR} $ qui s'étend en un isomorphisme d'algèbres graduées respectant le coproduit.}\\

Il est possible de calculer la série formelle de l'algèbre $ \mathcal{B} $ :

\begin{prop}
La série formelle de l'algèbre de Hopf $ \mathcal{B} $ est donnée par la formule:
$$ F_{\mathcal{B}}(x) = \dfrac{1+x-\sqrt{1-6x+x^{2}}}{4x} .$$
\end{prop}

\begin{proof}
Notons $ f^{\mathcal{B}}_{n} $ le nombre de forêts de degré $ n $ de $ \mathcal{B} $ et $ f_{1,n}^{\mathcal{B}} $ le nombre d'arbres de degré $ n $. On déduit de la construction de $ \mathcal{B} $ les relations suivantes :
\begin{eqnarray*}
f_{1}^{\mathcal{B}} & = & 1 \\
f_{1,1}^{\mathcal{B}} & = & 1 \\
f_{1,n}^{\mathcal{B}} & = & 2 f_{n-1}^{\mathcal{B}}  ~ {\rm si} ~ n \geq 2.
\end{eqnarray*}

Introduisons la convention suivante: $ f_{1,0}^{\mathcal{B}} = 0 $ et $ f_{0}^{\mathcal{B}} = 1 $. On pose alors $ F_{\mathcal{B}}(x)=\sum_{n \geq 0} f^{\mathcal{B}}_{n} x^{n} $ et $ T_{\mathcal{B}}(x) = \sum_{n \geq 0} f_{1,n}^{\mathcal{B}} x^{n} $. L'algèbre étant libre,
$$ F_{\mathcal{B}}(x) = \dfrac{1}{1-T_{\mathcal{B}}(x)} .$$
Alors :
\begin{eqnarray*}
T_{\mathcal{B}}(x) - x & = & \sum_{n \geq 2} f_{1,n}^{\mathcal{B}} x^{n}\\
& = & 2x \left( \sum_{n \geq 1} f_{n}^{\mathcal{B}} x^{n} \right) \\
& = & 2x \left( \dfrac{1}{1-T_{\mathcal{B}}(x)} - 1 \right) 
\end{eqnarray*}

Donc: $ T_{\mathcal{B}}^{2}(x) + (x-1)T_{\mathcal{B}}(x)+x=0 $. Comme $ f_{1,0}^{\mathcal{B}} = 0 $,
$$ T_{\mathcal{B}}(x) = \dfrac{1-x-\sqrt{1-6x+x^{2}}}{2} \hspace{0.3cm} {\rm et } \hspace{0.3cm} F_{\mathcal{B}}(x) = \dfrac{1+x-\sqrt{1-6x+x^{2}}}{4x} .$$
\end{proof}
\\

Voici quelques valeurs numériques:

$$\begin{array}{c|c|c|c|c|c|c|c|c}
k&1&2&3&4&5&6&7&8\\
\hline f_{1,k}^{\mathcal{B}}&1&2&6&22&90&394&1806&8558\\
\hline f_{k}^{\mathcal{B}}&1&3&11&45&197&903&4279&20793
\end{array}$$

Ce sont les séquences A006318 et A001003 de \cite{Sloane}.

\subsection{Coliberté de $ \mathcal{B} $}

Introduisons une nouvelle opération sur $ \mathcal{B} $:\\

\'Etant données deux forêts non vides $ F $ et $ G $ appartenant à $ \mathcal{B} $, on définit une forêt $ F \nwarrow G $ en greffant à la feuille la plus à droite de $ F $ la forêt $ G $ et en indexant les sommets comme suit: on laisse les indices de $ F $ inferieurs strictement à l'indice de sa feuille la plus à droite invariant; on numérote ensuite la forêt $ G $ en préservant l'ordre d'origine de ses indices; on finit alors en numérotant le reste des sommets non encore indexés de $ F $ en préservant ici encore l'ordre d'origine des indices dans $ F $. Si $ F $ est une forêt de $ \mathcal{B} $, on pose $ 1 \nwarrow F = F \nwarrow 1 = F $. Par linéarité, on définit ainsi une nouvelle opération $ \nwarrow : \mathcal{B} \times \mathcal{B} \rightarrow \mathcal{B} $.\\

{\bf Remarque.} {Dans le cas particulier où $ F = \tdun{1} $, $ \tdun{$1$} \nwarrow G = B^{-}(G_{1}, \hdots ,G_{n}) $, pour toute forêt $ G = G_{1} \hdots G_{n} \in \mathcal{B} $.}
\\

{\bf Exemples.} On illustre ci-dessous l'opération $ \nwarrow $:
$$\begin{array}{|rclcl|rclcl|rclcl|}
\hline \tdun{1}\tdun{2}\tdun{3} &\nwarrow& \tddeux{1}{2} &=& \tdun{1}\tdun{2}\tdtroisdeux{5}{3}{4} &\tdun{1}\tdun{2}\tdun{3} &\nwarrow& \tddeux{2}{1} &=&\tdun{1}\tdun{2}\tdtroisdeux{5}{4}{3}&
\tdun{1} \tdun{2} &\nwarrow& \tdun{1} \tdun{2} &=& \tdun{1} \tdtroisun{4}{3}{2} \\
\tddeux{2}{1} &\nwarrow& \tdun{1}\tdun{2} &=&\tdquatrequatre{4}{3}{2}{1}& \tdtroisun{2}{3}{1} &\nwarrow& \tdun{1} &=&\tdquatretrois{2}{4}{3}{1}&
\tdtroisun{1}{3}{2}&\nwarrow& \tdun{1} &=&\tdquatretrois{1}{4}{3}{2}\\
\tdun{1} \tddeux{2}{3} &\nwarrow& \tddeux{2}{1}&=&\tdun{1} \tdquatrecinq{2}{5}{4}{3}& \tddeux{2}{1} &\nwarrow& \tddeux{2}{1} &=&\tdquatrecinq{4}{3}{2}{1}&
\tddeux{1}{2}&\nwarrow&\tdun{1}\tdun{2}&=&\tdquatrequatre{1}{4}{3}{2}\\
\hline \end{array}$$

\vspace{0.5cm}

Il faut montrer que $ \mathcal{B} $ est stable par l'opération de greffe $ \nwarrow $ ainsi définie. Par définition de la greffe, il suffit de montrer le résultat lorsque $ F $ est un arbre non vide et $ G = G_{1} \hdots G_{n} $ une forêt non vide de $ \mathcal{B} $. Comme $ \mathcal{B} $ est librement engendrée par $ \mathcal{T} \cup \left\lbrace 1 \right\rbrace $ en tant qu'algèbre, $ G_{1}, \hdots, G_{n} \in \mathcal{T} $. Si $ F $ est l'arbre constitué d'un unique sommet, alors $ F \nwarrow G = B^{-}(G_{1}, \hdots , G_{n}) \in \mathcal{T} \subseteq \mathcal{B} $. Supposons maintenant $ F $ de degré $ \geq 2 $. Alors, la branche de $ F $ sur laquelle on greffe $ G $ comporte au moins deux sommets. Il y a deux cas à distinguer:
\begin{enumerate}
\item Si l'indice de la feuille la plus à droite de $ F $ (celle où on greffe) est plus grand que l'indice de son père. Dans ce cas, cette feuille a été insérée lors de la construction de $ F $ par un $ B^{+}(\hdots,1) $. Pour obtenir l'arbre $ F \nwarrow G $, il suffit alors de reproduire les opérations faites pour construire $ F $ à un changement près, en remplaçant $ B^{+}(\hdots, 1) $ par $ B^{+}(\hdots,G_{1}, \hdots , G_{n}) $. Ainsi, $ F \nwarrow G $ appartient à $ \mathcal{T} \subseteq \mathcal{B} $.
\item Si l'indice de la feuille la plus à droite de $ F $ est plus petit que l'indice de son père. Dans ce cas, cette feuille a été insérée lors de la construction de $ F $ par un $ B^{-}(H_{1},\hdots,H_{k}, \tdun{1}) $ où $ H_{1}, \hdots ,H_{k}, \tdun{1} $ est la suite des sous-arbres issus du même sommet de $ F $, le père de la feuille la plus à droite de $ F $. Par construction, $ H_{1}, \hdots , H_{k} \in \mathcal{T} $. Pour obtenir l'arbre $ F \nwarrow G $, il suffit alors de reproduire les opérations faites pour construire $ F $ à un changement près, en remplaçant $ B^{-}(H_{1},\hdots,H_{k}, \tdun{1}) $ par $ B^{-}(H_{1}, \hdots,H_{k}, B^{-}(G_{1}, \hdots , G_{n})) $. Donc $ F \nwarrow G $ appartient aussi à $ \mathcal{T} \subseteq \mathcal{B} $ dans ce cas.
\end{enumerate}

\vspace{0.5cm}

Le lemme qui suit est évident:

\begin{lemma}
Pour toutes forêts $ F, G, H \in \mathcal{B} $,
\begin{eqnarray*}
(F \nwarrow G) \nwarrow H & = & F \nwarrow (G \nwarrow H) ,\\
(FG) \nwarrow H & = & F (G \nwarrow H) .
\end{eqnarray*}
\end{lemma}

Rappelons la définition suivante (voir par exemple \cite{Foissy4,Loday2}) :

\begin{defi}
Une algèbre dupliciale est un triplet $ (A,\ast,\nwarrow) $, où $ A $ est un espace vectoriel et $ \ast,\nwarrow: A \otimes A \longrightarrow A $, avec les axiomes suivants: pour tout $ x,y,z \in A $,
\begin{equation}\label{E1}\left\{\begin{array}{rcl}
(x \ast y) \ast z & = & x \ast (y \ast z) ,\\
(x \nwarrow y) \nwarrow z & = & x \nwarrow (y \nwarrow z) ,\\
(x \ast y) \nwarrow z & = & x \ast (y \nwarrow z) .
\end{array}\right. \end{equation}
Une algèbre dupliciale unitaire $ A $ est un espace vectoriel $ A = \mathbb{K}1 \oplus \overline{A} $ tel que $ \overline{A} $ est une algèbre dupliciale et où on a étendu les deux opérations $ \ast $ et $ \nwarrow $ comme suit: Pour tout $ a \in A $,
\begin{equation*}\begin{array}{rcl}
1 \ast a = a \ast 1 & = & a, \\
1 \nwarrow a = a \nwarrow 1 & = & a.
\end{array}\end{equation*}
\end{defi}

Ainsi, $ \mathcal{B} $ munit de la concaténation et de $ \nwarrow $ est une algèbre dupliciale unitaire , avec $ \overline{\mathcal{B}} = (\mathcal{B})_{+} $ et $ 1 $ l'arbre vide, de tel sorte que $ \mathcal{B} = \mathbb{K}1 \oplus \overline{\mathcal{B}} $.

\begin{prop}
$ (\mathcal{B}_{l})_{+} $ est l'algèbre dupliciale libre générée par l'élément $ \tdun{$1$} $.
\end{prop}

\begin{proof}
Soit $ A $ une algèbre dupliciale et soit $ a \in A $. Il s'agit de montrer qu'il existe un unique morphisme d'algèbres dupliciales $ \phi : (\mathcal{B}_{l})_{+} \rightarrow A $ tel que $ \phi (\tdun{$1$}) = a $. On définit $ \phi (F) $ pour toute forêt non vide $ F \in (\mathcal{B}_{l})_{+} $ inductivement sur le degré de $ F $:
\begin{equation*}\begin{array}{rcl}
\phi (\tdun{$1$}) & = & a,\\
\phi (T_{1} \hdots T_{k}) & = & \phi (T_{1}) \hdots \phi (T_{k})  ~ {\rm si} ~ k \geq 2,\\
\phi( B^{-}(T_{1}, \hdots ,T_{k})) & = & a \nwarrow \phi(T_{1} \hdots T_{k}) ~ {\rm si} ~ k \geq 1,
\end{array}\end{equation*}
avec $ T_{1}, \hdots ,T_{k} \in \mathcal{T} \cap \mathcal{B}_{l} $. Comme le produit dans $ A $ est associatif, $ \phi $ est bien définie. $ \phi $ s'étend par linéarité en une application $ \phi : (\mathcal{B}_{l})_{+} \rightarrow A $. Montrons que c'est un morphisme d'algèbres dupliciales. Par le second point, $ \phi(FG)= \phi(F) \phi(G) $ pour toutes forêts $ F,G \in (\mathcal{B}_{l})_{+} $. Il reste à prouver que $ \phi (F \nwarrow G) = \phi(F) \nwarrow \phi(G) $ pour toutes forêts $ F,G \in (\mathcal{B}_{l})_{+} $. On pose $ F=F_{1} \hdots F_{k} $, avec $ k \geq 1 $ et $ F_{1}, \hdots ,F_{k} \in \mathcal{T} \cap \mathcal{B}_{l} $. Par récurrence sur le degré $ n $ de $ F_{k} $. Si $ n=1 $, alors $ F_{k} = \tdun{$1$} $, et, en notant $ G = G_{1} \hdots G_{m} $, avec $ m \geq 1 $ et $ G_{1}, \hdots ,G_{m} \in \mathcal{T} \cap \mathcal{B}_{l} $:
\begin{eqnarray*}
\phi (F \nwarrow G) & = & \phi((F_{1} \hdots F_{k-1} \tdun{1}) \nwarrow G)\\
& = & \phi(F_{1} \hdots F_{k-1} B^{-}(G_{1}, \hdots ,G_{m}))\\
& = & \phi(F_{1} \hdots F_{k-1}) \phi(B^{-}(G_{1}, \hdots ,G_{m}))\\
& = & \phi(F_{1} \hdots F_{k-1}) (a \nwarrow \phi (G_{1} \hdots G_{m}))\\
& = & \phi(F_{1} \hdots F_{k-1}) (\phi(F_{k}) \nwarrow \phi(G))\\
& = & \phi (F) \nwarrow \phi (G) .
\end{eqnarray*}
Soit $ n \geq 2 $ et supposons le résultat vérifié si le degré de $ F_{k} $ est strictement inferieur à $ n $. $ F_{k} \in \mathcal{T} \cap \mathcal{B}_{l} $, donc $ F_{k} $ est de la forme $ F_{k} = B^{-}(H_{1}, \hdots , H_{l}) $ avec $ l \geq 1 $ et $ H_{1}, \hdots , H_{l} \in \mathcal{T} \cap \mathcal{B}_{l} $. Alors, en utilisant l'hypothèse de récurrence sur la forêt $ H = H_{1} \hdots H_{l} $ :
\begin{eqnarray*}
\phi (F \nwarrow G) & = & \phi (F_{1} \hdots F_{k-1} B^{-}(H_{1}, \hdots , H_{l}) \nwarrow G)\\
& = & \phi (F_{1} \hdots F_{k-1}) \phi(B^{-}(H_{1}, \hdots , H_{l} \nwarrow G)) \\
& = & \phi (F_{1} \hdots F_{k-1}) (a \nwarrow \phi(H \nwarrow G) )\\
& = & \phi (F_{1}) \hdots \phi(F_{k-1}) (a \nwarrow (\phi(H) \nwarrow \phi(G)))\\
& = & \phi (F_{1}) \hdots \phi(F_{k-1}) ((a \nwarrow \phi(H)) \nwarrow \phi(G))\\
& = & \phi (F_{1}) \hdots \phi(F_{k-1}) (\phi(F_{k}) \nwarrow \phi(G))\\
& = & (\phi (F_{1}) \hdots \phi(F_{k-1}) \phi(F_{k})) \nwarrow \phi(G)\\
& = & \phi (F) \nwarrow \phi(G)).
\end{eqnarray*}
Ainsi $ \phi $ est un morphisme d'algèbres dupliciales.\\

Soit $ \phi': (\mathcal{B}_{l})_{+} \rightarrow A $ un deuxième morphisme d'algèbres dupliciales tel que $ \phi'(\tdun{$1$}) = a $. Pour tout arbre non vide $ T_{1}, \hdots ,T_{k} \in \mathcal{T} \cap \mathcal{B}_{l} $, $ \phi'(T_{1} \hdots T_{k}) = \phi'(T_{1}) \hdots \phi'(T_{k}) $ et $ \phi'( B^{-}(T_{1}, \hdots ,T_{k})) = \phi'(\tdun{$1$} \nwarrow (T_{1} \hdots T_{k})) = a \nwarrow \phi'(T_{1} \hdots T_{k}) $. Donc $ \phi = \phi' $.
\end{proof}

\begin{defi}
Pour toute forêt non vide $ F \in (\mathcal{B})_{+} $, notons $ r_{F} $ la feuille la plus à droite de $ F $. On pose:
\begin{eqnarray*}
\tdelta_{\prec} (F) & = & \sum_{\boldsymbol{v} \mmodels V(F) {\rm ~ et ~} r_{F} \in Lea_{\boldsymbol{v}} (F)} Lea_{\boldsymbol{v}} (F) \otimes Roo_{\boldsymbol{v}} (F),\\
\tdelta_{\succ} (F) & = & \sum_{\boldsymbol{v} \mmodels V(F) {\rm ~ et ~} r_{F} \in Roo_{\boldsymbol{v}} (F)} Lea_{\boldsymbol{v}} (F) \otimes Roo_{\boldsymbol{v}} (F).
\end{eqnarray*}
\end{defi}

Remarquons que $ \tdelta_{\prec} + \tdelta_{\succ} = \tdelta $.

\begin{lemma}
Pour tout $ F \in (\mathcal{B})_{+} $,
\begin{equation} \label{E2} \left\{\begin{array}{rcl}
(\tdelta_{\prec} \otimes Id) \circ \tdelta_{\prec} (F) & = & (Id \otimes \tdelta) \circ \tdelta_{\prec} (F),\\
(\tdelta_{\succ} \otimes Id) \circ \tdelta_{\prec} (F) & = & (Id \otimes \tdelta_{\succ}) \circ \tdelta_{\prec} (F),\\
(\tdelta \otimes Id) \circ \tdelta_{\succ} (F) & = & (Id \otimes \tdelta_{\succ}) \circ \tdelta_{\succ} (F).
\end{array}\right. \end{equation}
En d'autres termes, $ (\mathcal{B})_{+} $ est une coalgèbre dendriforme.
\end{lemma}

\begin{proof}
Soit $ F $ une forêt non vide de $ \mathcal{B} $. Par coassociativité de $ \tdelta $, on peut poser:
$$ (\tdelta \otimes Id) \circ \tdelta (F) = (Id \otimes \tdelta) \circ \tdelta (F) = \sum F^{(1)} \otimes F^{(2)} \otimes F^{(3)} ,$$
où $ F^{(1)} , F^{(2)} , F^{(3)} $ sont des sous-forêts non vides de $ F $. Alors:
\begin{eqnarray*}
(\tdelta_{\prec} \otimes Id) \circ \tdelta_{\prec} (F) = (Id \otimes \tdelta) \circ \tdelta_{\prec} (F) =  \sum_{r_{F} \in F^{(1)}} F^{(1)} \otimes F^{(2)} \otimes F^{(3)},\\
(\tdelta_{\succ} \otimes Id) \circ \tdelta_{\prec} (F) = (Id \otimes \tdelta_{\succ}) \circ \tdelta_{\prec} (F) = \sum_{r_{F} \in F^{(2)}} F^{(1)} \otimes F^{(2)} \otimes F^{(3)},\\
(\tdelta \otimes Id) \circ \tdelta_{\succ} (F) = (Id \otimes \tdelta_{\succ}) \circ \tdelta_{\succ} (F) = \sum_{r_{F} \in F^{(3)}} F^{(1)} \otimes F^{(2)} \otimes F^{(3)}.
\end{eqnarray*}
\end{proof}

{\bf Notations.} {Soit $ (A, \tdelta_{\prec} , \tdelta_{\succ} )$ une coalgèbre dendriforme.
\begin{enumerate}
\item On notera $ Prim_{tot}(A) = Ker(\tdelta_{\succ}) \cap Ker(\tdelta_{\prec}) $.
\item Nous utiliserons les notations de Sweedler suivantes: pour tout $ a \in A $, $ \tdelta (a) = a' \otimes a'' $, $ \tdelta_{\prec} (a) = a'_{\prec} \otimes a''_{\prec} $ et $ \tdelta_{\succ} (a) = a'_{\succ} \otimes a''_{\succ} $.
\end{enumerate}
}

\begin{prop} \label{3}
\begin{enumerate}
\item Soit $ x,y \in (\mathcal{B})_{+} $. Alors:
\begin{equation}\label{E3}\left\{\begin{array}{rcl}
\tdelta_{\prec} (xy) & = & y \otimes x + x'y \otimes x'' + xy'_{\prec} \otimes y''_{\prec} + y'_{\prec} \otimes xy''_{\prec} + x'y'_{\prec} \otimes x''y''_{\prec},\\
\tdelta_{\succ} (xy) & = & x \otimes y + x' \otimes x''y + xy'_{\succ} \otimes y''_{\succ} + y'_{\succ} \otimes xy''_{\succ} + x'y'_{\succ} \otimes x''y''_{\succ}.
\end{array}\right. \end{equation}
En d'autres termes, $ (\mathcal{B})_{+} $ est une bialgèbre codendriforme.
\item Soit $ x,y \in (\mathcal{B})_{+} $. Alors:
\begin{equation}\label{E4}\left\{\begin{array}{rcl}
\tdelta_{\prec} (x \nwarrow y) & = & y \otimes x + y'_{\prec} \otimes x \nwarrow y''_{\prec} + x'_{\prec} \nwarrow y \otimes x''_{\prec} + x'_{\succ} y \otimes x_{\succ}\\
& \hspace{0.3cm} & + x'_{\succ}y'_{\prec} \otimes x''_{\succ} \nwarrow y''_{\prec},\\
\tdelta_{\succ} (x \nwarrow y) & = & y'_{\succ} \otimes x \nwarrow y''_{\succ} + x'_{\succ} \otimes x''_{\succ} \nwarrow y + x'_{\succ}y'_{\succ} \otimes x''_{\succ} \nwarrow y''_{\succ}.
\end{array}\right. \end{equation}
\end{enumerate}
\end{prop}

\begin{proof}
Il suffit de prouver ces formules pour $x=F$, $y=G$ des forêts non vides $ \in \mathcal{T} $. Commençons par prouver les premières formules, à savoir $\tdelta_\prec(FG)$ et $\tdelta_\succ(FG)$.
Pour toute coupe admissible $\boldsymbol{v} \mmodels V(FG)$, soit $\boldsymbol{v}'$ la restriction de $\boldsymbol{v}$ à $F$ et $\boldsymbol{v}''$ la restriction de $\boldsymbol{v}$ à $G$. 
Alors $\boldsymbol{v}'\models V(F)$ et $\boldsymbol{v}'' \models V(G)$. Par ailleurs, $\boldsymbol{v}'$ et $\boldsymbol{v}''$ ne sont pas simultanément totales, ni simultanément vides. \\

Pour $\tdelta_\prec(FG)$: soit $\boldsymbol{v} \mmodels V(FG)$, telle que $r_{FG}=r_G$ appartient à $Lea_{\boldsymbol{v}}(FG)$.
Comme $r_G \in Lea_{\boldsymbol{v}''}(G)$, $\boldsymbol{v}''$ est non vide. On a alors cinq possibilités pour $\boldsymbol{v}$:
\begin{itemize}
\item $\boldsymbol{v}'$ est vide et $\boldsymbol{v}''$ est totale: cela donne le terme $G\otimes F$. 
\item $\boldsymbol{v}'$ est non vide et $\boldsymbol{v}''$ est totale: alors $\boldsymbol{v}' \mmodels V(F)$, et cela donne le terme $F'G\otimes F''$.
\item $\boldsymbol{v}'$ est vide et $\boldsymbol{v}''$ est non totale: comme $r_G \in Lea_{\boldsymbol{v}''}(G)$, cela donne le terme $G'_\prec \otimes F G''_\prec$.
\item $\boldsymbol{v}'$ est totale et $\boldsymbol{v}''$ est non totale: comme $r_G \in Lea_{\boldsymbol{v}''}(G)$, cela donne le terme $FG'_\prec \otimes G''_\prec$.
\item $\boldsymbol{v}' \mmodels V(F)$ et $\boldsymbol{v}''$ sont non totales: comme $r_G \in Lea_{\boldsymbol{v}''}(G)$, cela donne le terme $F'G'_\prec \otimes F''G''_\prec$.
\end{itemize}

\hspace{1cm}

Intérressons nous à présent à $\tdelta_\succ(FG)$. Soit $\boldsymbol{v} \mmodels V(FG)$, telle que $r_{FG}=r_G$ appartient à $Roo_{\boldsymbol{v}}(FG)$.
Comme $r_G \in Roo_{\boldsymbol{v}''}(G)$, $\boldsymbol{v}''$ n'est pas totale. Il y a cinq possibilités pour $\boldsymbol{v}$:
 \begin{itemize}
\item $\boldsymbol{v}'$ est totale et $\boldsymbol{v}''$ est vide: cela donne le terme $F\otimes G$. 
\item $\boldsymbol{v}'$ est non totale et $\boldsymbol{v}''$ est vide: alors $\boldsymbol{v}' \mmodels V(F)$, et cela donne le terme $F'\otimes F''G$.
\item $\boldsymbol{v}'$ est totale et $\boldsymbol{v}''$ est non vide: comme $r_G \in Roo_{\boldsymbol{v}''}(G)$, cela donne le terme $FG'_\succ\otimes G''_\succ$.
\item $\boldsymbol{v}'$ est vide et $\boldsymbol{v}''$ est non totale: comme $r_G \in Roo_{\boldsymbol{v}''}(G)$, cela donne le terme $G'_\succ \otimes FG''_\succ$.
\item $\boldsymbol{v}' \mmodels V(F)$ et $\boldsymbol{v}''$ sont non totales: comme $r_G \in Roo_{\boldsymbol{v}''}(G)$, cela donne le terme $F'G'_\succ \otimes F''G''_\succ$.
\end{itemize}

\hspace{1cm}

Pour une coupe admissible $\boldsymbol{v}\mmodels V(F\nwarrow G)$, soit $\boldsymbol{v}'$ la restriction à $\boldsymbol{v}$ de $F$ et soit $\boldsymbol{v}''$ l'unique coupe admissible de $G$ telle que $Lea_{\boldsymbol{v}''}(G)$ est la sous-forêt de $Lea_{\boldsymbol{v}} (F\nwarrow G)$ formée par les sommets qui appartiennent à $V(G)$. Remarquons que, $\boldsymbol{v}$ n'étant pas totale, $\boldsymbol{v}'$ n'est pas totale.\\

Calculons $\tdelta_\prec(F\nwarrow G)$.  Soit $\boldsymbol{v} \mmodels V(F\nwarrow G)$, telle que $r_{F\nwarrow G}=r_G$ appartient à $Lea_{\boldsymbol{v}}(F\nwarrow G)$. Comme $r_G \in Lea_{\boldsymbol{v}''}(G)$, $\boldsymbol{v}''$ est non vide. Il y a quatre possibilités pour $\boldsymbol{v}$:
\begin{itemize}
\item $\boldsymbol{v}'$ est vide et $\boldsymbol{v}''$ est totale: cela donne le terme $G \otimes F$.
\item $\boldsymbol{v}'$ est vide et $\boldsymbol{v}''$ est non totale: alors $\boldsymbol{v}'' \mmodels V(G)$ et $r_G \in Lea_{\boldsymbol{v}''}(G)$, et cela donne le terme $G'_\prec \otimes F \nwarrow G''_\prec$.
\item $\boldsymbol{v}'$ est non vide et $\boldsymbol{v}''$ est totale: on a alors deux sous-cas:
\begin{itemize}
\item $Lea_{\boldsymbol{v}'}(F)$ contient $r_F$: cela donne le terme $F'_\prec \nwarrow G \otimes F''_\prec$.
\item $Roo_{\boldsymbol{v}'}(F)$ contient $r_F$: cela donne le terme $F'_\succ G \otimes F''_\succ$.
\end{itemize}
\item $\boldsymbol{v}'$ est non vide et $\boldsymbol{v}''$ est non totale: alors $r_F$ n'appartient pas à $Lea_{\boldsymbol{v}'}(F)$, $r_G$ appartient à $Lea_{\boldsymbol{v}''}(G)$, et cela donne le terme $F'_\succ G'_\prec \otimes F''_\succ \nwarrow G''_\prec$.
\end{itemize}

\hspace{1cm}

Pour terminer, il reste à calculer $\tdelta_\succ(F\nwarrow G)$.  Soit $\boldsymbol{v} \mmodels V(F\nwarrow G)$, telle que $r_{F\nwarrow G}=r_G$ appartient à $Roo_{\boldsymbol{v}}(F\nwarrow G) $. Comme $r_G \in Roo_{\boldsymbol{v}''}(G)$, $\boldsymbol{v}''$ n'est pas totale. Donc $\boldsymbol{v}'$ ne contient pas $r_F$. Il y a trois possibilités pour $\boldsymbol{v}$:
\begin{itemize}
\item $\boldsymbol{v}'$ est vide: alors $\boldsymbol{v}'' \mmodels V(G)$ et $Roo_{\boldsymbol{v}''} (G)$ contient $r_G$, et cela donne le terme $G'_\succ \otimes F\nwarrow G''_\succ$.
\item $\boldsymbol{v}'$ est non vide et $\boldsymbol{v}''$ est vide: alors $r_F \in Roo_{\boldsymbol{v}'}(F)$ et on obtient le terme $F'_\succ \otimes F''_\succ \nwarrow G$.
\item $\boldsymbol{v}'$ est non vide et $\boldsymbol{v}''$ est non vide:  alors $r_F \in Roo_{\boldsymbol{v}'}(F)$, $r_G \in Roo_{\boldsymbol{v}''}(G)$ et on obtient le terme $F'_\succ G'_\succ \otimes F''_\succ \nwarrow G'_\succ$.
\end{itemize}
\end{proof}

Rappelons la définition suivante suggèrée par les résultats précédents:

\begin{defi} Une bialgèbre dupliciale dendriforme est une famille $(A,\ast,\nwarrow,\tdelta_\prec,\tdelta_\succ)$, où $A$ est un espace vectoriel,
$\ast,\nwarrow:A\otimes A\longrightarrow A$ et $\tdelta_\prec,\tdelta_\succ:A \longrightarrow A \otimes A$, avec les propriétés suivantes:
\begin{enumerate}
\item $(A,\ast,\nwarrow)$ est une algèbre dupliciale (axiomes (\ref{E1})).
\item $(A,\tdelta_\prec,\tdelta_\succ)$ est une coalgèbre dendriforme (axiomes (\ref{E2})).
\item Les compatibilités de la proposition \ref{3} sont satisfaites (axiomes (\ref{E3}) et (\ref{E4})).
\end{enumerate} \end{defi}

Nous avons besoin du théorème de rigidité suivant, démontré dans \cite{Foissy4},

\begin{theo}
Soit $A$ une bialgèbre dupliciale dendriforme. On suppose que $A$ est graduée et connexe, c'est-à-dire que $A_0=(0)$.
Soit $(p_\D)_{d\in \D}$ une base de $Prim_{tot}(A)$ formée par des éléments homogènes, indexés par un ensemble gradué $\D$. Il existe un unique isomorphisme de bialgèbres dupliciales dendriformes graduées:
$$\phi:\left\{\begin{array}{rcl}
(\mathcal{H}_{PR}^\D)_+&\longrightarrow&A\\
\tdun{$d$},\:d\in \D&\longrightarrow&p_d.
\end{array}\right.$$
\end{theo}

On en déduit alors le résultat suivant:

\begin{theo}
Il existe un ensemble gradué $ \mathcal{D} $ tel que $ (\mathcal{B})_{+} $ est isomorphe à $ (\mathcal{H}_{PR}^\D)_+ $ comme bialgèbres dupliciales dendriformes graduées.
\end{theo}

La série formelle de $ \mathcal{D} $ est donnée par:
$$ f_{\D}(x)=\frac{f_{\mathcal{B}}(x)-1}{f_{\mathcal{B}}(x)^2} .$$

Cela donne:

$$\begin{array}{c|c|c|c|c|c|c|c|c}
k&1&2&3&4&5&6&7&8\\
\hline \left|\mathcal{D} \right| &1&1&2&6&22&90&394&1806
\end{array}$$

On retrouve les nombres de Schroeder, correspondants à la séquence A006318 de \cite{Sloane}.\\

{\bf Remarque.} {Comme corollaire de ce théorème, $ \mathcal{B} $ est libre (ce qu'on savait déjà, librement engendré par $ \mathcal{T} $), colibre et auto-duale.}

\subsection{Structure d'algèbre bigreffe de $ \mathcal{B} $}

\subsubsection{Algèbre de greffes à gauche}

Définissons à présent une nouvelle loi de composition interne sur $ (\mathcal{B})_{+} $ qui va permettre de munir $ \mathcal{B} $ d'une structure d'algèbre de greffes à gauche unitaire.\\

\'Etant donnés deux arbres non vides $ T $ et $ G $ de $ \mathcal{T} $, on définit un arbre $ T \succ G $ en greffant par la gauche $ T $ sur la racine de $ G $ et en indexant les sommets comme suit: on conserve l'indexation des sommets de $ T $ et on numérote ensuite les sommets de $ G $ dans leurs ordres de départ mais en décalant leurs indices par le nombre de sommets de $ T $. Considérons maintenant une forêt $ T_{1} \hdots T_{n} $ et un arbre $ G $, avec $ n \geq 1 $ et $ T_{1}, \hdots , T_{n}, G \in \mathcal{T} $ (tous non vides), on définit l'arbre $ (T_{1} \hdots T_{n}) \succ G $ en le posant égal à $ T_{1} \succ ( T_{2} \succ ( \hdots (T_{n} \succ G) \hdots )) $. \'Etant données deux forêts non vides $ T_{1} \hdots T_{n} $ et $ G_{1} \hdots G_{m} $, avec $ n,m \geq 1 $ et $ T_{1}, \hdots ,T_{n} ,G_{1}, \hdots, G_{m} \in \mathcal{T} $, on pose $ (T_{1} \hdots T_{n}) \succ (G_{1} \hdots G_{m}) = ((T_{1} \hdots T_{n}) \succ G_{1} ) G_{2} \hdots G_{m} $. En étendant par linéarité $ \succ $, on définit ainsi une nouvelle opération sur $ (\mathcal{B})_{+} $. On utilisera la convention suivante: si $ T \in (\mathcal{B})_{+} $, $ 1 \succ T = T $ et $ T \succ 1 = 0 $.\\

{\bf Exemples.} Illustrons ci-dessous l'opération $ \succ $:
$$\begin{array}{|rclcl|rclcl|rclcl|}
\hline \tdun{1} &\succ& \tddeux{1}{2} &=& \tdtroisun{2}{3}{1} &\tdun{1}\tdun{2} &\succ& \tdun{1} &=& \tdtroisun{3}{2}{1} &
\tdun{1}&\succ&\tdun{1}\tddeux{2}{3}&=&\tddeux{2}{1} \tddeux{3}{4}\\
\tddeux{2}{1} &\succ& \tddeux{1}{2} \tdun{3} &=&\tdquatredeux{3}{4}{2}{1} \tdun{5} & \tdun{1} \tdun{2} \tdun{3} &\succ& \tdun{1} &=& \tdquatreun{4}{3}{2}{1} &
\tdun{1} &\succ& \tdtroisdeux{1}{3}{2} &=& \tdquatretrois{2}{4}{3}{1} \\
\tdtroisun{2}{3}{1} &\succ& \tdun{1}&=& \tdquatrequatre{4}{2}{3}{1} & \tdun{1} \tddeux{3}{2} &\succ& \tdun{1} &=& \tdquatretrois{4}{3}{2}{1} &
\tdun{1}\tdun{2}&\succ& \tdtroisdeux{1}{3}{2} &=&\tdcinqquatre{3}{5}{4}{2}{1}\\
\hline \end{array}$$

Montrons que $ ( \mathcal{B})_{+} $ est bien stable pour la loi $ \succ $. Il suffit de voir que si $ T $ et $ G $ sont deux arbres $ \in \mathcal{T} $, l'arbre $ T \succ G $ définit précédemment est encore un élément de $ \mathcal{T} $. Pour cela, notons $ G_{1}, \hdots ,G_{n} $ la suite des sous-arbres issus de la racine de $ G $, et $ G_{i,1}, \hdots ,G_{i,m_{i}} $ la suite des sous-arbres issus de la racine de $ G_{i} $, pour tout $ i \in \left\lbrace 1, \hdots, n \right\rbrace $. Il y a alors deux cas:
\begin{enumerate}
\item Si $ G = B^{+}(\hdots B^{+}(\tdun{1}, G_{1,1} , \hdots , G_{1,m_{1}}) \hdots), G_{n,1} , \hdots, G_{n,m_{n}}) $, alors par définition de $ \succ $,
$$ T \succ G = B^{+}(\hdots B^{+}(B^{-}(T), G_{1,1} , \hdots , G_{1,m_{1}}), \hdots ), G_{n,1} , \hdots, G_{n,m_{n}}) $$
et ceci est bien un élément de $ \mathcal{T} $.
\item Si il existe un $ 1 \leq i \leq n $ tel que
$$ G = \overbrace{B^{+}(\hdots B^{+}}^{n-i ~ {\rm fois}}(B^{-}(G_{1} , \hdots , G_{i}), G_{i+1,1} , \hdots , G_{i+1,m_{i+1}}), \hdots ), G_{n,1} , \hdots, G_{n,m_{n}}) ,$$
alors par définition de $ \succ $,
$$ T \succ G = \overbrace{B^{+}(\hdots B^{+}}^{n-i ~ {\rm fois}}(B^{-}(T,G_{1} , \hdots , G_{i}), G_{i+1,1} , \hdots , G_{i+1,m_{i+1}}), \hdots ), G_{n,1} , \hdots, G_{n,m_{n}}) ,$$
et ceci est ici encore un élément de $ \mathcal{T} $.
\end{enumerate}
Ainsi $ (\mathcal{B})_{+} $ est stable pour l'opération $ \succ $.\\

{\bf Remarque.} {Pour toute forêt non vide $ T_{1} \hdots T_{n} \in (\mathcal{B})_{+} $, $ (T_{1} \hdots T_{n}) \succ \tdun{1} = B^{-}(T_{1}, \hdots ,T_{n}) $.}\\

La propriété suivante vient directement de la définition de $ \succ $:

\begin{lemma} \label{gauche}
\'Etant donné $ F,G,H \in (\mathcal{B})_{+} $, 
\begin{eqnarray}
(FG) \succ H & = & F \succ (G \succ H), \label{dipt}\\
(F \succ G) H & = & F \succ (GH) \label{dipt2}.
\end{eqnarray}
\end{lemma}

{\bf Remarque.} {L'opération $ \succ $ n'est pas associative. Par exemple,
$$ \begin{array}{rclcl}
\tdun{1} \succ (\tdun{1} \succ \tdun{1}) & = & \tdun{1} \succ \tddeux{2}{1} & = & \tdtroisun{3}{2}{1},\\
(\tdun{1} \succ \tdun{1}) \succ \tdun{1} & = & \tddeux{2}{1} \succ \tdun{1} & = & \tdtroisdeux{1}{2}{3}.
\end{array} $$
}

\begin{defi}
Une algèbre de greffes à gauche est un espace vectoriel $ A $ muni de deux opérations $ \ast $ et $ \succ $ satisfaisant les deux relations suivantes : $ \forall ~ x,y,z \in A $
\begin{eqnarray*}
(x \ast y) \ast z & = & x \ast (y \ast z),\\
(x \ast y) \succ z & = & x \succ (y \succ z),\\
(x \succ y) \ast z & = & x \succ (y \ast z).
\end{eqnarray*}
Une algèbre de greffes à gauche unitaire $ A $ est un espace vectoriel $ A = \mathbb{K}1 \oplus \overline{A} $ tel que $ \overline{A} $ est une algèbre de greffes à gauche et où on a étendu les deux opérations $ \ast $ et $ \succ $ comme suit:
\begin{equation}\label{tens2}\begin{array}{rcl}
1 \ast a = a \ast 1 = a, {\it ~ pour ~ tout ~} a \in A, \\
1 \succ a = a, ~ a \succ 1 = 0, {\it ~ pour ~ tout ~} a \in \overline{A}.
\end{array}\end{equation}
\end{defi}

Remarquons que $ 1 \succ 1 $ n'est pas défini. Si $ A = \mathbb{K}1 \oplus \overline{A} $ et $ B = \mathbb{K}1 \oplus \overline{B} $ sont deux algèbres de greffes à gauches unitaires, on peut étendre les deux opérations $ \ast $ et $ \succ $ à leur produit tensoriel $ A \otimes B $ en posant:
\begin{equation}\label{tens1}\left\{\begin{array}{rcl}
(a \otimes b) \ast (a' \otimes b') & := & (a \ast a') \otimes (b \ast b'),\\
(a \otimes b) \succ (a' \otimes b') & := & (a \ast a') \otimes (b \succ b'), {\rm ~ si ~} b \otimes b' \neq 1 \otimes 1,\\
(a \otimes 1) \succ (a' \otimes 1) & := & (a \succ a') \otimes 1,
\end{array}\right. \end{equation}
pour $ a,a' \in A $, $ b,b' \in B $.\\

On définira dans la section \ref{greffeadroite} la notion d'algèbre de greffes à droite.\\

Grâce au lemme \ref{gauche}, $ \mathcal{B} $ muni de la concaténation et de $ \succ $ est une algèbre de greffes à gauche unitaire, où $ \overline{\mathcal{B}} = (\mathcal{B})_{+} $ et $ 1 $ est l'arbre vide, de tel sorte que $ \mathcal{B} = \mathbb{K}1 \oplus \overline{\mathcal{B}} $.

\begin{theo} \label{libregauche}
$ (\mathcal{B}_{l})_{+} $ est l'algèbre de greffes à gauche libre engendrée par l'élément $ \tdun{1} $.
\end{theo}

\begin{proof}
Soit $ A $ une algèbre de greffes à gauche et soit $ a \in A $. Il faut montrer qu'il existe un unique morphisme d'algèbres de greffes à gauche $ \phi : (\mathcal{B}_{l})_{+} \rightarrow A $ tel que $ \phi (\tdun{$1$}) = a $. On définit $ \phi (F) $ pour toute forêt non vide $ F \in (\mathcal{B}_{l})_{+} $ inductivement sur le degré de $ F $:
\begin{eqnarray*}
\phi (\tdun{$1$}) & = & a,\\
\phi (T_{1} \hdots T_{k}) & = & \phi (T_{1}) \hdots \phi (T_{k})  ~ {\rm si} ~ k \geq 2,\\
\phi( B^{-}(T_{1}, \hdots ,T_{k})) & = & \phi(T_{1} \hdots T_{k}) \succ a ~ {\rm si} ~ k \geq 1.
\end{eqnarray*}
Comme le produit $ \ast $ dans $ A $ est associatif, $ \phi $ est bien définie. $ \phi $ s'étend par linéarité en une application $ \phi : (\mathcal{B})_{+} \rightarrow A $. Montrons que c'est un morphisme d'algèbres de greffes à gauche. Par le second point, $ \phi(FG)= \phi(F) \phi(G) $ pour toutes forêts $ F,G \in (\mathcal{B}_{l})_{+} $. Considérons deux forêts non vides $ F $ et $ G $. Il faut prouver que $ \phi (F \succ G) = \phi(F) \succ \phi(G) $. On travaille par induction sur le degré $ n $ de $ G $. Si $ n=1 $, $ G = \tdun{1} $, et :
$$ \phi (F \succ G) = \phi (B^{-}(F)) = \phi(F) \succ a = \phi (F) \succ \phi (G) .$$
Supposons le résultat vérifié pour toutes forêts de degré $ < n $. Considérons alors $ G \in (\mathcal{B}_{l})_{+} $ une forêt de degré $ n \geq 2 $ et une forêt non vide $ F=F_{1} \hdots F_{m} \in (\mathcal{B}_{l})_{+} $. Notons $ k $ la longueur de $ G $. Il y a deux cas suivant la longueur $ k $ de $ G $:
\begin{enumerate}
\item Si $ k \geq 2 $, $ G=G_{1} \hdots G_{k} $. Alors
\begin{eqnarray*}
\phi(F \succ G) & = & \phi((F_{1} \hdots F_{m}) \succ (G_{1} \hdots G_{k}))\\
& = & \phi(((F_{1} \hdots F_{m}) \succ G_{1}) G_{2} \hdots G_{k})\\
& = & \phi((F_{1} \hdots F_{m}) \succ G_{1}) \phi(G_{2}) \hdots \phi(G_{k})\\
& = & (\phi(F_{1} \hdots F_{m}) \succ \phi(G_{1})) \phi(G_{2}) \hdots \phi(G_{k})\\
& = & \phi(F) \succ (\phi(G_{1}) \phi(G_{2}) \hdots \phi(G_{k}))\\
& = & \phi(F) \succ \phi(G),
\end{eqnarray*}
en utilisant l'hypothèse de récurrence à la quatrième égalité.
\item Si $ k = 1 $, $ G $ est un arbre de degré $ n \geq 2 $, donc $ G = B^{-}(G_{1}, \hdots ,G_{l}) $, avec $ G_{1}, \hdots ,G_{l} \in (\mathcal{B}_{l})_{+} $ et $ l \geq 1 $. Alors
\begin{eqnarray*}
\phi(F \succ G) & = & \phi((F_{1} \hdots F_{m}) \succ B^{-}(G_{1}, \hdots ,G_{l}))\\
& = & \phi(B^{-}(F_{1}, \hdots ,F_{m},G_{1}, \hdots ,G_{l}))\\
& = & \phi(F_{1} \hdots F_{m} G_{1} \hdots G_{l}) \succ a\\
& = & (\phi(F_{1} \hdots F_{m}) \phi(G_{1} \hdots G_{l})) \succ a\\
& = & \phi(F) \succ (\phi(G_{1} \hdots G_{l}) \succ a)\\
& = & \phi(F) \succ \phi(B^{-}(G_{1}, \hdots ,G_{l}))\\
& = & \phi(F) \succ \phi(G).
\end{eqnarray*}
\end{enumerate}
Ainsi, dans tous les cas, $ \phi(F \succ G) = \phi(F) \succ \phi(G) $. Par le principe de récurrence, cette formule est donc démontrée pour toute forêt $ F, G \in (\mathcal{B}_{l})_{+} $.\\

Soit $ \phi': (\mathcal{B}_{l})_{+} \rightarrow A $ un deuxième morphisme d'algèbres de greffes à gauche tel que $ \phi'(\tdun{$1$}) = a $. Soient $ k \geq 1 $ et $ T_{1}, \hdots ,T_{k} \in \mathcal{G} \cap \mathcal{B}_{l} $. Alors $ \phi'(T_{1} \hdots T_{k}) = \phi'(T_{1}) \hdots \phi'(T_{k}) $. De plus,
\begin{eqnarray*}
\phi'(B^{-}(T_{1}, \hdots ,T_{k})) & = & \phi'((T_{1} \hdots T_{k}) \succ \tdun{1})\\
& = & \phi'(T_{1} \hdots T_{k}) \succ \phi'(\tdun{1})\\
& = & \phi'(T_{1} \hdots T_{k}) \succ a.
\end{eqnarray*}
Donc $ \phi' = \phi $ et ceci termine la démonstration.
\end{proof}

\vspace{0.5cm}

Il existe une relation entre $ \succ $ et le coproduit :

\begin{prop} \label{deltasucc}
Rappelons que: $ \Delta_{\boldsymbol{l}} (F) = \displaystyle\sum_{\boldsymbol{v} \models V(F) {\rm ~ et ~} roo_{\boldsymbol{l}}(F) \in Roo_{\boldsymbol{v}} (F)} Lea_{\boldsymbol{v}} (F) \otimes Roo_{\boldsymbol{v}} (F) $ pour toute forêt non vide $ F \in (\mathcal{B})_{+} $, où $ roo_{\boldsymbol{l}}(F) $ est la racine de l'arbre le plus à gauche de la forêt $ F $. Alors pour tout $ F \in \mathcal{B} $ et $ G \in (\mathcal{B})_{+} $,
\begin{equation}
\Delta_{\boldsymbol{l}} (F \succ G) = \Delta(F) \succ \Delta_{\boldsymbol{l}} (G).
\end{equation}
De plus: $ ( \Delta \otimes Id) \circ \Delta_{\boldsymbol{l}} (G) = ( Id \otimes \Delta_{\boldsymbol{l}}) \circ \Delta_{\boldsymbol{l}} (G). $
\end{prop}

\begin{proof}
Pour démontrer la première égalité, on utilisera les notations suivantes: si $ a \in \mathcal{B} $, $ \Delta (a) = a \otimes 1 + 1 \otimes a + a' \otimes a'' $, et si $ a \in (\mathcal{B})_{+} $, $ \Delta_{\boldsymbol{l}} (a) = 1 \otimes a + a'_{\boldsymbol{l}} \otimes a''_{\boldsymbol{l}} $. Alors, pour tout $ x \in \mathcal{B} $ et $ y \in (\mathcal{B})_{+} $,
\begin{eqnarray*}
\Delta(x) \succ \Delta_{\boldsymbol{l}} (y) & = & (x \otimes 1 + 1 \otimes x + x' \otimes x'') \succ (1 \otimes y + y'_{\boldsymbol{l}} \otimes y''_{\boldsymbol{l}}) \\
& = & x \otimes (1 \succ y) + xy'_{\boldsymbol{l}} \otimes (1 \succ y''_{\boldsymbol{l}}) + 1 \otimes (x \succ y) + y'_{\boldsymbol{l}} \otimes (x \succ y''_{\boldsymbol{l}})\\
& \hspace{0.5cm} & + x' \otimes (x'' \succ y) + x'y'_{\boldsymbol{l}} \otimes (x'' \succ y''_{\boldsymbol{l}}) \\
& = & x \otimes y + xy'_{\boldsymbol{l}} \otimes y''_{\boldsymbol{l}} + 1 \otimes (x \succ y) + y'_{\boldsymbol{l}} \otimes (x \succ y''_{\boldsymbol{l}}) + x' \otimes (x'' \succ y)\\
& \hspace{0.5cm} & + x'y'_{\boldsymbol{l}} \otimes (x'' \succ y''_{\boldsymbol{l}}) \\
& = & \Delta_{\boldsymbol{l}} (x \succ y),
\end{eqnarray*}
en utilisant (\ref{tens1}) à la deuxième égalité et (\ref{tens2}) à la troisième.\\

Pour la deuxième égalité, on peut poser, pour toute forêt $ x \in (\mathcal{B})_{+} $, par coassociativité de $ \Delta $:
$$ (\Delta \otimes Id) \circ \Delta (x) = (Id \otimes \Delta) \circ \Delta (x) = \sum x^{(1)} \otimes x^{(2)} \otimes x^{(3)} ,$$
où $ x^{(1)} , x^{(2)} , x^{(3)} $ sont des sous-forêts de $ x $. Alors:
\begin{eqnarray*}
(\Delta \otimes Id) \circ \Delta_{\boldsymbol{l}} (x) = (Id \otimes \Delta_{\boldsymbol{l}}) \circ \Delta_{\boldsymbol{l}} (x) = \sum_{roo_{\boldsymbol{l}}(x) \in x^{(3)}} x^{(1)} \otimes x^{(2)} \otimes x^{(3)}.
\end{eqnarray*}
Le résultat suit par linéarité.
\end{proof}

\subsubsection{Algèbre de greffes à droite} \label{greffeadroite}

On peut aussi définir une opération $ \prec : (\mathcal{B})_{+} \times (\mathcal{B})_{+} \rightarrow (\mathcal{B})_{+} $. Cela va permettre de munir $ \mathcal{B} $ d'une structure d'algèbre de greffes à droite unitaire.\\

 \'Etant donnés deux arbres non vides $ T $ et $ G $ appartenant à $ \mathcal{T} $, on définit un arbre $ T \prec G $ en greffant par la droite $ G $ sur la racine de $ T $ et en indexant les sommets comme suit: on conserve l'indexation des sommets de $ T $, puis on numérote les sommets de $ G $ différents de la racine de $ G $ en conservant l'ordre initial, et on termine en numérotant la racine de $ G $ (par $ \left| T \right| + \left| G \right| $). Considérons maintenant un arbre $ T $ et une forêt $ G_{1} \hdots G_{m} $, avec $ T,G_{1}, \hdots ,G_{m} \in \mathcal{T} $, on définit l'arbre $ T \prec (G_{1} \hdots G_{m}) $ en le posant égal à $ ( \hdots ((T \prec G_{1}) \prec G_{2}) \hdots \prec G_{m}) $. \'Etant données deux forêts non vides $ T_{1} \hdots T_{n} $ et $ G_{1} \hdots G_{m} $, avec $ n,m \geq 1 $ et $ T_{1}, \hdots ,T_{n} ,G_{1}, \hdots, G_{m} \in \mathcal{T} $, on pose $ (T_{1} \hdots T_{n}) \prec (G_{1} \hdots G_{m}) = T_{1} \hdots T_{n-1} (T_{n} \prec (G_{1}  \hdots G_{m})) $. En étendant par linéarité $ \prec $, ceci définit une nouvelle opération sur $ (\mathcal{B})_{+} $. Si $ T \in (\mathcal{B})_{+} $, on pose $ T \prec 1 = T $ et $ 1 \prec T = 0 $.\\

{\bf Exemples.} Nous illustrons ci-dessous l'opération $ \prec $:
$$\begin{array}{|rclcl|rclcl|rclcl|}
\hline \tdun{1}\tdun{2} &\prec& \tdun{1} \tdun{2} &=& \tdun{1} \tdtroisun{2}{4}{3} &\tdun{1} \tddeux{3}{2} &\prec& \tdun{1} &=&\tdun{1} \tdtroisun{3}{4}{2} &
\tdun{1} &\prec& \tddeux{1}{2} &=& \tdtroisdeux{1}{3}{2} \\
\tdun{1} \tdun{2} &\prec& \tdtroisun{1}{3}{2} &=&\tdun{1} \tdquatrequatre{2}{5}{4}{3}& \tddeux{1}{2} &\prec& \tddeux{1}{2} &=&\tdquatretrois{1}{4}{3}{2}&
\tdun{1}&\prec& \tdun{1} \tdun{2} \tdun{3} &=&\tdquatreun{1}{4}{3}{2}\\
\tddeux{1}{2} &\prec& \tdun{1} \tddeux{2}{3}&=& \tdcinqquatre{1}{5}{4}{3}{2} & \tddeux{1}{2} &\prec& \tdtroisun{2}{3}{1} &=&\tdcinqsept{1}{5}{4}{3}{2}&
\tddeux{1}{2} &\prec& \tdtroisdeux{2}{3}{1} &=&\tdcinqneuf{1}{5}{4}{3}{2}\\
\hline \end{array}$$

Vérifions que $ (\mathcal{B})_{+} $ est bien stable par $ \prec $. Il suffit de montrer que si $ T $ et $ G $ sont deux arbres non vides appartenant à $ \mathcal{T} $, l'arbre $ T \prec G $ définit précédemment est encore un élément de $ \mathcal{T} $. Si $ G = \tdun{1} $, $ T \prec G = B^{+}(T) $ appartient à $ \mathcal{T} $. Supposons maintenant que $ \left| G \right| \geq 2 $. On note $ G_{1}, \hdots ,G_{n} $ la suite des sous-arbres issus de la racine de $ G $, et $ G_{i,1}, \hdots ,G_{i,m_{i}} $ la suite des sous-arbres issus de la racine de $ G_{i} $ pour tout $ i \in \left\lbrace 1, \hdots, n \right\rbrace $ (si $ G_{i} = \tdun{1} $, alors $ m_{i}=1 $ et $ G_{i,m_{i}} = 1 $). Il y a alors plusieurs cas:
\begin{enumerate}
\item Si $ G = B^{+}( \hdots B^{+}(\tdun{1}, G_{1,1} , \hdots , G_{1,m_{1}}), \hdots ), G_{n,1} , \hdots, G_{n,m_{n}}) $, où $ G_{1,1}, \hdots , G_{n,m_{n}} \in \mathcal{T} $ par construction. Alors, pour tout $ i \in \left\lbrace 1, \hdots, n \right\rbrace $, $ G_{i} = B^{-}(G_{i,1}, \hdots ,G_{i,m_{i}}) \in \mathcal{T} $, et donc
\begin{eqnarray*}
T \prec G = B^{+} (T, G_{1}, \hdots, G_{n})
\end{eqnarray*}
est un élément de $ (\mathcal{B})_{+} $.
\item Si il existe un $ 1 \leq i \leq n $ tel que
$$ G = \overbrace{B^{+}(\hdots B^{+}}^{n-i ~ {\rm fois}}(B^{-}(G_{1} , \hdots , G_{i}), G_{i+1,1} , \hdots , G_{i+1,m_{i+1}}), \hdots ), G_{n,1} , \hdots, G_{n,m_{n}}) ,$$
avec par construction $ G_{1}, \hdots , G_{i}, G_{i+1,1}, \hdots, G_{n,m_{n}} \in \mathcal{T} $. Alors, pour tout $ i+1 \leq j \leq n $, $ G_{j} = B^{-}(G_{j,1}, \hdots ,G_{j,m_{j}}) \in \mathcal{T} $, et donc
\begin{eqnarray*}
T \prec G = B^{+} (T, G_{1}, \hdots, G_{i}, G_{i+1}, \hdots, G_{n})
\end{eqnarray*}
est un élément de $ (\mathcal{B})_{+} $.
\end{enumerate}
Ainsi, dans tout les cas, $ (\mathcal{B})_{+} $ est stable par l'opération $ \prec $.\\

{\bf Remarque.} {Soient $ n \geq 1 $ et $ T_{1} ,\hdots ,T_{n} $ $ n $ arbres non vides $ \in (\mathcal{B})_{+} $. Alors
$$ T_{1} \prec B^{-}(T_{2}, \hdots, T_{n}) = B^{+}(T_{1}, \hdots, T_{n}) .$$}

Le lemme suivant est évident :

\begin{lemma} \label{droite}
\'Etant donné $ F,G,H \in (\mathcal{B})_{+} $,
\begin{eqnarray}
F \prec (GH) & = & (F \prec G) \prec H, \label{dipt3}\\
F (G \prec H) & = & (FG) \prec H \label{dipt4}.
\end{eqnarray}
\end{lemma}

{\bf Remarque.} {Comme pour $ \succ $, $ \prec $ n'est pas associative. Par exemple,
$$ \begin{array}{rclcl}
\tdun{1} \prec (\tdun{1} \prec \tdun{1}) & = & \tdun{1} \prec \tddeux{1}{2} & = & \tdtroisdeux{1}{3}{2},\\
(\tdun{1} \prec \tdun{1}) \prec \tdun{1} & = & \tddeux{1}{2} \prec \tdun{1} & = & \tdtroisun{1}{3}{2}.
\end{array} $$
}

\begin{defi}
Une algèbre de greffes à droite est un espace vectoriel $ A $ muni de deux opérations $ \ast $ et $ \prec $ satisfaisant les deux relations suivantes : $ \forall ~ x,y,z \in A $
\begin{eqnarray*}
(x \ast y) \ast z & = & x \ast (y \ast z),\\
x \prec (y \ast z) & = & (x \prec y) \prec z,\\
x \ast (y \prec z) & = & (x \ast y) \prec z.
\end{eqnarray*}
Une algèbre de greffes à droite unitaire $ A $ est un espace vectoriel $ A = \mathbb{K}1 \oplus \overline{A} $ tel que $ \overline{A} $ est une algèbre de greffes à droite et où on a étendu les deux opérations $ \ast $ et $ \prec $ comme suit:
\begin{equation}\label{tens3}\begin{array}{rcl}
1 \ast a = a \ast 1 = a, {\it ~ pour ~ tout ~} a \in A, \\
a \prec 1 = a, ~ 1 \prec a = 0, {\it ~ pour ~ tout ~} a \in \overline{A}.
\end{array}\end{equation}
\end{defi}

Remarquons que $ 1 \prec 1 $ n'est pas défini. Si $ A = \mathbb{K}1 \oplus \overline{A} $ et $ B = \mathbb{K}1 \oplus \overline{B} $ sont deux algèbres de greffes à droite unitaires, on peut étendre, de la même façon que pour les algèbres de greffes à gauche, les deux opérations $ \ast $ et $ \prec $ au produit tensoriel $ A \otimes B $ en posant:
\begin{equation}\label{tens4}\left\{\begin{array}{rcl}
(a \otimes b) \ast (a' \otimes b') & := & (a \ast a') \otimes (b \ast b'),\\
(a \otimes b) \prec (a' \otimes b') & := & (a \ast a') \otimes (b \prec b'), {\rm ~ si ~} b \otimes b' \neq 1 \otimes 1,\\
(a \otimes 1) \prec (a' \otimes 1) & := & (a \prec a') \otimes 1,
\end{array}\right. \end{equation}
pour $ a,a' \in A $, $ b,b' \in B $.\\

Ainsi, grâce au lemme \ref{droite}, $ \mathcal{B} $ muni de la concaténation et de $ \prec $ est une algèbre de greffes à droite unitaire, où $ \overline{\mathcal{B}} = (\mathcal{B})_{+} $ et $ 1 $ est l'arbre vide de tel sorte que $ \mathcal{B} = \mathbb{K}1 \oplus \overline{\mathcal{B}} $.\\

Notons $ \overline{\mathcal{B}_{r}} $ l'algèbre de greffes à droite engendrée par l'élément $ \tdun{1} $, et $ \mathcal{B}_{r} $ l'algèbre de greffes à droite unitaire égale à $ \mathbb{K} 1 \oplus \overline{\mathcal{B}_{r}} $ de sorte que $ \overline{\mathcal{B}_{r}} = (\mathcal{B}_{r})_{+} $. Alors :

\begin{theo} $ (\mathcal{B}_{r})_{+} $ est libre.
\end{theo}

\begin{proof}
Il faut montrer que si $ T_{1} \hdots T_{m} $ est une forêt non vide appartenant à $ (\mathcal{B}_{r})_{+} $ (avec $ m \geq 2 $) alors $ T_{1}, \hdots, T_{m} \in (\mathcal{B}_{r})_{+} $. Raisonnons par récurrence sur le degré $ n = \left| T_{1} \right| + \hdots + \left| T_{m} \right| $ de la forêt $ T_{1} \hdots T_{m} $, le cas $ n = 1 $ étant trivial. Supposons $ T_{1} \hdots T_{m} \in (\mathcal{B}_{r})_{+} $ de degré $ n \geq 2 $. Par construction de $ \mathcal{B}_{r} $, il existe $ F_{1}, F_{2} \in (\mathcal{B}_{r})_{+} $ tel que $ T_{1} \hdots T_{m} = F_{1} F_{2} $ ou $ T_{1} \hdots T_{m} = F_{1} \prec F_{2} $.
\begin{enumerate}
\item Si $ T_{1} \hdots T_{m} = F_{1} F_{2} $, $ F_{1} $ et $ F_{2} $ étant non vides, il existe $ i \in \{1,\hdots,m-1\} $ tel que $ F_{1} = T_{1} \hdots T_{i} $ et $ F_{2} = T_{i+1} \hdots T_{m} $. Par hypothèse de récurrence, comme $ \left| F_{1} \right| < n $ et $ \left| F_{2} \right| < n $, $ T_{1}, \hdots, T_{i} \in (\mathcal{B}_{r})_{+} $ et $ T_{i+1}, \hdots ,T_{m} \in (\mathcal{B}_{r})_{+} $ ce qui démontre le résultat dans ce cas.
\item Si $ T_{1} \hdots T_{m} = F_{1} \prec F_{2} $, en notant $ G_{1} \hdots G_{k} $ la forêt $ F_{1} $, on a l'égalité $ T_{1} \hdots T_{m-1} T_{m} = G_{1} \hdots G_{k-1} (G_{k} \prec F_{2}) $. Nécessairement, $ k = m $, $ \forall i \in \{1,\hdots,m-1\} $, $ T_{i} = G_{i} $ et $ T_{m} = G_{m} \prec F_{2} $. Par hypothèse de récurrence, comme $ \left| F_{1} \right| < n $, $ G_{1}, \hdots ,G_{m} \in (\mathcal{B}_{r})_{+} $. Donc, $ \forall i \in \{1,\hdots,m-1\} $, $ T_{i} = G_{i} \in (\mathcal{B}_{r})_{+} $. De plus, $ G_{m}, F_{2} \in (\mathcal{B}_{r})_{+} $, donc $ T_{m} = G_{m} \prec F_{2} \in (\mathcal{B}_{r})_{+} $.
\end{enumerate}
Dans tout les cas, $ T_{1}, \hdots, T_{m} \in (\mathcal{B}_{r})_{+} $. On conclut par le principe de récurrence.
\end{proof}

\vspace{0.5cm}

Voici par exemple les arbres appartenant à $ \mathcal{B}_{r} $ de degré $ \leq 4 $:
$$ 1, \tdun{1} , \tddeux{1}{2} , \tdtroisdeux{1}{3}{2}, \tdtroisun{1}{3}{2} , \tdquatreun{1}{4}{3}{2} , \tdquatredeux{1}{4}{3}{2} , \tdquatrequatre{1}{4}{3}{2} , \tdquatretrois{1}{4}{3}{2} , \tdquatrecinq{1}{4}{3}{2} .$$
Remarquons que $ \mathcal{B}_{r} $ n'est ni une algèbre de Hopf ni une sous-algèbre de $ \mathcal{B}^{\infty} $.\\

Il y a une relation pour le coproduit comparable à celle de la proposition \ref{deltasucc}:

\begin{prop} \label{deltaprec}
Posons $ \Delta_{\boldsymbol{r}} (F) = \displaystyle\sum_{\boldsymbol{v} \models V(F) {\rm ~ et ~} roo_{\boldsymbol{r}}(F) \in Roo_{\boldsymbol{v}} (F)} Lea_{\boldsymbol{v}} (F) \otimes Roo_{\boldsymbol{v}} (F) $ pour toute forêt non vide $ F $, où $ roo_{\boldsymbol{r}}(F) $ est la racine de l'arbre le plus à droite de la forêt $ F $. Alors, pour tout $ F \in (\mathcal{B})_{+} $ et $ G \in \mathcal{B} $,
\begin{equation}
\Delta_{\boldsymbol{r}} (F \prec G) = \Delta_{\boldsymbol{r}} (F) \prec \Delta(G).
\end{equation}
De plus: $ ( \Delta \otimes Id) \circ \Delta_{\boldsymbol{r}} (F) = ( Id \otimes \Delta_{\boldsymbol{r}}) \circ \Delta_{\boldsymbol{r}} (G). $
\end{prop}

\begin{proof}
Comme dans la preuve de la proposition \ref{deltasucc}, on utilise les notations suivantes: si $ a \in \mathcal{B} $, $ \Delta (a) = a \otimes 1 + 1 \otimes a + a' \otimes a'' $, et si $ a \in (\mathcal{B})_{+} $, $ \Delta_{\boldsymbol{r}} (a) = 1 \otimes a + a'_{\boldsymbol{r}} \otimes a''_{\boldsymbol{r}} $. Alors, pour tout $ x \in (\mathcal{B})_{+} $ et $ y \in \mathcal{B} $,
\begin{eqnarray*}
\Delta_{\boldsymbol{r}} (x) \prec \Delta(y) & = & (1 \otimes x + x'_{\boldsymbol{r}} \otimes x''_{\boldsymbol{r}}) \prec (y \otimes 1 + 1 \otimes y + y' \otimes y'') \\
& = & y \otimes (x \prec 1) + 1 \otimes (x \prec y) + y' \otimes (x \prec y'') + x'_{\boldsymbol{r}} y \otimes (x''_{\boldsymbol{r}} \prec 1)\\
& \hspace{0.5cm} & + x'_{\boldsymbol{r}} \otimes (x''_{\boldsymbol{r}} \prec y) + x'_{\boldsymbol{r}} y' \otimes (x''_{\boldsymbol{r}} \prec y'') \\
& = & y \otimes x + 1 \otimes (x \prec y) + y' \otimes (x \prec y'') + x'_{\boldsymbol{r}} y \otimes x''_{\boldsymbol{r}} + x'_{\boldsymbol{r}} \otimes (x''_{\boldsymbol{r}} \prec y)\\
& \hspace{0.5cm} & + x'_{\boldsymbol{r}} y' \otimes (x''_{\boldsymbol{r}} \prec y'') \\
& = & \Delta_{\boldsymbol{r}} (x \prec y),
\end{eqnarray*}
en utilisant (\ref{tens4}) à la deuxième égalité et (\ref{tens3}) à la troisième.\\

Pour la deuxième égalité, la démonstration est identique à celle de la proposition \ref{deltasucc}.
\end{proof}

\subsubsection{Algèbre bigreffe}

{\bf Remarque.} {Pour toute forêt $ F,G,H \in (\mathcal{B})_{+} $,
\begin{equation} \label{bigreffe}
(F \succ G) \prec H = F \succ (G \prec H).
\end{equation}
Cette relation est encore vraie si $ F, G $ ou $ H $ est égal à $ 1 $. Et elle n'a pas de sens si au moins deux des trois éléments $ F,G,H $ sont égaux à $ 1 $, car $ 1 \succ 1 $ et $ 1 \prec 1 $ ne sont pas définis.}\\

Ceci nous amène à introduire la notion d'algèbre bigreffe:

\begin{defi}
Une algèbre bigreffe est un espace vectoriel $ A $ muni de trois opérations $ \ast $, $ \succ $ et $ \prec $ telles que $ (A, \ast , \succ) $ est une algèbre de greffes à gauche, $ (A, \ast , \prec) $ est une algèbre de greffes à droite, et pour tout $ x,y,z \in A $,
\begin{eqnarray*}
(x \succ y) \prec z = x \succ (y \prec z).
\end{eqnarray*}
Une algèbre bigreffe unitaire $ A $ est un espace vectoriel $ A = \mathbb{K}1 \oplus \overline{A} $ tel que $ \overline{A} $ est une algèbre bigreffe, $ (A, \ast , \succ) $ est une algèbre de greffes à gauche unitaire d'unité $ 1 $ et $ (A, \ast , \prec) $ est une algèbre de greffes à droite unitaire d'unité $ 1 $.
\end{defi}

Ainsi, avec (\ref{bigreffe}), $ \mathcal{B} $ munit de la concaténation, de $ \succ $ et de $ \prec $ est une algèbre bigreffe unitaire, en prenant $ \overline{\mathcal{B}} = (\mathcal{B})_{+} $, et pour $ 1 $ l'arbre vide.

\begin{theo} \label{bigreffelibre}
$ (\mathcal{B})_{+} $ est engendrée comme algèbre bigreffe par l'élément $ \tdun{$1$} $.
\end{theo}

\begin{proof}
Il suffit de montrer que tout les arbres appartenant à $ \mathcal{T} $ peuvent être construit à partir de $ \tdun{$1$} $ avec les opérations $ m $, $ \succ $ et $ \prec $. Raisonnons par récurrence sur le degré, le résultat étant évidement vrai au rang $ 1 $. Soit $ T \in \mathcal{T} $ de degré $ n \geq 2 $. Il y a deux cas:
\begin{enumerate}
\item Si $ T = B^{-}(T_{1}, \hdots , T_{k}) $, avec $ T_{1}, \hdots , T_{k} \in \mathcal{T} $ de degré $ < n $ et $ k \geq 1 $. Par hypothèse de récurrence, $ T_{1}, \hdots , T_{k} $ peuvent être construits à partir de $ \tdun{$1$} $ avec les opérations $ m $, $ \succ $ et $ \prec $. Alors $ T = (T_{1} \hdots T_{k}) \succ \tdun{$1$} $ peut aussi être construit dans ce cas à partir de $ \tdun{$1$} $ avec $ m $, $ \succ $ et $ \prec $.
\item Si $ T = B^{+}(T_{1}, \hdots , T_{k}) $, avec $ T_{1}, \hdots , T_{k} \in \mathcal{T} $ de degré $ < n $ et $ k \geq 1 $. Toujours par récurrence, $ T_{1}, \hdots , T_{k} $ peuvent être construits à partir de $ \tdun{$1$} $ avec $ m $, $ \succ $ et $ \prec $. Si $ k = 1 $, $ T = B^{+}(T_{1}) = T_{1} \prec \tdun{$1$} $ et le résultat est démontré. Sinon
$$ T = T_{1} \prec B^{-}(T_{2}, \hdots , T_{k}) = T_{1} \prec ((T_{2} \hdots T_{k}) \succ \tdun{$1$}) $$
et $ T $ peut ici encore être construit à partir de $ \tdun{$1$} $ avec $ m $, $ \succ $ et $ \prec $.
\end{enumerate}
Le principe de récurrence permet de conclure.
\end{proof}
\\

{\bf Remarque.} {
\begin{enumerate}
\item $ (\mathcal{B})_{+} $ n'est pas librement engendrée comme algèbre bigreffe par l'élément $ \tdun{$1$} $. Par exemple,
$$ \tdun{$1$} \prec (\tdun{$1$} \prec \tdun{$1$}) = \tdtroisdeux{1}{3}{2} = \tdun{$1$} \prec (\tdun{$1$} \succ \tdun{$1$}) .$$
\item Une description de l'algèbre bigreffe libre engendrée par un générateur et de l'opérade bigreffe associée seront faites dans un autre article.
\end{enumerate}}

\end{document}